\documentclass[a4paper,10pt]{amsart}
 \usepackage{
  amsmath,
  amssymb,
  enumerate,
  fullpage,
  mathrsfs,
  mathtools,
  multicol,
  thmtools,
  tikz-cd,
  xspace
}

\DeclareFontFamily{U}{wncy}{}
\DeclareFontShape{U}{wncy}{m}{n}{<->wncyr10}{}
\DeclareSymbolFont{mcy}{U}{wncy}{m}{n}
\DeclareMathSymbol{\Sha}{\mathord}{mcy}{"58}

\usepackage[pagebackref, hypertexnames=false, colorlinks]{hyperref}
\usepackage[capitalise]{cleveref}

\newenvironment{smatrix}{\left( \begin{smallmatrix} } {\end{smallmatrix} \right) }

\newcommand{\stbt}[4]{\begin{smatrix}#1 & #2 \\ #3 & #4\end{smatrix}}

\theoremstyle{plain}
\newtheorem{theorem}{Theorem}[subsection]
\newtheorem{lemma}[theorem]{Lemma}
\newtheorem{proposition}[theorem]{Proposition}
\newtheorem{corollary}[theorem]{Corollary}
\newtheorem{definition}[theorem]{Definition}
\newtheorem{conjecture}[theorem]{Conjecture}
\newtheorem{assumption}[theorem]{Assumption}

\newtheorem*{notation}{Notation}

\theoremstyle{remark}
\declaretheorem[name=Remark,sibling=theorem,qed={\lower-0.3ex\hbox{$\diamond$}}]{remark}
\declaretheorem[name=Note,sibling=theorem,qed={\lower-0.3ex\hbox{$\diamond$}}]{note}

\theoremstyle{plain}
\newtheorem{lettertheorem}{Theorem}

\newcommand{\qedthere}{%
 \par \vspace{-1.2\baselineskip}
 \hfill\qed
}

\DeclareMathOperator{\GL}{GL}
\DeclareMathOperator{\GSp}{GSp}
\DeclareMathOperator{\Gal}{Gal}
\DeclareMathOperator{\Gr}{Gr}
\DeclareMathOperator{\Hom}{Hom}
\DeclareMathOperator{\Ind}{Ind}
\DeclareMathOperator{\Nm}{Nm}

\DeclareMathOperator{\SL}{SL}
\DeclareMathOperator{\alg}{alg}
\DeclareMathOperator{\comp}{comp}
\DeclareMathOperator{\loc}{loc}
\DeclareMathOperator{\norm}{norm}
\DeclareMathOperator{\pr}{pr}
\DeclareMathOperator{\rank}{rank}
\DeclareMathOperator{\Char}{char}

\newcommand{\syn}{\mathrm{syn}}
\newcommand{\Iw}{\mathrm{Iw}}

\newcommand{\AKf}{\AA_{K, \mathrm{f}}}
\newcommand{\Af}{\AA_{\mathrm{f}}}
\newcommand{\CC}{\mathbf{C}}
\newcommand{\Dcris}{\mathbf{D}_{\mathrm{cris}}}
\newcommand{\DdR}{\mathbf{D}_{\mathrm{dR}}}

\newcommand{\Ht}{\widetilde{H}}
\newcommand{\Pif}{\Pi_{\mathrm{f}}}
\newcommand{\Qb}{\overline{\QQ}}
\newcommand{\QQ}{\mathbf{Q}}
\newcommand{\Qp}{\QQ_p}
\newcommand{\RG}{\operatorname{R\Gamma}}
\newcommand{\RGt}{\widetilde{\RG}}
\newcommand{\RR}{\mathbf{R}}
\newcommand{\ZZ}{\mathbf{Z}}
\newcommand{\Zp}{\ZZ_p}

\DeclareMathOperator{\rk}{rank}

\newcommand{\bz}{\mathbf{z}}
\newcommand{\bq}{\mathbf{q}}
\newcommand{\bt}{\mathbf{t}}
\newcommand{\bc}{\mathbf{c}}
\newcommand{\bka}{\pmb{\kappa}}
\newcommand{\TT}{\mathbf{T}}

\newcommand{\cE}{\mathcal{E}}
\newcommand{\cF}{\mathcal{F}}
\newcommand{\cH}{\mathcal{H}}
\newcommand{\cK}{\mathcal{K}}
\newcommand{\cL}{\mathcal{L}}
\newcommand{\cM}{\mathcal{M}}
\newcommand{\cO}{\mathcal{O}}
\newcommand{\cP}{\mathcal{P}}
\newcommand{\cR}{\mathcal{R}}
\newcommand{\cV}{\mathcal{V}}
\newcommand{\cW}{\mathcal{W}}
\newcommand{\cY}{\mathcal{Y}}
\newcommand{\cZ}{\mathcal{Z}}
\newcommand{\can}{\mathrm{can}}
\newcommand{\cl}{\mathrm{cl}}
\newcommand{\cyc}{\mathrm{cyc}}
\newcommand{\dR}{\mathrm{dR}}
\newcommand{\disc}{\operatorname{disc}(K/\QQ)}
\newcommand{\eps}{\epsilon}
\newcommand{\et}{\text{\textup{\'et}}}
\newcommand{\fQ}{\mathfrak{Q}}
\newcommand{\fa}{\mathfrak{a}}
\newcommand{\fb}{\mathfrak{b}}
\newcommand{\fn}{\mathfrak{n}}
\newcommand{\fp}{\mathfrak{p}}
\newcommand{\fq}{\mathfrak{q}}

\newcommand{\htimes}{\mathop{\hat\otimes}}
\newcommand{\id}{\mathrm{id}}
\newcommand{\into}{\hookrightarrow}
\newcommand{\onto}{\twoheadrightarrow}
\newcommand{\ord}{\mathrm{ord}}
\newcommand{\pif}{\pi_{\mathrm{f}}}
\newcommand{\rig}{\mathrm{rig}}
\newcommand{\uPi}{\underline{\Pi}}
\newcommand{\uk}{\underline{k}}
\newcommand{\upi}{\underline{\pi}}
\newcommand{\ut}{\underline{t}}

\DeclareMathOperator{\KS}{KS}

\numberwithin{equation}{section}

\renewcommand{\AA}{\mathbf{A}}

\renewcommand{\ge}{\geqslant}

\renewcommand{\le}{\leqslant}

\newcommand{\Qi}{\QQ_\infty}
\newcommand{\Qpi}{\QQ_{p, \infty}}

\newcommand{\Nek}{Nekov\'{a}\v{r}\xspace}

\newcommand{\LPR}{\mathscr{L}^{\mathrm{PR}}}

\author{David Loeffler}
\author{Sarah Livia Zerbes}

\title{Iwasawa theory for quadratic Hilbert modular forms}

\begin{document}

 \begin{abstract}
  We study the Iwasawa main conjecture for quadratic Hilbert modular forms over the $p$-cyclotomic tower. Using an Euler system in the cohomology of Siegel modular varieties, we prove  the ``Kato divisibility'' of the Iwasawa main conjecture under certain technical hypotheses. By comparing this result with the opposite divisibility due to Wan, we obtain the full Main Conjecture over the cyclotomic $\ZZ_p$-extension.

  As a consequence, we prove new cases of the Bloch--Kato conjecture for quadratic Hilbert modular forms, and of the equivariant Birch--Swinnerton-Dyer conjecture in analytic rank $0$ for elliptic curves over real quadratic fields twisted by Dirichlet characters.

  As a ``by-product'' of the theory developed here, we also present new results on Iwasawa theory for Rankin--Selberg convolutions of modular forms, relaxing hypotheses of $p$-distinction or $p$-regularity assumed in previous works. This gives new cases of the equivariant BSD conjecture for elliptic curves over $\QQ$ twisted by 2-dimensional odd Artin representations, giving finiteness of the $p$-part of the Tate--Shafarevich group for all but finitely many ordinary primes.
 \end{abstract}

   \maketitle

\setcounter{tocdepth}{1}
\tableofcontents

\section{Introduction}

 \subsection{The setting}

  Let $K$ be a real quadratic field, and let $\pi$ be a cuspidal automorphic representation of $\GL_2 / K$ of weight $(k_1, k_2, t_1, t_2)$, where $k_i, t_i$ are integers with $k_i \ge 2$ and $k_1 + 2t_1 = k_2 + 2t_2$. 
  Attached to $\pi$ we have an $L$-function $L(\pi, s)$, and a Galois representation $\rho_{\pi, v}: G_K \to \GL_2(E_v)$, for every prime $v$ of the coefficient field $E$ of $\pi$.

  \begin{conjecture}[Bloch--Kato conjecture]
   \label{conj:BK}
   If we have $L(\pi, 1+j) \ne 0$, for some integer $j$ with $t_i \le j \le k_i + t_i - 2$ for all $i$, then $H^1_{\mathrm{f}}(K, \rho_{\pi, v}^*(-j)) = 0$.
  \end{conjecture}

  To state our second conjecture, we shall suppose that $\pi$ is unramified and ordinary at the primes $\fp \mid p$ of $K$, so that the $p$-adic $L$-function $L_p(\pi)$ is well-defined; this is a $p$-adic analytic function on a rigid space $\cW$ which naturally contains $\ZZ$, and at integers $j$ with $t_0 \le j \le t_0 + k_0 - 2$, the value $L_p(\pi)(j)$ is equal to $L(\pi, 1 + j)$ multiplied by an explicit, non-zero factor (depending on a choice of complex periods).

  The ordinarity of $\pi$ also allows us to define a \Nek--Selmer complex $\RGt_{\Iw}(K_\infty, \rho_{\pi, v}^*)$, where $K_\infty = K(\mu_{p^\infty})$ (see \cref{sect:selmer} below). This is a perfect complex of $\Lambda_{\cO}(\Gamma)$-modules, where $\Gamma = \Gal(K_\infty / K) \cong \Zp^\times$ and $\cO$ is a finite integral extension of $\Zp$. We can interpret $L_p(\pi)$ as an element of $\Lambda_{\cO}(\Gamma)$.

  \begin{conjecture}[Iwasawa main conjecture]
   \label{conj:Iw}
   The cohomology groups of $\RGt_{\Iw}(K_\infty, \rho_{\pi, v}^*)$ are zero except in degree 2, and the characteristic ideal of $\Ht^2_{\Iw}(K_\infty, \rho_{\pi, v}^*)$ as a $\Lambda_{\cO}(\Gamma)$-module is generated by $L_p(\pi)$.
  \end{conjecture}

  If $\pi$ corresponds to an elliptic curve $A$, then $\Ht^2_{\Iw}(K_\infty, \rho_{\pi, v}^*)$ is pseudo-isomorphic to the Pontryagin dual of $\operatorname{Sel}_{p^\infty}(A / K_\infty)$, so we recover the classical formulation of the Iwasawa main conjecture.

 \subsection{Our results}

  Our first main theorem proves one inclusion (the ``Kato divisibility'') in Conjecture \ref{conj:Iw}, for $\pi$ and $p$ satisfying certain hypotheses.

  \begin{lettertheorem}[{\cref{thm:withzeta}}]
   Assume that:
   \begin{enumerate}[(i)]
    \item $p$ is split in $K$.
    \item $\pi$ is unramified and ordinary at the primes above $p$.
    \item The central character of $\pi$ factors through $\Nm_{K/\QQ}$.
    \item $\pi$ is not a twist of a base-change from $\GL_2 / \QQ$.
    \item The Galois representation of $\pi$ satisfies the large-image condition (BI) (see \cref{sect:bigimage}).
   \end{enumerate}
   Then the characteristic ideal of $\Ht^2_{\Iw}(K_\infty, \rho_{\pi, v}^*)$ divides $L_p(\pi)$ in $\Lambda_{\cO}(\Gamma)$.

   In particular, if $L_p(\pi)$ is not a zero-divisor, then $\Ht^2_{\Iw}(K_\infty, \rho_{\pi, v}^*)$ is torsion, and $\Ht^i_{\Iw}(K_\infty, \rho_{\pi, v}^*)$ vanishes for $i \ne 2$.
  \end{lettertheorem}

  Comparing our result with the opposite bounds established by Wan \cite{wan15b}, we obtain the full Iwasawa Main Conjecture over the cyclotomic $\Zp$-extension, under somewhat stronger hypotheses. Let $K_\infty^1$ be the unique $\Zp$-extension of $K$ contained in $K_\infty$, and $\Gamma^1 \cong \Zp$ its Galois group, so we can write $\Lambda_{\cO}(\Gamma^1) = e_0 \Lambda_{\cO}(\Gamma)$ for an idempotent $e_0$.

  \begin{lettertheorem}[{\cref{thm:mainconj}}]
   Suppose the hypotheses of the previous theorem are satisfied, and also that
   \begin{itemize}
   \item $k_1 = k_2 = k$ for some even $k \ge 2$, and we take $t_1 = t_2 = 1 - \frac{k}{2}$;
   \item the central character of $\varepsilon_{\pi}$ is trivial;
   \item $\rho_{\pi, v}$ is minimally ramified, i.e.~the residual representation $\bar{\rho}_{\pi, v}$ has Artin conductor $\fn$;
   \item either $k > 2$, or the sign in the functional equation of $L(\pi, s)$ is $+1$.
   \end{itemize}
   Then we have
   \[\Char_{\Lambda_{\cO}(\Gamma^1)} \Ht^2_{\Iw}(K^1_\infty, \rho_{\pi, v}^*) = \left(e_0 \cdot L_p(\pi) \right) \]
   as ideals of $\Lambda_{\cO}(\Gamma^1)$.
  \end{lettertheorem}

  As consequences of Theorem A, we obtain cases of the Bloch--Kato conjecture \ref{conj:BK}.
  \begin{lettertheorem}[{\cref{thm:BK}}]
   Assume that the conditions of Theorem A are satisfied. Then, for every Dirichlet character $\chi$ of $p$-power conductor, and every $j$ in the critical range, we have
   \[
    L(\pi \otimes \chi, 1 + j) \ne 0\ \Longrightarrow\
    H^1_{\mathrm{f}}(K, \rho^*_{\pi, v}(-j-\chi)) = 0.
   \]
  \end{lettertheorem}

  Note that this is already known if $\pi$ has parallel weight and trivial character, $s = 1+j$ is the central critical value, and $\chi = 1$ (and some other auxiliary conditions), using the anticyclotomic Euler system of Heegner cycles \cite{nekovar12, wang19}.

  In particular, we may take $\pi$ to correspond to an elliptic curve $A / K$ (via the modularity theorem of \cite{freitaslehungsiksek15}). If we assume $A$ has no complex multiplication, and $A$ is not a $\QQ$-curve (i.e.~not isogenous over $\Qb$ to its Galois conjugate $A^\sigma$), then our hypotheses are automatically satisfied for all primes split in $K$ outside a set of density 0; so we obtain the following:

  \begin{lettertheorem}[{\cref{thm:BSD}}]
   The (weak) equivariant Birch--Swinnerton-Dyer conjecture holds in analytic rank 0 for twists of $A$ by Dirichlet characters. That is, for any $\eta: (\ZZ / N\ZZ)^\times \to \Qb^\times$, we have the implication
   \[ L(A, \eta, 1) \ne 0\quad \Longrightarrow\quad A(K(\mu_N))^{(\eta)} \text{\ is finite}.\]
   Moreover, if $L(A, \eta, 1) \ne 0$ then, for a set of primes $p$ of relative density $ \ge \tfrac{1}{2}$, the group $\Sha_{p^\infty}(A / K(\zeta_N))^{(\eta)}$ is trivial.
  \end{lettertheorem}

  We also obtain results on the non-vanishing of $L$-functions twisted by $p$-power Dirichlet characters, see \cref{sect:nonvanish}.\medskip


  \noindent\emph{Remark}: We note that the results of this paper rely crucially on those of \cite{LZ20}. Until recently, the results of that reference were conditional on an assertion regarding Hecke eigenspaces in rigid cohomology whose proof was to appear elsewhere. However, we have recently revised the preprint \cite{LZ20} to remove this dependence on forthcoming work; so the results of that paper, and hence of this one, are now unconditional.

 \subsection{Outline of the paper}

  The overall strategy we will follow is to construct an Euler system for the 4-dimensional Galois representation $\Ind_K^{\QQ}(\rho_{\pi, v}^*)$. For $\pi$ with sufficiently regular weight, we construct this Euler system as an application of the results of \cite{LZ20}, using the ``twisted Yoshida lift'' -- an instance of Langlands functoriality, transferring automorphic representations from $\GL_2 / K$ to $\GSp_4 / \QQ$. The main technical obstacle to be overcome is that the most interesting Hilbert modular forms -- those of parallel weight -- transfer to \emph{non-cohomological} automorphic representations of $\GSp_4$, so the results of \emph{op.cit.} do not immediately apply. We approach these non-cohomological points using deformation along eigenvarieties, which requires an extremely delicate study of the ``leading terms'' of families of Euler systems at a given specialisation.

  (This ``leading term argument'' also gives new results in other settings. As an example, in Appendix \ref{sect:appendix} we apply it to the Euler system of Beilinson--Flach elements, in order to bound Selmer groups associated to Artin twists of elliptic curves over $\QQ$. If the eigenvalues of Frobenius at $p$ on the Artin representation are distinct, this case was treated in \cite{KLZ17}; but this $p$-regularity assumption fails for a positive-density set of primes. The new methods allow us to handle these $p$-regular cases, and to show that the Tate--Shafarevich group is finite for a density 1 set of primes.)

  Before embarking on the proofs of the main theorems of this paper, we will need to recall a great deal of machinery. For the convenience of the reader, we have attempted to split this into several independent blocks, so that sections \ref{sect:prelimhilb}-\ref{sect:gl2keigen} (dealing with Hilbert modular forms) are almost entirely independent of sections \ref{sect:gsp4prelim}--\ref{sect:gsp4eigen} (dealing with Siegel modular forms).

  We shall begin to draw the threads together in \cref{sect:yoshida}, in which we construct an Euler system for automorphic representations $\pi$ whose weights are ``very regular''. In the following section \ref{sect:ESYfam}, we remove this very regular condition using variation along an eigenvariety, and at the same time prove an explicit reciprocity law relating the Euler system to $p$-adic $L$-functions. The analysis of leading terms is carried out in sections \ref{sect:core} and \ref{sect:axiom}, leading up to the construction of the Euler system in full generality in \cref{sect:constructES}, and the proof of Theorem A in \cref{sect:mainthms}.

  In \S \ref{sect:wan}, we recall a selection of results from \cite{wan15b} and compare these with Theorem A, leading to the proof of Theorem B.

%
%

  \subsection*{Acknowledgements} We would like to thank John Coates for his interest and encouragement, and his comments on a preliminary draft of this paper; Xin Wan for his patient explanations regarding the results of \cite{wan15b}; and Chris Williams for answering our questions on $p$-adic Langlands functoriality. \footnote{This research was supported by the following grants: Royal Society University Research Fellowship ``$L$-functions and Iwasawa theory'' (Loeffler);  ERC Consolidator Grant ``Euler systems and the Birch--Swinnerton-Dyer conjecture'' (Zerbes).}

\section{Preliminaries I: Hilbert modular forms}
\label{sect:prelimhilb}

 In this section we collect together the notations and definitions we will need regarding Hilbert modular forms; no originality is claimed here (except for a few minor remarks in \cref{sect:bigimage}).

 \subsection{Weights}

   Throughout this paper, $K$ is a real quadratic field with ring of integers $\cO_K$, and $\sigma_1, \sigma_2: K \into \RR$ are the two real embeddings of $K$.

  \begin{definition}
   In this paper a \emph{weight} for $\GL_2/K$ will mean a 4-tuple $(k_1, k_2, t_1, t_2)$ of integers with $k_i \ge 2$ and $k_1 + 2t_1 = k_2 + 2t_2$; we write $w$ for the common value. We shall often abbreviate this to $(\uk, \ut)$, with $\uk = (k_1, k_2)$ and $\ut = (t_1, t_2)$.
  \end{definition}

  \begin{remark}
   It is a common convention to choose $w = \max(k_1, k_2)$, but this is inconvenient from the perspective of $p$-adic families, so for our present purposes it is simpler to allow $w$ to be arbitrary.
  \end{remark}

 \subsection{Automorphic representations and L-functions}

  By a \emph{Hilbert cuspidal automorphic representation} of weight $(\uk, \ut)$, we shall mean an irreducible $\GL_2(\AKf)$-subrepresentation of the space of holomorphic Hilbert modular forms of weight $(\uk, \ut)$, defined as in \cite[Definition 4.1.2]{leiloefflerzerbes18} for example. Thus $\pi$ is an essentially unitary representation, and its central character is of the form $\|\cdot\|^{2-w} \varepsilon_{\pi}$, for a finite-order character $\varepsilon_{\pi}$. We let $\fn$ be the level of $\pi$, which is the unique integral ideal such that the invariants of $\pi$ under the group
  \[ K_1(\fn) = \{ x \in \GL_2(\widehat{\cO}_K): x = \stbt \star \star {} {1} \bmod \fn\}\]
  are 1-dimensional.

  \begin{definition}
   For $\fq$ a prime of $K$, let $a_{\fq}(\pi)$ denote the eigenvalue of the double coset operator $\left[ K_1(\fn) \stbt{\varpi_{\fq}}{}{}1 K_1(\fn) \right]$ acting on $\pi^{K_1(\fn)}$. We denote this operator by $T_{\fq}$ if $\fq \nmid \fn$, and by $U_{\fq}$ otherwise.
  \end{definition}

  The $a_{\fq}(\pi)$ are in $\Qb$ (and are algebraic integers if $t_1, t_2 \ge 0$), and the subfield $E \subset \CC$ they generate is a finite extension of $\QQ$. For each prime $\fq$, we define local Euler factors $P_\fq(\pi, X) \in E[X]$ by
  \[
   P_\fq(\pi, X) =
   \begin{cases}
    1 - a_\fq(\pi)X + \varepsilon_{\pi}(\fq) \Nm(\fq)^{w-1} X^2 & \text{if $\fq \nmid \fn$},\\
    1 - a_\fq(\pi)X & \text{if $\fq \mid \fn$},
   \end{cases}
  \]
  so that
  \begin{align*}
   L(\pi, s) &= \prod_{\text{$\fq$ prime}} P_{\fq}(\pi,\Nm(\fq)^{-s})^{-1}.
  \end{align*}


  \begin{remark}
   We have followed \cite{bergdallhansen17} here in adopting a slightly unusual definition of the $L$-function, shifted by $\tfrac{1}{2}$ relative to the usual conventions in the analytic theory. (This corresponds to using the ``arithmetically normalised'' local Langlands correspondence, respecting fields of definition). Note that with our conventions, the functional equation of the $L$-series relates $L(\pi, s)$ to $L(\pi \otimes \varepsilon_{\pi}^{-1}, w - s)$; and if $A / K$ is an elliptic curve, then the main result of \cite{freitaslehungsiksek15} shows that we can find a $\pi$ of weight $(2, 2, 0, 0)$ such that $L(A/K, s) = L(\pi, s)$.
  \end{remark}

 \subsection{Periods and rationality of L-values}

  Let $\pi$ be a Hilbert cuspidal automorphic representation of level $\fn$ and
  weight $(\uk, \ut)$.

  \subsubsection{Betti cohomology}

   The weight $(\uk, \ut)$ determines a $\GL_2(\AKf)$-equivariant locally constant sheaf of $E$-vector spaces $\cV_{\uk,\ut}$ on the arithmetic symmetric space for $\GL_2 / K$ of level\footnote{We assume here that $K_1(\fn)$ is sufficiently small that this space is a smooth manifold. If this is not the case, one can alternatively define $M_B(\pi, E)$ as the $K_1(\fn) / V$-invariants in the cohomology at level $V$, where $V\triangleleft K_1(\fn)$ is some auxiliary subgroup which is sufficiently small in this sense.} $K_1(\fn)$. Let $M_B(\pi, E)$ denote the $\pif$-isotypical subspace of the compactly-supported Betti $H^2$ of this local system, and $M_B(\pi, \CC)$ its base-extension to $\CC$.

   The space $M_B(\pi, E)$ is 4-dimensional over $E$, and has an $E$-linear action of the component group of $\GL_2(K \otimes \RR)$, which is $(\ZZ/2\ZZ) \times (\ZZ/ 2\ZZ)$. For each sign $\eps \in \{\pm 1\}$, let  $M^{\eps}_B(\pi, E)$ denote the eigenspace where both factors of the component group act via $\eps$; these spaces are each 1-dimensional.

   \begin{remark}
    We could also define ``mixed'' eigenspaces $M^{(\eps_1, \eps_2)}_B(\pi, E)$, which would be relevant if we wanted to study the periods of $\pi$ twisted by general Hecke characters of $K$. However, using Yoshida lifts to $\GSp_4$ we can only see twists by characters factoring through the norm map, which always have the same sign at both infinite places; so the mixed-sign cohomology eigenspaces will play no role in our construction. (Of course, they can be accessed by replacing $\pi$ with a suitable quadratic twist.)
   \end{remark}

   \begin{remark}
    If $\pi[n] = \pi \otimes \|\cdot\|^n$, which is an automorphic representation of weight $(\uk, \ut - (n, n))$, then there is a canonical isomorphism $M_B(\pi[n], L) \cong M_B(\pi, L)$ but this twists the action of the component group by $(-1)^n$; so $M_B^{\eps}(\pi[n], L)$ is canonically isomorphic to $M_B^{(-1)^n\eps}(\pi, L)$.
   \end{remark}

  \subsubsection{Complex periods}

   If $\phi$ denotes the normalised holomorphic newform generating $\pi$, the image of $\phi$ under the Eichler--Shimura comparison map $\comp$ is canonically an element of $M^\eps_B(\pi, \CC)$, and for each $\eps$, the projection $\pr^{\eps}(\comp(\phi))$ spans $M^\eps_B(\pi, \CC)$.

   \begin{definition}
    If $v_\eps$ is a basis of $M^\eps_B(\pi, E)$, we denote by $\Omega_{\infty}^{\eps}(\pi, v_\eps) \in \CC^\times$ the complex scalar such that
    \[ \pr^{\eps}(\comp(\phi)) = \Omega_{\infty}^{\eps}(\pi, v_\eps) \cdot v_{\eps}\]
    as elements of $M^\eps_B(\pi, \CC)$.
   \end{definition}

  \subsubsection{Gauss sums}

   If $M \ge 1$ is an integer and $\chi: (\ZZ/ M\ZZ)^\times \to \CC^\times$ is a character of conductor exactly $M$, we write
   \begin{equation}
    \label{def:gaussum}
    G(\chi) = \sum_{a \in (\ZZ / M\ZZ)^\times} \chi(a) \exp(2\pi i a / M)
   \end{equation}
   for its Gauss sum. We can also consider $\chi$ as a character of $\AA_{\QQ}^\times / \QQ^\times$, by setting $\chi(\varpi_q) = \chi(q)$ for all primes $q \nmid M$, where $\varpi_q$ is a uniformiser at $q$; note that the restriction of this adelic $\chi$ to $\widehat{\ZZ}^\times$ is the \emph{inverse} of the original $\chi$.

   Using the inductivity properties of Gauss sums (the classical Hasse--Davenport relation), one can check that if $M$ is coprime to $D_K = |\disc|$, and $\chi_K$ denotes the character $\chi \circ \Nm$ of $\AA_K^\times$, then the Gauss sum of $\chi_K$, defined as in \cite[\S 4.3]{bergdallhansen17}, is given by
   \[ G(\chi_K) = G(\chi)^2 \cdot \chi(-D_K \bmod M) \cdot \varepsilon_K(M),
   \]
   where $\varepsilon_K$ is the quadratic character (of conductor $D_K$) associated to $K$.

  \subsubsection{The $L$-function as a linear map}

   \begin{definition}
    We say an integer $j \in \ZZ$ is \emph{Deligne-critical} for weight $(\uk, \ut)$ if we have
    \[ t_0  \le j \le t_0 + k_0 - 2, \]
    where $k_0 = \min(k_i)$ and $t_0 = \tfrac{w-k_0}{2} = \max(t_i)$.
   \end{definition}

   (This is the same definition used in \cite{bergdallhansen17} except for the fact that our modular forms have weight $(k_1, k_2)$ rather than weight $(k_1 + 2, k_2 + 2)$ as in \emph{op.cit}.). Then the critical values of the $L$-function $L(\pi, s)$ are $L(\pi, 1 + j)$ as $j$ varies over Deligne-critical integers.

   \begin{note}
    $j$ is Deligne-critical if and only if the Galois representation $\Ind_K^{\QQ} \rho_{\pi, v}^*(-j)$ (see below) has exactly two of its four Hodge--Tate weights $\ge 1$.
   \end{note}

   \begin{proposition}
   \label{prop:Lalgpi}
    Let $\chi$ be a Dirichlet character, and $j$ be a Deligne-critical integer. Let $\eps = (-1)^j \chi(-1)$.

    Then there exists a linear functional
    \[ \cL^{\mathrm{alg}}(\pi, \chi, j): E(\chi) \otimes_E M^{\eps}_B(\pi, E)^\vee, \]
    where $E(\chi)$ is the extension of $E$ generated by the values of $\chi$, such that after base-extending to $\CC$ we have
    \[
    \left\langle \cL^{\mathrm{alg}}(\pi, \chi, j), \pr^{\eps}(\comp(\phi))\right\rangle =
    \frac{(j - t_1)! (j-t_2)!L\left(\pi \otimes \chi_K, 1 + j\right) }{G(\chi)^2 (-2\pi i)^{2j+2-t_1-t_2}}.
    \]
   \end{proposition}

   \begin{proof}
    This is an elementary reformulation of standard facts on rationality of $L$-values.
   \end{proof}

   Note that if $v_\eps$ is a choice of basis of $M_B^{\eps}(\pi, E)$ as above, then we have
   \[ \langle \cL^{\mathrm{alg}}(\pi, \chi, j), v_\eps\rangle =
   \frac{(j - t_1)! (j-t_2)!L\left(\pi \otimes \chi_K, 1 + j\right) }{G(\chi)^2 (-2\pi i)^{2j+2-t_1-t_2} \Omega^{\eps}_\infty(\pi, v_\eps)}, \]
   with both sides lying in $E(\chi)$ and depending $\Gal(E(\chi)/E)$-equivariantly on $\chi$.

 \subsubsection{Non-vanishing of L-values}

  The following is a standard result:

  \begin{proposition}
   \label{prop:nonvanish}
   Let $j$ be a Deligne-critical integer. Then $\cL^{\alg}(\pi, \chi, j) \ne 0$ for all characters $\chi$, except possibly if $w$ is even and $j = \tfrac{w-2}{2}$.
  \end{proposition}

  \begin{proof}
   This is immediate from the convergence of the Euler product if $j \ne \tfrac{w\pm 1}{2}$, and in these boundary cases it follows from the main theorem of \cite{jacquetshalika76} (the ``prime number theorem'' for automorphic $L$-functions).
  \end{proof}

\subsection{Galois representations}

  Let $\pi$ be as in the previous section, and $v$ a prime of $E$ above $p$.

  \begin{theorem}
   There exists an irreducible $2$-dimensional Galois representation
   \[ \rho_{\pi,v}: \Gal\left(\overline{K}/K\right) \to \GL_2(E_v),\]
   unique up to isomorphism, such that for all primes $\fq\nmid \mathfrak{n}\cdot p$ the representation is unramified at $\fq$, and we have
   \[ \det{}_{E_v}\left(1-X\rho_{\pi,v}(\mathrm{Frob}^{-1}_\fq)\right)=P_\fq(\pi, X).\]
   The representation $\rho_{\pi,v}$ is de Rham at the primes $\fp \mid p$, and its Hodge numbers\footnote{\label{fn:hodge}Negatives of Hodge--Tate weights, so the Hodge number of the cyclotomic character is $-1$.} are $(t_i, t_i + k_i - 1)$ at the $i$-th embedding $K \into \overline{E}_v$. are $(t_1,t_1 + k_1-1)$ at $\sigma_1$, and $(t_2, t_2+k_2-1)$ at $\sigma_2$. If $\fp \mid p$ but $\fp \nmid \fn$, then $\rho_{\pi,v}$ is crystalline at $\fq$ and we have
   \[
   \det{}_{E_v}\left(1-X \varphi^{[K_\fp: \Qp]}\, \middle|\, \Dcris(K_{\fp}, \rho_{\pi,v})\right)=P_\fp(\pi, X).\qedhere
   \]
   \qedthere
  \end{theorem}

 \subsection{Big Galois image}
 \label{sect:bigimage}

  In this section we shall recall, and slightly extend, some results from \cite[\S 9.4]{leiloefflerzerbes18} on the images of the representations $\rho_{\pi,v}$.

  \begin{notation}
   Let $\sigma: K \to K$ be the nontrivial Galois automorphism, and let $\pi^{\sigma}$ be the image of $\pi$ under $\sigma$ (so that $a_{\fq}(\pi^{\sigma}) = a_{\sigma(\fq)}(\pi)$).
  \end{notation}

  Note that $\pi^{\sigma}$ has the same coefficient field as $\pi$, so the direct-product Galois representation
  \[ \rho_{\pi, v} \times \rho_{\pi^\sigma, v}: G_K \to \GL_2(E_v) \times \GL_2(E_v)\]
  is well-defined.

  \begin{definition}
  \label{def:BI}
   Let us say that \textbf{condition (BI)} (for ``big image'') is satisfied for $\pi$ and $v$ if:
   \begin{itemize}
   \item $p \ge \max(7, 2k_1, 2k_2)$ and $p$ is unramified in $K$.
   \item $p$ is coprime to $\#(\cO_K / \fn)$ (where $\fn$ is the level of $\pi$).
   \item $(\rho_{\pi, v} \times \rho_{\pi^\sigma, v})(G_K)$ contains a conjugate of $\SL_2(\Zp) \times \SL_2(\Zp)$.
   \end{itemize}
  \end{definition}

  \begin{proposition}
   Suppose $\pi$ satisfies condition (BI). Then:
   \begin{enumerate}[(a)]
    \item The residual representation $\bar{\rho}_{\pi, v}$ is irreducible, and remains irreducible restricted to $\Gal(\overline{K} / K^{\mathrm{ab}})$.

    \item The tensor product $\bar{\rho}_{\pi, v} \otimes \bar{\rho}_{\pi^{\sigma}, v}$ is irreducible as a representation of $\Gal(\overline{K} / K^{\mathrm{ab}})$.

    \item The induced representation $\Ind_K^{\QQ}(\bar{\rho}_{\pi, v})$ is irreducible as a representation of $\Gal\left(\Qb/\QQ(\mu_{p^\infty})\right)$.

    \item There exists $\tau\in\Gal\left(\Qb/\QQ^{\mathrm{ab}}\right)$ such that $(\Ind_K^{\QQ} T) / (\tau - 1)$ is free of rank 1 over $\cO_v$, for some (and hence any) $\cO_v$-lattice $T$ in $\rho_{\pi, v}$.

    \item There exists $\gamma \in \Gal\left(\Qb/\QQ^{\mathrm{ab}}\right)$ which acts as $-1$ on $\Ind_K^{\QQ} \rho_{\pi, v}$.
   \end{enumerate}
  \end{proposition}

  \begin{proof}
   Parts (a) and (b) are obvious, since $\SL_2(\Zp)$ has no nontrivial abelian quotients, and hence the image of $G_{K^{\mathrm{ab}}}$ still contains $\SL_2(\Zp) \times \SL_2(\Zp)$. Similarly, this shows that $\Ind_K^{\QQ}(\bar{\rho}_{\pi, v})$ is a direct sum of two non-isomorphic irreducibles as a representation of $G_{K(\mu_{p^\infty})}$, and since these are not stable under $G_{\QQ(\mu_{p^\infty})}$, we obtain (c).

   For (d), we take $\tau$ to be any element whose image under $\rho_{\pi, v} \times \rho_{\pi^\sigma, v}$ has the form
   \[ \begin{pmatrix} 1 & 1 \\ & 1 \end{pmatrix} \times \begin{pmatrix} x & \\ & x^{-1}\end{pmatrix},\]
   for some $x \in \ZZ_p^\times$ such that $x\ne 1\pmod p$. This $\tau$ clearly has the right property. Similarly, the image will contain $(-\mathrm{Id}) \times (-\mathrm{Id})$ which gives the element $\gamma$.
  \end{proof}

  \begin{remark}
   Parts (c) and (d) of the proposition show that if condition (BI) holds, then Rubin's $\operatorname{Hyp}(\QQ(\mu_{p^\infty}), T(\chi))$ holds for every Dirichlet character $\chi$ and every Galois-invariant lattice $T$. Moreover, part (e) implies that
   \[ H^1(\QQ(T, \mu_{p^\infty}) / \QQ, T/\varpi_v) = H^1(\QQ(T, \mu_{p^\infty}) / \QQ, T^*(1) / \varpi_v) = 0\]
   (Hypothesis H.3 of \cite{mazurrubin04}); equivalently, the quantities $n(W)$ and $n^*(W)$ appearing in \cite{rubin00} are both zero, for $W = T \otimes \Qp/\Zp$.

   We shall need part (b) since it is required as a hypothesis for some results of Dimitrov \cite{dimitrov09} which we shall need to quote later. This is also the reason why we assumed $p \ge \max(2k_1, 2k_2)$. Otherwise, the tensor product plays no role in this paper.
  \end{remark}

  We now show that if we fix $\pi$ and let $v$ vary, then condition (BI) is satisfied for a large set of primes $v$. We assume that $\pi$ is not of CM type.

  \begin{definition}
   Let $F$ be the subfield of $E$ generated by the ratios
   \[ t_{\fq}(\pi) \coloneqq \frac{a_{\fq}(\pi)^2}{\Nm(\fq)^{w}\varepsilon_{\pi}(\fq)}\]
   for all primes $\fq \nmid \fn$ of $K$ (the field of definition of the adjoint $\operatorname{Ad}(\pi)$).
  \end{definition}

  A well-known theorem proved by Ribet for elliptic modular forms \cite{ribet85}, generalised to Hilbert modular forms by \Nek \cite[Theorem B.4.10(3)]{nekovar12}, shows that for all but finitely many primes $v$ of $E$, $\rho_{\pi, v}(G_{K^{\mathrm{ab}}})$ is a conjugate of $\SL_2(\cO_{F, w})$, where $w$ is the prime of $F$ below $v$.
%

  \begin{definition}
   An \emph{exceptional automorphism} of $\pi$ is a field automorphism $\gamma: F \to F$ such that $\gamma\left(t_{\fq}(\pi)\right)= t_{\sigma(\fq)}(\pi)$ for almost all primes $\fq$.
  \end{definition}

  Note that $\gamma$ is unique if it exists (by the definition of $F$) and $\gamma^2$ must be the identity.

  \begin{proposition} \
   \begin{enumerate}
    \item If $\tilde\gamma: E \into \CC$ is a field embedding, then $\tilde{\gamma}(\pi)$ is a twist of $\pi$ if and only if $\tilde\gamma|_F = \id$, and $\tilde\gamma(\pi)$ is a twist of $\pi^{\sigma}$ if and only if $\tilde\gamma|_F = \gamma$ is an exceptional automorphism of $\pi$.
    \item If $\gamma = \id_F$ is an exceptional automorphism of $\pi$, then $\pi$ is a twist of a base-change from $\GL_2 /\QQ$.
   \end{enumerate}
  \end{proposition}

  \begin{proof}
   The first theorem is a consequence of the strong multiplicity one theorem for $\SL_2$ \cite[Theorem 4.1.2]{ramakrishnan00}, which shows that the ratios $t_{\fq}(\pi)$ for all but finitely many $\fq$ characterise $\pi$ uniquely up to twisting by characters. The second assertion is a special case of a result of Lapid and Rogawski \cite{lapidrogawski98}.
  \end{proof}

  \begin{proposition}
   \label{prop:HMFlargeimage}
   If $\pi$ does not have an exceptional automorphism, then for all but finitely many primes $v$ of $E$, we have
   \[(\rho_{\pi, v} \times \rho_{\pi^\sigma, v}) (G_{K^{\mathrm{ab}}}) = \SL_2(\cO_{F, w}) \times \SL_2(\cO_{F, w}).\]
   If $\pi$ has an exceptional automorphism $\gamma$, then the above conclusion holds for all but finitely many primes $v$ satisfying the condition that $\gamma(w) \ne w$. For all but finitely many primes with $\gamma(w) = w$, the image of $G_{K^{\mathrm{ab}}}$ is conjugate to the group of matrices of the form $\{ (u, \gamma(u)): u \in \SL_2(\cO_{F, w})\}$.
  \end{proposition}

  See \cite[Proposition 9.4.3]{leiloefflerzerbes18}. In particular, this shows that if $\pi$ is not a twist of a base-change form, then Condition BI holds for a set of primes of density at least $\tfrac{1}{2}$; and if $E = \QQ$, then there is no exceptional automorphism, so the condition holds for all but finitely many $v$.
%
%
%

\section{Preliminaries II: Iwasawa theory for \texorpdfstring{$\GL_2 / K$}{GL2/K}}
\label{sect:gl2keigen}

 Throughout this section, we fix a prime $p$, and a finite extension $L /\Qp$ with ring of integers $\cO$. We shall suppose throughout that $p$ is unramified in $K / \QQ$. We shall write $\nu_v$ for the $v$-adic valuation on $L$ normalised by $\nu_v(p) = 1$.

 \subsection{Hecke parameters and ordinarity for \texorpdfstring{$\GL_2 / K$}{GL2/K}}

  Let $\pi$ be a Hilbert cuspidal automorphic representation of weight $(\uk, \ut)$ and level $\fn$, with $(p, \fn)= 1$, whose coefficient field $E$ embeds into $L$ (and we fix a choice of such an embedding).

  \begin{assumption}
   \label{ass:primenumbers}
   If $E$ contains the images of the embeddings $\sigma_i: K \into \RR$ (which is always the case if $k_1 \ne k_2$), and $p$ is split in $K$, we shall always assume that the primes above $p$ in $K$ are numbered as $\fp_1, \fp_2$ in such a way that $\sigma_i(\fp_i)$ lies below $v$.
  \end{assumption}

  \begin{proposition} \
   \begin{enumerate}[(i)]
    \item If $p\cO_K = \fp$ is an inert prime, then $\nu_v(a_\fp(\pi)) \ge t_1 + t_2$.
    \item If $p\cO_K = \fp_1 \fp_2$ is split in $K$, then we have
    \[ \nu_v(a_{\fp_1}(\pi)) \ge t_1, \quad \nu_v(a_{\fp_2}(\pi)) \ge t_2.\]
   \end{enumerate}
  \end{proposition}

  \begin{definition}
   For $\fp \mid p$ we shall define $a_{\fp}^\circ(\pi)$ to be $p^{-h} a_{\fp}(\pi) \in \cO$, where $h$ is the lower bound on the valuation given in the proposition.
  \end{definition}

  Note that this quantity depends on the choice of the prime $v$ of the coefficient field. (The case of $\fp$ ramified can be handled by replacing the power of $p$ with a power of a uniformiser of $K_{\fp}$, as in \cite{bergdallhansen17}, but we shall not need this here.)

  \begin{definition}
   We define the \emph{slope} of $\pi$ at $\fp$ to be $\nu_v\left(a_{\fp}^\circ(\pi)\right) \in \QQ_{\ge 0}$, and we say $\pi$ is \emph{ordinary at $\fp$} if this slope is 0.
  \end{definition}

  We define similarly a renormalised Hecke polynomial $P_\fp^{\circ}(\pi, X) = P_\fp(\pi, p^{-h} X)$, so that if $p$ is inert we have $P_\fp^{\circ}(\pi, X) = 1 - a_{\fp}^\circ(\pi) X + p^{(k_1 + k_2 - 2)}\varepsilon_{\pi}(\fp) X^2$, and if $p = \fp_1 \fp_2$ is split we have $P_{\fp_i}^{\circ}(\pi, X) = 1 - a_{\fp_i}^{\circ}(\pi) X + p^{(k_i - 1)}\varepsilon_{\pi}(\fp_i) X^2$ with $k_i$ the weight corresponding to $\fp_i$ as above.

  We define normalised Hecke parameters $\fa_{\fp}^\circ, \fb_{\fp}^\circ \in \bar{E}_v$ by
  \[ P_\fp^{\circ}(\pi, X) = (1 - \fa_{\fp}^\circ X)(1 - \fb_{\fp}^\circ X).\]
  We shall assume that $\nu_v(\fa_{\fp}^\circ) \le \nu_v(\fb_{\fp}^\circ)$. In particular, if $\pi$ is ordinary at $\fp$, then $\fa_{\fp}^\circ$ will be the unique unit root.

  \begin{proposition}
   \label{prop:noexceptionalzero}
   If $\pi$ is ordinary at $\fp \mid p$, then neither $\fa_{\fp}^{\circ}$ nor $\fb_{\fp}^{\circ}$ is the product of a root of unity and an integral power of $p$.
  \end{proposition}

  \begin{proof}
   Suppose $p$ is split (the argument for $p$ inert is similar). Then, since $\pi_{\fp_i}$ is tempered, we have $|\fa_{\fp_i}^{\circ}| = |\fb_{\fp_i}^{\circ}| = p^{(k_i-1)/2}$ for every embedding $\overline{E} \into \CC$. So if either $\fa_{\fp_i}^{\circ}$ or $\fb_{\fp_i}^{\circ}$ is a power of $p$ times a root of unity, then its $v$-adic valuation must be exactly $\tfrac{k_i-1}{2}$. Since $k_i \ge 2$, this contradicts the assumption that $\fa_{\fp_i}^{\circ}$ has valuation $0$ and $\fb_{\fp_i}^{\circ}$ has valuation $k_i - 1$.
  \end{proof}

 \subsection{Cyclotomic p-adic L-functions}
  \label{sect:gl2kpadicL}

  For $j \in \ZZ$ and $\chi$ a Dirichlet character of $p$-power conductor, we write $\chi + j$ for the character $x \mapsto x^j \chi(x)$ of $\Zp^\times$.

  \begin{notation}
   For $? = L$ or $\cO$, let $\Lambda_{?}(\Gamma)$ denote the Iwasawa algebra of $\Gamma = \Zp^\times$ with coefficients in $?$, and $\Lambda_{?}(\Gamma)^\eps$ the quotient in which $[-1] = \eps$.

   Let $\cW$ denote the rigid-analytic space $\left(\operatorname{Spf} \Lambda_{\cO}(\Gamma)\right)^{\rig}$ over $L$ parametrising characters of $\Zp^\times$, and $\cW^{\eps}$ the subspace corresponding to $\Lambda(\Gamma)^\eps$.
  \end{notation}

  Note that for each sign $\eps$, $\cW^{\eps}$ is a union of components of $\cW$, and each component is isomorphic to an open disc. We can identify $\Lambda_{L}(\Gamma)$ with the ring of bounded rigid-analytic functions on $\cW$, and similarly for $\cW^{\eps}$.

  \begin{theorem}
   Suppose that $\pi_{\fp}$ is ordinary for all $\fp \mid p$. Then for each $\eps \in \{\pm 1\}$ there exists a unique element
   \[ \cL_p^{\eps}(\pi) \in M_B^{\eps}(\pi, L)^\vee \otimes \Lambda_{L}(\Gamma)^{\eps}
     \]
   such that for all characters of $\Zp^\times$ of the form $j+\chi$, with $\chi$ of finite order and $j$ a Deligne-critical integer, we have
   \[ \cL^{\eps}_p(\pi)(j + \chi) = \left(\textstyle \prod_{\fp \mid p} \cE_{\fp}(\pi, \chi, j)\right) \cdot \cL^{\mathrm{alg}}(\pi, \chi^{-1}, j),\]
   where
   \[ \cE_{\fp}(\pi, \chi, j) =
   \begin{cases}
   \left(\frac{q^{(j+1)}}{\fa_{\fp}}\right)^c & \text{if $c \ge 1$}\\
   (1 - \frac{q^j}{\fa_{\fp}}) (1 - \frac{\fa_{\fp}}{q^{1+j}}) & \text{if $c = 0$}

   \end{cases}
   \]
   for $\chi$ of conductor $p^c$ and $q = \Nm(\fp)$.
  \end{theorem}

  \begin{proof}
   See \cite[Theorem 0.1]{dimitrov13}.
  \end{proof}

  We can regard $\cL^{\eps}_p(\pi)$ as a (bounded) $M^\eps_B(\pi, L)^\vee$-valued rigid-analytic function on $\cW^{\eps}$. We shall write $\cL_p(\pi)$, without superscript, to denote the sum $\cL^{+}_p(\pi) \oplus \cL^{-}_p(\pi)$, considered as a rigid-analytic function with values in $M_B^+(\pi, L)^\vee \oplus M_B^-(\pi, L)^\vee$.

  \subsubsection{Non-vanishing}

   As before, let $k_0 = \min(k_i)$ and $t_0 = \max(t_i)$, so the Deligne-critical range is $t_0 \le j \le t_0 + k_0 - 2$.

   \begin{proposition}
    \label{prop:nonvanish-padic}
    Suppose $k_0 \ge 3$. Then $\cL_p(\pi)$ does not vanish identically on any component of $\cW$, and we have $\cL_p(\pi)(t_0) \ne 0$.
   \end{proposition}

   \begin{proof}
    We know that $L(\pi \otimes \chi_K^{-1}, 1 + j)$ is non-vanishing for all Deligne-critical $j \ne \tfrac{w-2}{2}$ by \cref{prop:nonvanish}. On the other hand, $\cE_{\fp}(\pi, \chi, j) \ne 0$ for all $j$ and $\chi$, by \cref{prop:noexceptionalzero}. So if $k_0 \ge 3$, it follows that $\cL_p(\pi)$ is non-zero at $t_0$, and also at $t_0 + \chi$ for all characters $\chi$, which clearly hit every component of $\cW$.
   \end{proof}

   \begin{remark}
    When $k_0 = 2$ we cannot \emph{a priori} rule out the possibility that $\cL_p(\pi)$ vanishes identically on some components of $\cW$. If $\cL_p(\pi)$ is identically zero, then the conjectures we are trying to prove are vacuous anyway; but a more subtle issue arises if $\cL_p(\pi)$ is zero on $\cW^+$ but not on $\cW^-$, or vice versa, since our machinery for constructing Euler systems can only ``see'' the products of pairs of $L$-values of opposite signs.
   \end{remark}

  \subsubsection{Equivariant $p$-adic $L$-functions}

   \begin{definition}
    Let $\cR$ be the set of square-free integers $n$ whose prime factors are 1 modulo $p$, and for $m \in \cR$, let $\Delta_m$ be the largest quotient of $(\ZZ / n\ZZ)^\times$ of $p$-power order.
   \end{definition}

   For $n \in \cR$ coprime to $\fn\disc$, we can define an equivariant $p$-adic $L$-function
   \[
    \cL_p^{\eps}(\pi; \Delta_m) \in
    M_B^\eps(\pi,L)^\vee
    \otimes_L \Lambda_{L}(\Gamma)^{\eps} \otimes_{L} L[\Delta_m],
   \]
   such that for all primitive characters $\eta: \Delta_m \to \CC^\times$, we have
   \[
    \cL_p^{\eps}(\pi; \Delta_m)(\eta, j+\chi) =
    \frac{\prod_{\fp \mid p} \cE_{\fp}(\pi \otimes \eta^{-1}, \chi, j)}{\eta(p)^{2c}\chi(m)^2}\cdot \cL^{\mathrm{alg}}(\pi, \chi^{-1}\eta^{-1}, j).
   \]
   (The correction factors in the denominator arise from the ratio of Gauss sums $\frac{G(\chi^{-1}\eta^{-1})}{G(\chi^{-1})G(\eta^{-1})} = \frac{1}{\eta(p) \chi(n)}$.)

   \begin{remark}
    Note that this equivariant $L$-function $\cL_p^\eps(\pi; \Delta_m)$ is always a linear functional on $M_B^\eps(\pi, L)^\vee$, irrespective of the value of $m$. This will be very important in the arguments below.
   \end{remark}

  \subsubsection{Integrality}

   \label{sect:integralL}

   We assume in this section that condition (BI) is satisfied (Definition \ref{def:BI}). Using Betti cohomology with coefficients in $\cO$, we can construct a canonical $\cO$-lattices $M_B(\pi, \cO) \subset M_B(\pi, L)$. We denote the dual of this lattice by $M_B(\pi, \cO)^\vee$. As before, we write $M_B^{\eps}(\pi, \cO)$ etc for the sign eigenspaces.

   (Here we do not need to distinguish between Betti cohomology with and without compact support, since these become isomorphic after localisation at the maximal ideal corresponding to $\bar\rho_{\pi, v}$, by \cite[Theorem 2.3]{dimitrov09}.)

   \begin{remark}
    Note that precisely, $M_B(\pi, \cO)$ is the maximal \emph{submodule} of the integral Betti cohomology on which the Hecke operators act via the eigensystem associated to $\pi$, and $M_B(\pi, \cO)^\vee$ the maximal \emph{quotient} on which the dual Hecke operators act via $\pi$.
   \end{remark}

   \begin{proposition}
    We have $\cL(\pi, \Delta_m) \in \Lambda_{\cO}(\Gamma) \otimes \cO[\Delta_m] \otimes
     M_B(\pi, \cO)^\vee$.
   \end{proposition}

   \begin{proof}
    This can be extracted from the construction of \cite{dimitrov13}. Dimitrov actually constructs a much more general equivariant $p$-adic $L$-function, taking values in the integral Iwasawa algebra of the ray class group modulo $\Sigma p^\infty$, for $\Sigma$ any set of primes of $K$. When $\Sigma$ is the set of primes dividing $m$, this group maps naturally to $\Delta_m \times \Gamma$.

    Dimitrov's $p$-adic $L$-function in fact takes values in the dual of a slightly different cohomology space $M_B(\pi,\Sigma, \cO)$, defined as the eigenspace in Betti cohomology at level $K_1(m\fn) \cap K_0(m^2)$ on which the the Hecke operators away from $\Sigma$ act via the eigensystem of $\pi$, and the operators $U_{\fq}$ for $\fq \in \Sigma$ act as 0. However, there is a natural map
    \[ M_B(\pi, \cO) \to M_B(\pi,\Sigma, \cO)\]
    given by composing the natural pullback map with the product of the $\fq$-depletion operators $1 - V_{\fq} U_{\fq}$ for $\fq \in \Sigma$, and composing Dimitrov's $p$-adic $L$-function with this map gives the $\cL(\pi, \Delta_m)$ above.
   \end{proof}

   The following invariant will be of relevance later. As a $\cO$-algebra, $\Lambda_{\cO}(\Gamma)$ is the direct product of $(p-1)$ local $\cO$-algebras, corresponding to the powers of the Teichm\"uller character; for $u \in \ZZ / (p-1) \ZZ$, we let $e_u$ be the corresponding idempotent. (Clearly, $e_u$ factors through $\Lambda(\Gamma)^{(-1)^u}$).

   \begin{definition}
   \label{def:muinv}
    Let $\mu_{m, u}(\pi) \in \ZZ_{\ge 0}$ denote the largest integer $n$ such that $e_u \cdot \cL_p(\pi, \Delta_m)$ is divisible by $\varpi^n$ in $e_u \Lambda_{\cO}(\Gamma) \otimes \cO[\Delta_m] \otimes
     M_B(\pi, \cO)^\vee$ (or $\infty$ if $e_u \cdot \cL_p(\pi, \Delta_m) = 0$). We define
     \[ \mu_{\mathrm{min}, u}(\pi) = \inf_{m \in \cR} \mu_{m, u}(\pi).\]
   \end{definition}

   This is the Iwasawa $\mu$-invariant of the equivariant $p$-adic $L$-function (in the $u$-th component of weight space).

 \subsection{Selmer complexes}
  \label{sect:selmer}

  Let $V$ be a 2-dimensional $L$-vector space on which $G_K$ acts\footnote{Later we shall choose a specific realisation of $\rho_{\pi, v}$ using eigenvarieties for $\GSp_4$, but the results in this section are valid for any representative of the abstract isomorphism class $\rho_{\pi, v}$.} via $\rho_{\pi, v}$, and $T$ a $G_{K}$-stable $\cO$-lattice in $V$.

  \subsubsection{Local cohomology at $p$}

   We will need the following simple remark, analogous to \cite[\S 13.13]{kato04} in the case $K = \QQ$. Let $K_\infty = K(\mu_{p^\infty})$.

   \begin{proposition}
    \label{prop:nolocalinvts}
    Let $\fp \mid p$ be an unramified prime with $\fp \nmid \fn$, and $W$ a subquotient of $V |_{G_{K_\fp}}$. Then $H^0(K_{\infty, \fp}, W) = H^0(K_{\infty, \fp}, W^*) = 0$.
   \end{proposition}

   \begin{proof}
    It suffices to consider $W$ (the case of $W^*$ follows by replacing $\pi$ with $\pi \otimes \varepsilon_{\pi}^{-1}$). Since $V$ is crystalline, so is $W$; and as $K_{\fp}$ is unramified over $\Qp$, it follows that if $H^0(K_{\infty, \fp}, W) \ne 0$, then there is some $n$ such that $H^0(K_{\fp}, W(n)) \ne 0$.

    If $H^0(K_{\fp}, W(n)) = 0$, then $q^n$ must be an eigenvalue of $\Phi = \varphi^{[K_\fp: \Qp]}$ on $\Dcris(K_{\fp}, V)$. Since these eigenvalues all have complex absolute value $q^{(w-1)/2}$, it follows that $w$ is odd and $n = \tfrac{w-1}{2}$. However, since $k_i\ge 2$, the integer $\tfrac{w-1}{2}$ is not among the Hodge numbers of $V$, which is a contradiction.
   \end{proof}

   \begin{remark} Note that the above proposition holds whether or not $\pi$ is ordinary at $\fp$ here. However, we will only use it in the case when $\pi$ is ordinary, in which case it can also be deduced from Proposition \ref{prop:noexceptionalzero}. \end{remark}

   \begin{theorem}[Wiles]
    Suppose $p = \fp_1 \fp_2$ is split (and we number the primes as in \cref{ass:primenumbers}). If $\pi$ is ordinary at $\fp_i$, for some $i \in \{1, 2\}$, then there is a 1-dimensional subrepresentation
    \[ \cF_{\fp_i}^+ V \subset V |_{G_{K_{\fp_i}}}, \]
    isomorphic to the tensor product of $\chi_{\cyc}^{-t_i}$ and the unramified character mapping geometric Frobenius to $\fa_{\fp_i}^{\circ}$.\qed
   \end{theorem}

   We define $\cF_{\fp_i}^+ V^*$ to be the orthogonal complement of $\cF_{\fp_i}^+ V$ (or, equivalently, we can define it by identifiying $V^*$ with $V(\pi \otimes \varepsilon_\pi^{-1})(1-w)$). We define $\cF_{\fp_i}^+ T^*$ to be the intersection of $\cF_{\fp_i}^+ V^*$ with $T^*$.

  \subsubsection{\Nek--Selmer complexes}

   We now recall the formalism of \emph{Selmer complexes} associated to global Galois representations \cite{nekovar06}; see \cite[\S 11.2]{KLZ17} for a summary. A \emph{local condition} $\square_{\fq}$ for $T$ at $\fq$ is a morphism of complexes
   \( U_{\fq}^+ \xrightarrow{\iota_{\fq}^+} C^\bullet(K_{\fq}, T), \)
   where $C^\bullet(K_{\fq}, T)$ is the usual complex of continuous cochains computing local Galois cohomology. A \emph{Selmer structure} for $T$ denotes a collection\footnote{In \cite{KLZ17} and various other sources, Selmer structures are denoted by $\Delta$; but we are using $\Delta$ for something else, so we have settled on $\square$ instead.} $\square = (\square_{\fq})$ of local conditions for each prime $\fq$, such that for almost all $\fq$, $\square_{\fq}$ is equal to the unramified local condition $C^\bullet(K_{\fq}^{\mathrm{ur}} / K_{\fq}, T^{I_{\fq}})\to C^\bullet(K_{\fq}, T)$.

   These data determine a Selmer complex $\RGt(K, T; \square)$,
   in the derived category of perfect complexes of $\cO$-modules, supported in degrees $\{0, \dots, 3\}$. We write $\Ht^i(K, T; \square)$ for the cohomology groups of the Selmer complex.


  \subsubsection{Greenberg Selmer structures}

   We now define the Selmer structures which are relevant for this paper. We suppose that $p = \fp_1 \fp_2$ is split in $K$ and that $\pi$ is ordinary at both primes above $p$. We let $\TT^* = T^* \otimes \Lambda_{\cO}(\Gamma)(-\bq)$, where $\bq$ denotes the canonical character $\Gamma \to \Lambda(\Gamma)^\times$. As is well known, for $i = 1$ and any $S$ including $p$ and the bad primes for $T$, we have a canonical isomorphism of $\Lambda_{\cO}(\Gamma)$-modules
   \begin{equation}
    \label{eq:Iwcoh}
    H^1(\cO_{K, S}, \TT^*) \cong H^1_{\Iw}(K_\infty, T^*) \coloneqq \varprojlim_n H^1(K(\mu_{p^n}), T^*).
   \end{equation}

   \begin{definition}
    The Selmer structure $\square^{(p)}$ on $\TT^*$ is defined as follows:
    \begin{itemize}
     \item For $\fq \nmid p$, $\square^{(p)}_{\fq}$ is the unramified condition.
     \item For $\fq \mid p$, $\square^{(p)}_{\fq}$ is given by the map
     \[ \RG(K_\fq, \cF_{\fp}^+ \TT) \longrightarrow  \RG(K_\fq, \TT^*) \]
     induced by the inclusion $\cF_{\fp}^+ \TT^* \into \TT^*$.
    \end{itemize}
    The Selmer structure $\square^{(\fp_1)}$ on $\TT^*$ is defined as follows:
    \begin{itemize}
    \item For $\fq \nmid p$, or $\fq = \fp_1$, $\square^{(\fp_1)}_{\fq}  = \square^{(p)}_{\fq}$.
    \item $\square^{(\fp_1)}_{\fp_2}$ is the relaxed local condition.
    \end{itemize}
   \end{definition}

   In keeping with \eqref{eq:Iwcoh} we write $\RGt_{\Iw}(K_\infty, T^*; \square^{(p)})$ and $\RGt_{\Iw}(K_\infty, T^*; \square^{(\fp_1)})$ for the corresponding Selmer complexes.

   \begin{proposition}
    The cohomology of the Selmer complexes is as follows:
    \begin{enumerate}[(a)]
     \item We have $\Ht^0_{\Iw}(K_\infty, T^*; \square^{(\fp_1)}) =\Ht^0_{\Iw}(K_\infty, T^*; \square^{(p)}) = 0$.
     \item There are exact sequences
     \[
     \begin{tikzcd}
      0 \rar& \Ht^1_{\Iw}(K_\infty, T^*; \square^{(p)}) \rar \dar[hook]&
       H^1_{\Iw}(K_\infty, T^*) \rar \dar[equals]& \bigoplus_{i=1,2} H^1_{\Iw}(K_{\infty, \fp_i}, T^* / \cF_{\fp_i}^+)\dar[two heads],\\
      0 \rar& \Ht^1_{\Iw}(K_\infty, T^*; \square^{(\fp_1)}) \rar& H^1_{\Iw}(K_\infty, T^*) \rar& H^1_{\Iw}(K_{\infty, \fp_1}, T^* / \cF_{\fp_1}^+).
     \end{tikzcd}
     \]
     Thus we can regard $\Ht^1_{\Iw}(K_\infty, T^*; \square^{(p)})$ and $\Ht^1_{\Iw}(K_\infty, T^*; \square^{(\fp_1)})$ as $\Lambda$-submodules of $H^1_{\Iw}(K_\infty, T)$.
     \item The modules $\Ht^3_{\Iw}(K_\infty, T^*; \square^{(p)})$ and $\Ht^3_{\Iw}(K_\infty, T^*; \square^{(\fp_1)})$ are finite groups; if $T \otimes k_L$ is irreducible, then they are zero.
     \item For each $u \in \ZZ/(p-1)$, we have
     \begin{align*}
      \rk_{(e_u \Lambda)} e_u \Ht^1_{\Iw}(K_\infty, T^*; \square^{(p)}) - \rk_{(e_u \Lambda)} e_u \Ht^2_{\Iw}(K_\infty, T^*; \square^{(p)}) &= 0,\\
      \rk_{(e_u \Lambda)} e_u \Ht^1_{\Iw}(K_\infty, T^*; \square^{(\fp_1)}) - \rk_{(e_u \Lambda)} e_u \Ht^2_{\Iw}(K_\infty, T^*; \square^{(\fp_1)}) &= 1.
     \end{align*}
    \end{enumerate}
   \end{proposition}

   \begin{proof}
    This follows from standard exact sequences and local and global duality.
   \end{proof}

   We can now give a full statement of the Iwasawa main conjecture in our setting:

   \begin{conjecture}[Iwasawa main conjecture] \
    \label{conj:mainconj}
    \begin{itemize}
    \item $\Ht^2_{\Iw}(K_\infty, V^*; \square^{(p)})$ is a torsion $\Lambda_{L}(\Gamma)$-module, and we have
    \[ \operatorname{char}_{\Lambda_L(\Gamma)}\Ht^2_{\Iw}(K_\infty, V^*; \square^{(p)}) = \left(L_p(\pi)\right).\]

    \item If the basis vectors $v_+$, $v_-$ used to define $L_p(\pi)$ are $\cO$-bases of $M_B^{\pm}(\pi, \cO)$, then
    \[ \operatorname{char}_{\Lambda_{\cO}(\Gamma)}\Ht^2_{\Iw}(K_\infty, T^*; \square^{(p)}) = \left(L_p(\pi)\right). \]
    \end{itemize}
   \end{conjecture}

   \begin{remark}
    Note that the Selmer structures $\square^{(\fp_1)}$ and $\square^{(p)}$ are \emph{not} ``simple Selmer structures'' in the sense of \cite[Proposition 11.2.9]{KLZ17}, since the local-condition complexes at the primes above $p$ can have nontrivial $H^2$. So $\Ht^2_{\Iw}(K_\infty, T^*; \square^{(p)})$ cannot be precisely described via the traditional (non-derived) theory. However, the group
    \[ \mathfrak{X}(K_\infty, T^*) = \ker \Big(\Ht^2_{\Iw}(K_\infty, T^*; \square^{(p)}) \to \bigoplus_{i = 1, 2} H^2_{\Iw}(K_\infty, \cF^+_{\fp_i} T^*)\Big)\]
    does admit a classical description as the Pontryagin dual of a Selmer group for $T \otimes \mu_{p^\infty}$; and although the local terms $H^2_{\Iw}(K_\infty, \cF^+_{\fp_i} T^*)$ are not necessarily zero, by \cref{prop:nolocalinvts} and local Tate duality they are finite groups, and hence they are pseudo-null as Iwasawa modules. Hence $\mathfrak{X}(K_\infty, T^*)$ and $\Ht^2_{\Iw}(K_\infty, T^*; \square^{(p)})$ have the same characteristic ideal, and we can use either to formulate the Main Conjecture.
   \end{remark}

 \subsection{P-adic families of ordinary Hilbert eigensystems}
  \label{sect:p1ordfam}

  We now suppose that $p = \fp_1 \fp_2$ is split, that $(p, \fn) = 1$, and that $\pi$ is ordinary at  both $\fp_i$; and we write $\fa_{\fp_i}^\circ$ for the unit root of the Hecke polynomial for each $i$.

  \begin{definition}
   We shall fix an integer $w \in \ZZ$, and write $\cW_K$ for the 2-dimensional rigid-analytic space over $L$ defined by
   \[ \{ (\kappa_1, \kappa_2, \tau_1,\tau_2) \in \cW^4: \kappa_1 + 2\tau_1 = \kappa_2 + 2\tau_2 = w\}.\]
  \end{definition}

  \begin{remark}
   Of course $\cW_K$ is isomorphic to $\cW^2$ by projection to the $\tau_i$, but it is convenient to carry around the $\kappa_i$ as well, for harmony with our notation for weights of classical Hilbert modular forms.
  \end{remark}

  \begin{notation}
   If $\Omega$ is a rigid-analytic subvariety of $\cW_K$, we write
   $\Omega_{\cl} \coloneqq \{ (\uk, \ut) \in \Omega \cap \ZZ^4 : k_i \ge 2\}$.
  \end{notation}

  \begin{definition}
   Let $\Omega$ be an open affinoid subdomain of $\cW_K$, such that $\Omega_{\cl}$ is Zariski-dense in $\Omega$.

   A \emph{family of $p$-ordinary eigensystems} $\upi$ over $\Omega$ consists of the following data:
   \begin{itemize}
     \item rigid-analytic functions $\lambda_{\fq} \in \cO(\Omega)$ for each prime $\fq \nmid \fn p$;
     \item rigid-analytic functions $\alpha_{\fp_1}, \alpha_{\fp_2} \in \cO(\Omega)^\times$ taking $p$-adic unit values;
   \end{itemize}
   such that the following condition is satisfied:
   \begin{itemize}
    \item for every $x = (\uk, \ut) \in \Omega_{\cl}$, there exists a cuspidal automorphic representation $\upi(x)$ of $\GL_2$ of conductor $\fn$ and weight $(\uk, \ut)$, and an embedding of its coefficient field into $L$, such that we have
    \begin{align*}
    \lambda_{\fq}(x) = a_{\fq}(\upi(x)) \text{ if $\fq \nmid \fn p$}, \quad
    \alpha_{\fp_i}(x) &= \fa_{\fp_i}^{\circ}(\upi(x)).
    \end{align*}
   \end{itemize}
  \end{definition}

  \begin{theorem}[Hida]
   For any $\pi$ of weight $(\uk, \ut)$ ordinary at $p$, there exists an open affinoid polydisc $\Omega$ in $\cW_K$ containing $(\uk, \ut)$, and a family of $p$-ordinary eigensystems $\upi$ over $\Omega$, with $\upi(\uk, \ut) = \pi$.
  \end{theorem}

  \begin{proof}
   This is a simple case of Hida theory for $\GL_2 / K$. It is straightforward to generalise the definition of a family of eigensystems to allow the base space to be an arbitrary affinoid with a map to $\cW$ (not necessary the inclusion of a subvariety), and to allow the specialisations at points in $\Omega_{\cl}$ where one of the $k_i$ is 2 to have level divisible by $\fp_i$ with Steinberg local component at $\fp_i$. With this slightly expanded definition, Hida theory shows that there is, in fact, a universal family of $\fp_1$-ordinary eigensystems, parametrised by a rigid space $\cE_K(\fn)$, the \emph{$p$-ordinary cuspidal eigenvariety}, which is finite and flat over $\cW_K$. (This space even has a natural formal-scheme model which is finite and flat over the Iwasawa algebra, although we shall not use this.)

   The automorphic representation $\pi$ determines a point in the fibre of $\cE_K(\fn)$ over $(\uk, \ut)$. Since the weight map is \'etale in a neighbourhood of any classical point, we can find a neighbourhood $\Omega \ni (\uk, \ut)$ which lifts isomorphically to a neighbourhood of $\pi$ in  $\cE_K(\fn)$. This gives the required family.
  \end{proof}

  \begin{proposition}
   For each sign $\eps$, there exists a finitely-generated locally free $\cO(\Omega)$-module $\cM^{\eps}_B(\upi)$, of rank 1, whose fibre at each $x \in \Omega_{\cl}$ is canonically identified with $M^{\eps}_B(\upi(x), L)$.
  \end{proposition}

  \begin{proof}
   This is a useful by-product of the construction of the ordinary eigenvariety. The eigenvariety can be defined as the spectrum of the Hecke algebra acting on the slope 0 part of the cohomology of a sheaf of Banach $\cO(\cW_K)$-modules on the $\GL_2/K$ symmetric space (overconvergent cohomology) whose fibre at $\lambda \in \Omega_{\cl}$ maps naturally to the classical local system of weight $\lambda$. So we obtain ``for free'' a coherent sheaf $\cM_B^{\eps}$ on $\cE_{K}(\fn)$, for each choice of signs $\eps$, whose fibres interpolate the cohomology of these Banach sheaves on the modular variety. At a classical point $\lambda \in \Omega_{\cl}$, the fibre of $\cM_B^{\eps}$ is the $(\upi(\lambda), \eps)$-eigenspace in the weight $\lambda$ overconvergent cohomology, but by the classicality theorem for overconvergent cohomology (see e.g.~\cite{barreradimitrovjorza}) this maps isomorphically to its classical analogue $M_B^\eps(\upi(\lambda), L)$.
     \end{proof}

%

 \subsection{The p-adic L-function in families}

  We now consider a family $\upi$ over $\Omega \subset \cW_K$, as above.

  \begin{theorem}
   \label{thm:gl2pLfam}
   For each sign $\eps$, there exists an element
   \[
    \cL^{\eps}_p(\upi) \in \cO(\Omega \otimes \cW^{\eps}) \otimes_{\cO(\Omega)} \cM^{\eps}_B(\upi)^\vee
    \]
    whose fibre at any $\lambda \in \Omega_{\cl}$ is the $p$-adic $L$-function $\cL^{\eps}_p(\upi(\lambda))$.
  \end{theorem}

  \begin{proof}
   This is a rephrasing of the main result of \cite{bergdallhansen17}. Since we are assuming that our families be ordinary (rather than just finite-slope), it can also be extracted from \cite{dimitrov13}.
  \end{proof}

  Similarly, for all $n\in\cR$ coprime to $\fn \disc$, we have an equivariant $p$-adic $L$-function
  \[ \cL^{\eps}_p(\upi; \Delta_m) \in \cO(\Omega \otimes \cW^{\eps}) \otimes_{\cO(\Omega)} \cM^{\eps}_B(\upi)^\vee \otimes_L L[\Delta_m],
  \]
  whose specialisation at any $\lambda \in\Omega_{\cl}$ is $\cL^{\eps}_p(\upi(\lambda); \Delta_m)$ as defined above.

\section{Preliminaries III: Siegel modular forms}
\label{sect:gsp4prelim}

 \subsection{Automorphic representations}

  Let $\ell_1 \ge \ell_2 \ge 2$ be integers. The pair $(\ell_1, \ell_2)$ determines a pair $\Pi_\infty^H$, $\Pi_\infty^W$ of irreducible $(\mathfrak{g}, K)$-modules for $\GSp_4(\RR)$, with $\Pi_\infty^H$ holomorphic and $\Pi_\infty^W$ generic, as in \cite[\S 10.1]{LSZ17} (but note that we are allowing $\ell_2 = 2$ here).

  We define a \emph{Siegel cuspidal automorphic representation} of weight $(\ell_1, \ell_2)$ to be a representation $\Pi$ of $\GSp_4(\Af)$ such that at least one of $\Pi \otimes \Pi_\infty^H$ and $\Pi \otimes \Pi_\infty^W$ is a cuspidal automorphic representation. (We normalise such that the central character of $\Pi$ is of finite order.)

  \begin{note}
   For concreteness, we note that $\Pi \otimes \Pi_\infty^H$ is automorphic if and only if $\Pi$ appears as a direct summand of the space of functions
   \[ \GSp_4(\Af) \times \cH^{(2)} \to \operatorname{Sym}^{\ell_1 - \ell_2}(\CC^2) \otimes \sideset{}{^{\ell_2}}\det, \]
   where $\cH^{(2)}$ is the genus 2 Siegel upper half-space, which are holomorphic, equivariant for left translation by $\GSp_4^+(\QQ)$, and rapidly-decreasing at the boundary. The restriction of such a function to any coset of $\cH^{(2)}$ is a classical Siegel modular form (vector-valued if $\ell_1 > \ell_2$).
  \end{note}

  We shall say that $\Pi$ is \emph{globally generic} if $\Pi \otimes \Pi_\infty^W$ is automorphic and globally generic in the sense of \cite[\S 9.1]{LPSZ1}.
  %

  \begin{note}
   If $\ell_2 \ge 3$, then $\Pi_\infty^H$ and $\Pi_\infty^W$ both have non-vanishing $(\mathfrak{g}, K)$-cohomology with coefficients in a suitable algebraic representation, and hence contribute to the \'etale cohomology of the $\GSp_4$ Shimura variety. However, we shall also be strongly interested in the ``non-regular weight'' case $\ell_2 = 2$.
  \end{note}

  Associated to $\Pi$ is an $L$-function $L(\Pi, s)$, normalised to have its functional equation centred at $s = \tfrac{1}{2}$. We write its Euler product in the form
  \[ L(\Pi, s - \tfrac{\ell_1 + \ell_2 - 3}{2}) = \prod_{\text{$\ell$ prime}} P_\ell(\Pi, \ell^{-s})^{-1}, \]
  where $P_\ell$ are polynomials. There exists a finite extension $E / \QQ$, the coefficient field of $\Pi$, such that all the $P_\ell(\Pi, X)$ have coefficients in $E[X]$.

 \subsection{Okazaki's newform theory}

  \begin{definition}
   For integers $M, N \ge 1$ with $M^2 \mid N$, let $K_1(N; M)$ be the \emph{quasi-paramodular subgroup} of level $(M, N)$ as defined in \cite{okazaki}:
   \[ K_1(N; M) = \left\{ x \in \GSp_4(\Af): x = 1 \bmod
   \begin{smatrix}
   \star & \star & \star & 1/L \\
   L     & \star & \star & \star \\
   L     & \star & \star & \star \\
   N     & L     & L     & M
   \end{smatrix}\right\}\qquad\text{where $L = N/M$.}
   \]
  \end{definition}

  We briefly recall some of the main results of \cite{okazaki}, which are generalisations of the results of \cite{robertsschmidt07} in the case $M = 1$:

  \begin{theorem}[Okazaki]
   Let $\Pi$ be an irreducible smooth representation of $\GSp_4(\Af)$ such that $\Pi_\ell$ is generic for every $\ell$. Suppose the central character of $\Pi$ has conductor $M$.

   Then there exists an integer $N$ with $M^2 \mid N$ such that $(\Pi)^{K_1(N; M)} \ne 0$. If $N$ is the least such integer, then the space of invariants is 1-dimensional, and we call $N$ the \emph{conductor} of $\Pi$. This integer coincides with the conductor of the functorial lift of $\Pi$ to $\GL_4(\Af)$.\qed
  \end{theorem}

%
%

 \subsection{The Whittaker period}

  Let $\Pi$ be a Siegel cuspidal automorphic representation of weight $(\ell_1, \ell_2)$ with $\ell_1 \ge \ell_2 \ge 3$ which is globally generic, and is not a Saito--Kurokawa or (untwisted) Yoshida lift. Let $N$ be the conductor of $\Pi$, and $M$ the conductor of its central character.

  As in \cite[\S 5.2]{LPSZ1}, we can realise (a suitable twist of) $\Pi$ over the number field $E$, as a direct summand of degree 2 coherent cohomology of the compactified $\GSp_4$ Shimura variety, giving a space
  \[ H^2(\Pi, E)\ \subseteq\ \varinjlim_{U, \Sigma} H^2\left(X^{\Sigma}_{G,E}(U), -\right). \]
  Here $U$ varies over open compacts in $\GSp_4(\Af)$ and $\Sigma$ over suitable toroidal boundary data, and $(-)$ denotes a coefficient sheaf depending on $(\ell_1, \ell_2)$. We can then define
  \[ M_W(\Pi, E) = H^2(\Pi, E)^{K_1(M, N)} \]
  which is a 1-dimensional $E$-vector space.

  \begin{remark}
   Alternatively one can define $M_W(\Pi, E)$ as a space of homomorphisms from the $E$-rational Whittaker model to $H^2(\Pi, E)$; cf.~\cite[Remark 6.4.2]{LPSZ1}.
  \end{remark}

  As in the $\GL_2 / K$ case, we have a canonical basis vector
  \[ \comp(\phi_\Pi) \in M_W(\Pi, \CC) \]
  given by the image of the standard Whittaker function at $\infty$ under the comparison isomorphism between coherent cohomology and $(\mathfrak{p}, \mathfrak{k})$-cohomology; and if we choose an $E$-basis $v_W$ of $M_W(\Pi, E)$, then we can define an Archimedean period $\Omega^W_\infty(\Pi, v_W) \in \CC^\times$ by
  \[ \comp(\phi_\Pi) = \Omega^W_\infty(\Pi, v_W) \cdot v_W. \]

  \begin{remark}
   In fact this construction is also valid when $\ell_2 = 2$, although we do not need this. However, the assumption $\ell_2 \ge 3$ is essential in the next section.
  \end{remark}

 \subsection{Galois representations}

  Let $\Pi$ as in the previous section, and $v \mid p$ a place of $E$. Using \'etale cohomology in place of coherent cohomology, we can construct a space $M_{\et}(\Pi, E_v)$. This is a 4-dimensional $E_v$-vector space with an action of $\Gal(\Qb / \QQ)$, unramified outside the primes dividing $Np$ (where $p$ is the prime below $v$); it is a canonical realisation of the spin Galois representation associated to $\Pi$ (cf.~\cite[Theorem 10.1.3]{LSZ17}). In particular, its Hodge numbers (see footnote on page \pageref{fn:hodge}) are $\{0, \ell_2 - 2, \ell_1 - 1, \ell_1 + \ell_2 -3\}$.

  The Faltings--Tsuji comparison theorem for \'etale cohomology (combined with the Fr\"olicher spectral sequence relating coherent and de Rham cohomology) gives a canonical isomorphism
  \[ M_W(\Pi, E) \otimes_E E_v \cong \Gr^{(\ell_2-2)}\DdR(\Qp, M_{\et}(\Pi, E_v)).\]

 \subsection{Rationality of $L$-values}

  We recall the following general rationality result for $\GSp_4$ spinor $L$-values, which is proved in \cite{LPSZ1}, building on the work of Harris in \cite{harris04}:

  \begin{theorem}
   \label{thm:LalgPi}
   For every pair of integers $(a_1, a_2)$ and Dirichlet characters $(\eta_1, \eta_2)$, with $0 \le a_1, a_2 \le \ell_1 - \ell_2$ and $(-1)^{a_1 + a_2}\eta_1(-1) \eta_2(-1) = -1$, there exists an element
   \[ \cL^{\alg}(\Pi, \eta_1, \eta_2, a_1, a_2) \in E(\eta_1, \eta_2) \otimes_E M_W(\Pi, E)^\vee,\]
   such that
   \[
    \langle \cL^{\alg}(\Pi, \eta_1, \eta_2, a_1, a_2), \comp(\phi_\Pi)\rangle =\prod_{i \in \{1, 2\}} \frac{a_i! (a_i + \ell_2 - 2)!L(\Pi \otimes \eta_i, \tfrac{1-\ell_1 + \ell_2}{2} + a_i)}
    {G(\eta_i)^2 (-2\pi i)^{\ell_2+ 2a_i}}
    .\qedhere
   \]
   \qedthere
  \end{theorem}

  As in the $\GL_2/K$ case, if we choose a basis $v_W$ of $M_W(\Pi, E)$, then the conclusion can be rephrased to state that
  \[ \frac{1}{\Omega_\infty^W(\Pi, v_W)} \prod_{i \in \{1, 2\}} \frac{L(\Pi \otimes \eta_i, \tfrac{1-\ell_1 + \ell_2}{2} + a_i)}
      {G(\eta_i)^2 (-2\pi i)^{\ell_2+ 2a_i}}
  \]
  lies in $E(\eta_1, \eta_2)$ and depends Galois-equivariantly on the $\eta_i$.

\section{Preliminaries IV: Iwasawa theory for \texorpdfstring{$\GSp_4 / \QQ$}{GSp4/Q}}
\label{sect:gsp4eigen}

 In this section, as in \cref{sect:gl2keigen}, we fix a prime $p$ and a finite extension $L/\Qp$ with ring of integers $\cO$.

 \subsection{Hecke parameters and ordinarity for \texorpdfstring{$\GSp_4 / \QQ$}{GSp4/Q}}
  \label{sect:gsp4heckeparam}

   Let $\Pi$ be a Siegel cuspidal automorphic representation of weight $(\ell_1, \ell_2)$, with $\ell_1 \ge \ell_2 \ge 2$, and fix an embedding of its coefficient field into $L$. We suppose $\Pi$ is unramified at $p$.

    We recall that the \emph{Hecke parameters} of $\Pi$ at $p$ are the quantities $(\alpha, \beta, \gamma, \delta)$ that are the eigenvalues (with multiplicity) of the normalised Hecke operator
   \[ U_{p, 1} = p^{(\ell_1 + \ell_2 - 6)/2} \left[ \Iw(p) \operatorname{diag}(p, p, 1, 1) \Iw(p)\right] \]
   acting on the Iwahori invariants of $\Pi_p$, ordered in such a way that $\alpha\delta = \beta\gamma = p^{(\ell_1 + \ell_2 - 3)}\psi(p)$ where $\psi$ is the (finite-order) central character of $\Pi$. The quantities $\left(\frac{\alpha\beta}{p^{(\ell_2 - 2)}},  \frac{\alpha\gamma}{p^{(\ell_2 - 2)}} \frac{\beta\delta}{p^{(\ell_2 - 2)}}, \frac{\gamma\delta}{p^{(\ell_2 - 2)}}\right)$ are the eigenvalues of the other normalised Hecke operator
   \[ U_{p, 2} = p^{(\ell_1 - 3)} \left[ \Iw(p) \operatorname{diag}(p^2, p, p, 1) \Iw(p)\right].\]
   We choose a prime $v \mid p$ of the coefficient field $E$ above $p$, and we order the parameters such that $\nu_v(\alpha) \le \dots \le \nu_v(\delta)$ (for some choice of extension of $\nu_v$ to $\overline{L}$).

   \begin{definition} \
    \begin{enumerate}[(i)]
     \item We have $\nu_v(\alpha) \ge 0$. If equality holds, we say $\Pi_p$ is \emph{Siegel-ordinary}, and $\alpha$ is a \emph{Siegel-ordinary refinement}.
     \item We have $\nu_v\left(\frac{\alpha\beta}{p^{(\ell_2 - 2)}}\right) \ge 0$. If equality holds, we say $\Pi_p$ is \emph{Klingen-ordinary}, and the \textbf{unordered} pair $(\alpha, \beta)$ is a \emph{Klingen-ordinary refinement}.
     \item We say $\Pi$ is \emph{ordinary} if it is both Siegel and Klingen ordinary, and we call the \textbf{ordered} pair $\left(\alpha, \beta\right)$ an \emph{ordinary refinement}.
    \end{enumerate}
   \end{definition}

   \begin{remark}
    In (ii), all the other pairwise products $\alpha\gamma, \beta\delta, \gamma\delta$ always have valuation strictly larger than $\ell_2 - 2$; so if a Klingen-ordinary refinement exists, it is unique. If $\ell_2 \ge 2$, or if $\ell = 2$ and $\Pi$ not Klingen-ordinary, then this is also true for Siegel-ordinarity: the Siegel-ordinary refinement is unique if it exists. On the other hand, if $\ell_2 = 2$ and $\Pi$ is Klingen-ordinary, then both $\alpha$ and $\beta$ are forced to have valuation 0, so $\Pi$ has two (possibly coincident) Siegel-ordinary refinements.
   \end{remark}

   \begin{theorem}[Urban]
    If $\Pi$ is ordinary at $p$ and $\ell_2 \ge 3$, then the 4-dimensional Galois representation $W_{\Pi} = M_{\et}(\Pi, L)$ has a (unique) filtration
    \[ 0\ \subsetneq\ \cF_1 W_{\Pi}\ \subsetneq\ \cF_2 W_{\Pi}\ \subsetneq\ \cF_3 W_{\Pi}\ \subsetneq\ W_{\Pi} \]
    where the Hodge numbers of the graded pieces are strictly increasing with $i$. We write $\cF^i W_{\Pi}^* \subseteq W_{\Pi}^*$ for the annihilator of $\cF_i W_\Pi$.
   \end{theorem}

   We are particularly interested in the quotient $\cF_2 W_{\Pi} / \cF_1 W_{\Pi}$, which has Hodge number $\ell_2 - 2$, and crystalline Frobenius eigenvalue $\beta$.

 \subsection{Cyclotomic $p$-adic $L$-functions for $\GSp(4)$}
 \label{sect:gsp4padicL}
  We first recall the relevant $p$-adic $L$-function, which is an equivariant version of the construction of \cite{LPSZ1}. Let $\Pi$ be as in \cref{sect:gsp4heckeparam}, and suppose that $\ell_2 \ge 4$, and that $\Pi$ is globally generic and Klingen-ordinary at $p$.

  \begin{definition}\label{def:R}
   Let $\cR$ be the set of square-free integers all of whose prime factors are $1 \bmod p$. For $m \in \cR$, define $\Delta_m$ to be the maximal quotient of $(\ZZ/ m\ZZ)^\times$ of $p$-power order.
  \end{definition}

  \begin{theorem}
   For any $m \in\cR$ coprime to $N$ and any $0 \le r \le \ell_1 - \ell_2$, there exists a $p$-adic $L$-function
   \[ \cL_p^{[r]}(\Pi; \Delta_m) \in \Lambda_L(\Gamma)^{\eps} \otimes M_W(\Pi, L)^\vee \otimes L[\Delta_m],\]
   where $\eps = (-1)^{1+r}$, such that
   \[
    \cL_p^{[r]}(\Pi; \Delta_m)(a + \chi, \eta) =
    \frac{\cE_p(\Pi \otimes \eta^{-1}, \chi, a) \cE_p(\Pi,\id, r)}{\eta(p)^{2m} \chi(m)^2} \cdot \cL^{\alg}(\Pi, \chi^{-1}\eta^{-1}, \id, a, r)
   \]
   for all primitive characters $\eta$ of $\Delta_m$ and locally-algebraic characters $a + \chi$ of $\Zp^\times$ with $(-1)^a \chi(-1) = \eps$.
  \end{theorem}

  Here $\cE_p(\Pi, \chi, a)$ are certain Euler factors as follows. If $\chi$ is trivial then
  \[ \cE_p(\Pi, \chi, a) =
  \left(1 - \tfrac{p^{a + \ell_2 - 2}}{\alpha} \middle)
  \middle(1 - \tfrac{p^{a + \ell_2 - 2}}{\beta} \middle)
  \middle(1 - \tfrac{\gamma}{p^{a + \ell_2 - 1}} \middle)
  \middle(1 - \tfrac{\delta}{p^{a + \ell_2 - 1}} \right).\]
  where $(\alpha, \beta)$ is the (unique) Klingen-ordinary refinement of $\Pi$. If $\chi$ is non-trivial of conductor $p^c$, then
  \[ \cE_p(\Pi, \chi, a) = \left(\frac{p^{2(a + \ell_2 - 1)}}{\alpha \beta}\right)^c.\]

  \begin{remark}
   In fact we can also allow $r$ to vary $p$-adically as well, but we do not need this here.
  \end{remark}

 \subsection{The Euler system}
  \label{sect:gsp4ES}

  Let $\Pi$ be as in \cref{sect:gsp4heckeparam}, and suppose that $\ell_2 \ge 3$, and that $\Pi$ is globally generic and Siegel-ordinary at $p$.

  \begin{definition}
   For $m \in \cR$ (see \cref{def:R}) we write $\QQ[m]$ for the maximal subfield of $\QQ(\mu_m)$ of $p$-power degree, so that $\Delta_m = \Gal(\QQ[m] / \QQ)$; and we write $\Qi[m]$ for the composite of $\QQ[m]$ and $\Qi \coloneqq \QQ(\mu_{p^\infty})$.
  \end{definition}

  Fix an integer $c > 1$ be an integer coprime to $6pN$.
  In \cite{LZ20}, refining results in \cite{LSZ17}, we constructed a family of Iwasawa cohomology classes
  \[ {}_c \bz_m^{[\Pi, r]} \in H^1_{\Iw}(\Qi[m], W_{\Pi}^*) \]
  for all $m \in \cR$ coprime to $c$, satisfying Euler-system norm compatibility relations as $m$ varies. (More precisely, we constructed them as classes over $\QQ(\mu_{mp^\infty})$, and the classes above are constructed by corestriction to the slightly smaller field $\Qi[m]$.) Here $r$ is an arbitrary integer with $0 \le r \le \ell_1 - \ell_2$, and the class ${}_c \bz_n^{[\Pi, r]}$ is supported on the component of weight space of sign $(-1)^{r + \ell_2 + 1}$.

  \begin{remark}
   In \emph{op.cit.}~these classes depend on an auxiliary choice of Schwartz functions $\Phi_S$ and Whittaker vectors $w_S$ at primes in $S$; we have here assumed that these are chosen ``optimally'', so that the normalised local zeta integrals are all 1. We have also chosen $c_1 = c_2 = c$ in the notation of \emph{op.cit.}.
  \end{remark}

  We recall that the image of ${}_c \bz_m^{[\Pi, r]}$ under the projection map
  \[ H^1_{\Iw}(\Qi[m], W_{\Pi}^*) \to H^1_{\Iw}(\Qpi[m], W_{\Pi}^* / \cF^1) \]
  is zero for every $m$ (\cite[Proposition 11.2.2]{LSZ17}).

 \subsection{The regulator formula}

  We now recall the regulator formula for these classes, which is the main result of \cite{LZ20}.

  \begin{definition}
   \label{def:PRreg}
   For $m \in \cR$, let
   \[ \LPR_{\Delta_m}: H^1_{\Iw}\left(\Qpi[m], \frac{\cF^1 W_{\Pi}^*}{\cF^2 W_{\Pi}^*}\right) \to \Lambda_L(\Gamma) \otimes L[\Delta_m] \otimes \Dcris\left(\Qp, \frac{\cF^1 W_{\Pi}^*}{\cF^2 W_{\Pi}^*}\right) \]
   be the equivariant Perrin-Riou regulator map.
  \end{definition}

  For details of the construction of this map, we refer to \cite{loefflerzerbes14} and (for $m = 1$) \cite[\S  8.2]{KLZ17}. Its construction relies on the fact that for any $m$-th root of unity $\zeta_m$, the element $\norm_{\QQ[m]}^{\QQ(\mu_m)}(\zeta_m)$ is a normal basis generator of $\QQ[m]$, i.e.~spans $\QQ[m]$ as a $\QQ[\Delta_m]$-module.

  We also have isomorphisms
  \[ M_W(\Pi, L) \xleftarrow{\ \cong\ } \operatorname{Fil}^1 \Dcris(\Qp, \cF_2 W_{\Pi})\xrightarrow{\ \cong\ } \Dcris(\Qp, \tfrac{\cF_2 W_{\Pi}}{\cF_1 W_{\Pi}}) = \Dcris\left(\Qp, \tfrac{\cF^1 W_{\Pi}^*}{\cF^2 W_{\Pi}^*}\right)^*;\]
  we denote by $\nu_{\dR}$ the image of $\nu \in M_W(\Pi, L)$ under these maps.

  \begin{theorem}[Equivariant regulator formula]
   For any $\nu \in M_W(\Pi, L)$, any $j \in \ZZ$ such that $j = 1 + r \bmod 2$ and $-1 \ge j \ge 2-\ell_2$, and any $m \in \cR$ coprime to $cN$, we have the special-value formula
   \begin{multline*}
    \left(1 - \tfrac{p^{r + \ell_2 - 2}}{\beta }\right) \left(1 - \tfrac{\gamma}{p^{r + \ell_2 - 1}}\right)\left\langle \LPR_{\Delta_m}\left({}_c z^{[\Pi, r]}_m\right)(j + \ell_2 - 2), \nu_{\mathrm{dR}} \right\rangle \\= (c^2 - c^{j+1 -r'}\psi(c) \sigma_c )(c^2 - c^{j+1 -r}\sigma_c ) \left\langle \cL_p^{[r]}(\Pi; \Delta_m)(j), \nu\right\rangle
   \end{multline*}
   as elements of $L[\Delta_m]$, where $r' = \ell_1-\ell_2-r$ and $\sigma_c$ is the image of $c$ in $\Delta_m$.
  \end{theorem}

  \begin{proof}
   This is a restatement (and slight generalisation to incorporate the equivariant twist) of \cite[Theorem A]{LZ20}. (Note that the extra Euler factor on the right-hand side is the `cowardly' factor appearing in \cite[Remark 17.3.10]{LZ20}.)
  \end{proof}

 \subsection{The \texorpdfstring{$\GSp_4$}{GSp(4)} ordinary eigenvariety}

  Let $N \ge 1$ be an integer coprime to $p$, and $\psi$ a Dirichlet character of conductor $M$ with $M^2 \mid N$, so the paramodular subgroup $K_1(N; M)$ is defined.

  \begin{theorem}[Tilouine--Urban]
   There exists a rigid-analytic space $\cE(N, \psi) \to \cW^2$, the \emph{$\GSp_4$ ordinary eigenvariety}, with the following property: if $(\ell_1, \ell_2) \in \ZZ$ with $\ell_1 \ge \ell_2 \ge 3$, then the fibre of $\cE(N, \psi)$ over $(\ell_1, \ell_2)$ bijects canonically with the set of cuspidal automorphic $\Pi$ of weight $(\ell_1, \ell_2)$ and central character $\psi$ such that
   \[ e_{p, 1}e_{p, 2} \cdot \left(\Pif \right)^{\left(K_1(N, M) \cap \Iw(p)\right)} \ne 0,\quad\text{where}\quad e_{p, i} \coloneqq \lim_{n \to \infty} (U_{p, i})^{n!}. \]
  \end{theorem}

  We shall say a point of $\cE(N, \psi)$ is a \emph{classical point} if it lies above a pair of integers $\ell_1 \ge \ell_2 \ge 3$ as above. At classical points, the weight map $\cE(N, \psi) \to \cW^2$ is \'etale.

  By construction, this space is equipped with a coherent sheaf $\cM_{\et}$ of Galois representations, whose fibre at a classical point $x$ is identified with $M_{\et}(\Pi, L)$ where $\Pi$ is the corresponding automorphic representation. For further details, see \cite{LZ20}.

  \begin{remark}
   The reader might justifiably expect to see at this point a theorem stating that $\cL_p^{[r]}(\Pi; \Delta_m)$ extends to the $\GSp_4$ eigenvariety, analogous to  \cref{thm:gl2pLfam} in the $\GL_2/K$ case. However, while we expect such a result to be true, it is \emph{not possible} to prove this using the methods of \cite{LPSZ1} alone: we can allow $\ell_1$ to vary $p$-adically with $\ell_2$ fixed, but we cannot vary $\ell_2$.
  \end{remark}
\section{Preliminaries V: Twisted Yoshida lifts}
\label{sect:yoshida}

 \subsection{Theta liftings}

  Let $\pi$ be a Hilbert cuspidal automorphic representation of weight $(\uk, \ut)$ in the sense of \cref{sect:prelimhilb} above; and assume the following additional conditions:
  \begin{itemize}
   \item We have $\varepsilon_{\pi} = \psi \circ \operatorname{Nm}$ for some Dirichlet character
   $\psi$;
   \item We have $\pi^{\sigma} \ne \pi$.
  \end{itemize}
  Note that $\psi$ is not uniquely determined by $\pi$: it is unique up to multiplication by the quadratic character associated to $K/\QQ$.
%

  \begin{theorem}
   There is a unique globally generic cuspidal automorphic representation $\Pi = \Theta(\pi, \psi)$ of $\GSp_4$, the \emph{twisted Yoshida lift} of $(\pi, \psi)$, characterised by the following two properties:
   \begin{itemize}
    \item We have the identity of $L$-series
    \[ L\left(\Pi,s\right)=L(\pi,s + \tfrac{w-1}{2})\]
    where $L(\Pi, s)$ is the spinor $L$-series.
    \item The central character of $\Pi$ is $\psi$.
   \end{itemize}
  \end{theorem}

  \begin{proof}
   Follows from \cite[Theorem 6.2]{vigneras86} (see also \cite[\S 7.3]{mokranetilouine02}.
  \end{proof}

  \begin{remark}
   Note that $\Theta(\pi, \psi)$ and $\Theta(\pi, \psi \varepsilon_K)$ are distinct, but their functorial transfers to $\GL_4$ coincide. We shall mostly be concerned with the case when $\varepsilon_{\pi} = 1$, in which case it is rather natural to choose $\psi$ to be the trivial character. However, in more general situations there would be no obviously `best' extension of $\varepsilon_{\pi}$ to a character of $\QQ$.
  \end{remark}

  \begin{proposition}
   The conductor of $\Theta(\pi, \psi)$ is $N = \Nm(\fn) \cdot \disc^2$.
  \end{proposition}

  \begin{proof}
   It follows from the results of \cite{okazaki} (applied to the local representations $\Pi_\ell$ for all $\ell$) that the paramodular conductor of $\Pi$ is equal to the conductor of the functorial lift of $\Pi_\ell$ to $\GL_4$. By construction, this coincides with the automorphic induction of $\pi$, whose conductor is given by the above formula.
  \end{proof}

  \begin{proposition}
   The weight $(\ell_1, \ell_2)$ of $\Pi$ is given by
   \[ (\ell_1, \ell_2) = \left(\frac{k_1 + k_2}{2}, \frac{|k_1 - k_2|}{2}+2\right).\]
   In particular, $\Pi_\infty$ is discrete series if $k_1 \ne k_2$ and limit of discrete series if $k_1 = k_2$.
  \end{proposition}

  \begin{proof}
   This follows from the compatibility of the $\Theta$-lifting with Harish-Chandra parameters at $\infty$.
  \end{proof}

  \begin{remark}
   It is extremely important to note that if $k_1 = k_2$, then $\pi$ is cohomological as a representation of $\GL_2 / K$, but its lifting $\Theta(\pi, \psi)$ is \emph{not} cohomological as a representation of $\GSp_4$. This will substantially complicate the proofs of our main theorems in the case of parallel-weight Hilbert modular forms (which is, of course, the most interesting one for arithmetic applications).
  \end{remark}

 \subsection{Matching periods}

  We now suppose that $\Pi = \Theta(\pi, \psi)$ is a twisted Yoshida lift, with $\pi$ of weight $(\uk ,\ut)$.

  \begin{proposition}
   \label{prop:aleph}
   If $k_1 - k_2 \ge 2$ and $k_2 \ge 3$, then there is a unique isomorphism of $E$-vector spaces
   \[ \aleph: M_B^\natural(\pi, E) \xrightarrow{\ \cong\ } M_W(\Pi, E)\]
   such that for any basis $v_+ \otimes v_-$ of $M_B^\natural(\pi, E)  = M_B^+(\pi, E)\otimes M_B^-(\pi, E)$, we have
   \[ \Omega_\infty^W\left(\Pi, \aleph(v_+ \otimes v_-)\right) = \Omega^+(\pi, v_+) \cdot \Omega^-(\pi, v_-).\]
  \end{proposition}

  \begin{proof}
   By twisting, we shall assume that $w = k_1$, so $t_1 = 0$ and $t_2 = -\frac{k_1-k_2}{2} \le 0$. Let $v_+$, $v_-$, and $v_W$ be arbitrary bases of the three spaces concerned, and let $\rho_+, \rho_-$ be arbitrary Dirichlet characters, one of either sign. We can clearly choose these to take values in $E$ (for instance, by supposing them to be quadratic). By \cref{prop:nonvanish}, we have $L(\pi \otimes \rho_+, 1)\cdot L(\pi \otimes \rho_-, 1) \ne 0$, so the complex numbers $\Lambda^+$, $\Lambda^-$ defined by
   \[ \Lambda^\eps = \frac{(\tfrac{k_1-k_2}{2})!L(\pi \otimes \rho_\eps, 1)}{(-2\pi i)^{2 + (k_1 -k_2)/2}G(\rho_\eps)^2} \]
   for $\eps \in \{\pm\}$ are both non-zero.

   On one hand, by \cref{prop:Lalgpi}, for each sign $\eps$ we have
   \[
    \langle \cL^{\alg}(\pi, \rho_\eps, 0), v_\eps\rangle = \frac{\Lambda^{\eps}}{\Omega_\infty^{\eps}(\pi, v_\eps)} \in E^\times.
   \]
   On the other hand, by \cref{thm:LalgPi} we have
   \[
    \langle \cL^{\alg}(\Pi, \rho_+, \rho_-, 0, 0), v_W\rangle = \\
    \frac{\Lambda^+ \Lambda^-}{\Omega^W_{\infty}(\Pi,v_W)}.
   \]
   Comparing these two formulae we deduce that the ratio
   \[ \frac{\Omega_\infty^W(\Pi, v_W)}{\Omega^+(\pi, v_+) \cdot \Omega^-(\pi, v_-)}\]
   is in $E^\times$. Thus we can re-scale $v_W$ so this ratio becomes 1, and we define $\aleph(v_+ \otimes v_-)$ to be this rescaled value. It is clear that this depends $E$-linearly on the $v_{\pm}$.
  \end{proof}

  \begin{corollary}
   \label{cor:compareLalg}
   If $\uk$ satisfies the hypotheses of \cref{prop:aleph}, then for any integers $a_i$ and characters $\eta_i$ such that $\cL^{\alg}(\Pi, \eta_1, \eta_2, a_1, a_2)$ is defined, we have
   \[ \aleph^\vee\left( \cL^{\alg}(\Pi, \eta_1, \eta_2, a_1, a_2)\right) = \cL^{\mathrm{alg}}(\pi, \eta_1, t_0 + a_1) \otimes \cL^{\mathrm{alg}}(\pi, \eta_2, t_0 + a_2).\]
  \end{corollary}

 \subsection{Ordinarity for theta-lifts}

  Let $\Pi = \Theta(\pi, \psi)$ be the theta-lift of a Hilbert cuspidal form.  Let $p$ be a prime, which we suppose to be split in $K$ and coprime to the level $\fn$ of $\pi$; then $\Pi_p$ is unramified. If $E$ is the field of definition of $\pi$, then $\Pi$ is also defined over $E$ (and if $\pi$ admits no exceptional automorphisms, then this is the minimal field of definition of $\Pi$).

  Let $v$ be a prime of $E$ above $p$; as above, we number the primes above $p$ as $\fp_1 \fp_2$ with $\sigma_i(\fp_i)$ lying below $v$.

  \begin{proposition} The Hecke parameters of $\Pi$ are given by the Weyl orbit of the quadruple
   \[ (\alpha, \beta, \gamma, \delta) = \left(p^{(k-k_1)/2} \fa^\circ_{\fp_1}, p^{(k-k_2)/2}\fa_{\fp_2}^\circ, p^{(k-k_2)/2}\fb_{\fp_2}^\circ,p^{(k-k_1)/2} \fb_{\fp_1}^\circ \right) \]
   up to the action of the Weyl group, where $k = \max(k_1, k_2)$.
  \end{proposition}
%

  \begin{corollary} \
   \begin{enumerate}[(i)]

    \item $\Pi$ is Klingen-ordinary at $p$ if and only if $\pi$ is ordinary at both $\fp_1$ and $\fp_2$, and its Klingen-ordinary refinement is $\fa_{\fp_1}^\circ \fa_{\fp_2}^\circ$.

    \item $\Pi$ is Siegel-ordinary at $p$ if and only if \emph{either} of the following conditions holds:
    \begin{itemize}
     \item $k_1 \ge k_2$ and $\pi$ is ordinary at $\fp_1$, in which case $\alpha = \fa_{\fp_1}^\circ$ is a Siegel-ordinary refinement;
     \item $k_2 \ge k_1$ and $\pi$ is ordinary at $\fp_2$, in which case $\alpha = \fa_{\fp_2}^\circ$ is a Siegel-ordinary refinement.
    \end{itemize}
    In particular, if $k_1 = k_2$ and $\pi$ is ordinary at both $\fp_1$ and $\fp_2$, then $\Pi$ has two Siegel-ordinary refinements $\fa^\circ_{\fp_1}$ and $\fa^\circ_{\fp_2}$ (which may coincide).
%
   \end{enumerate}
  \end{corollary}

  \begin{proof}
   The assertions regarding Siegel ordinarity are clear from the formulae above and the fact that the Hecke parameters $\fa_{\fp}^\circ$ and $\fb_{\fp}^\circ$ for $\fp \mid p$ are $p$-adically integral (and $\fb_{\fp}^\circ$ cannot be a unit).

   For Klingen ordinarity, since we have $\ell_2 - 2 = \frac{|k_1 - k_2|}{2}$, if we order the Hecke parameters as above we have $\frac{\alpha\beta}{p^{\ell_2 - 2}} = \prod_{\fp \mid p} \fa_{\fp}^\circ$, which is a unit if and only if all the $ \fa_{\fp}^\circ$ are, and for all the other reorderings $\frac{\alpha\beta}{p^{\ell_2 - 2}}$ can never be a unit.
  \end{proof}

  \begin{remark}
   One can check similarly that if $p$ is inert in $K$, then $\Pi$ is Klingen-ordinary if and only if $\pi$ is ordinary at $\fp = p\cO_K$. However, it is \emph{never} Siegel-ordinary if $k_1 \ne k_2$. This is an analogue in our setting of the well-known fact that elliptic modular forms of CM type are always ordinary at primes split in the CM field, but never ordinary at primes inert in the field, except in the special case of weight 1 forms.
  \end{remark}

 \subsection{Matching p-adic L-functions}

  We now place ourselves in the situation where the results of \cref{sect:gl2kpadicL} and \cref{sect:gsp4padicL} both apply, and compare them. That is, we suppose the following:
  \begin{itemize}
   \item $p$ is a prime split in $K$.
   \item $\pi$ is a Hilbert cuspidal automorphic representation of weight $(\uk, \ut)$, with $\pi^{\sigma} \ne \pi$, and level $\fn$ coprime to $p$.
   \item $\pi$ is ordinary at $\fp_1$ and $\fp_2$.
   \item $\psi$ is a Dirichlet character such that $\varepsilon_{\pi} = \psi \circ \Nm$.
   \item We have $k_1 - k_2 \ge 4$ and $k_2 \ge 3$ (``very regular weight'').
  \end{itemize}

  The last condition implies that the weight $(\ell_1, \ell_2)$ of $\Pi = \Theta(\pi, \psi)$ satisfies $\ell_1 - \ell_2 \ge 1$ and $\ell_2 \ge 4$, so the constructions of \cref{sect:gsp4padicL} apply to $\Pi$. One checks that
  \begin{equation}
   \label{eq:Efact} \cE_p(\Pi, \chi, a) = \prod_{\fp \mid p} \cE_\fp(\pi, \chi, t_0 + a),
  \end{equation}
  and hence, by the interpolating property and \cref{cor:compareLalg} above, the following is immediate:

  \begin{proposition}
   For any $0 \le r \le k_2 - 2$, we have
   \[ \aleph^\vee\left(\cL_p^{[r]}(\Pi, \Delta_m)(\mathbf{j})\right)= \cL_p^{\eps}(\pi, \Delta_m)(t_0 + \mathbf{j}) \otimes \cL^{-\eps}_p(\pi)(t_0 + r)\]
   as elements of $\cO(\cW^{\eps}) \otimes M_B^{\natural}(\pi, L)^\vee$, where $\eps = (-1)^{1+r}$.
  \end{proposition}

 \subsection{Euler systems for \texorpdfstring{$\pi$}{pi} of very regular weight}

  We maintain the assumptions of the previous section. By the Chebotarev density theorem we have
  \[  W_{\Pi}(-t_1)|_{G_{K}}  \cong \rho_{\pi, v} \oplus \rho_{\pi^{\sigma}, v}, \]
  and the representations $\rho_{\pi, v}$ and $\rho_{\pi^{\sigma}, v}$ are irreducible and non-isomorphic. So there is a unique 2-dimensional $L$-subspace $V(\pi)$ of $W_{\Pi}(-t_1)$ which is stable under $G_K$ and gives a representative of the isomorphism class $\rho_{\pi, v}$ of $G_K$-representations.

  We have an isomorphism of $G_{\Qp}$-representations
  \[ W_{\Pi}^*(t_1)|_{G_{\Qp}} \cong V(\pi)^*|_{G_{K_{\fp_1}}} \oplus V(\pi)^*|_{G_{K_{\fp_2}}}.\]
  and by considering Hodge--Tate weights, it is clear that we must have
  \[ \cF^1 W_{\Pi}^*(t_1) \cong \cF_{\fp_1}^+ V(\pi)^*|_{G_{K_{\fp_1}}} \oplus V(\pi)^*|_{G_{K_{\fp_2}}},\]
  and
  \[ \cF^2 W_{\Pi}^*(t_1) \cong \cF_{\fp_1}^+ V(\pi)^*|_{G_{K_{\fp_1}}} \oplus \cF_{\fp_2}^+ V(\pi)^*|_{G_{K_{\fp_2}}}.\]

  \begin{definition}
   \label{def:veryregES}
   For $m \in \cR$ coprime to $c \Nm(\fn) \disc$, and $0 \le r \le k_2 - 2$, we define ${}_c\bz_m^{[\pi, r]}$ to be the image of ${}_c\bz_m^{[\Pi, r]} \otimes e_{t_1}$ in the group
   \[ H^1_{\Iw}(\Qi[m], W_{\Pi}^*(t_1)) = H^1_{\Iw}(K_\infty[m], V(\pi)^*),\]
   where $e_{t_1}$ is the standard basis of $\Zp(1)$, and $K_\infty[m] = K \cdot \Qi[m]$.
  \end{definition}

  By the results recalled in \cref{sect:gsp4ES}, we have in fact
  \[ {}_c\bz_m^{[\pi, r]} \in \Ht^1_{\Iw}\left(K_\infty[m], V(\pi)^*; \square_{\Lambda}^+\right).\]

  Hence ${}_c\bz_m^{[\pi, r]}$ lies in the eigenspace of sign $\eps = (-1)^{r+t_2+1}$, and the reciprocity law can be rewritten as
  \[
   \cE_{\fp_2}(\pi, \id,t_2 + r) \cdot \aleph^\vee\left( \LPR_{\Delta_m}\left({}_c z^{[\pi, r]}_m\right)(j)\right) =
    {}_c C_m^{[r]}(\uk, \ut, j)\cdot \cL_p^{\eps}(\pi, \Delta_n)(j) \otimes
    \cL_p^{-\eps}(\pi)(t_2 + r)
  \]
  for $j$ of sign $\eps$ and satisfying $t_2 - 1 \ge j \ge t_1$, where we have written
  \[{}_c C_m^{[r]}(\uk, \ut, j) \coloneqq (c^2 - c^{(j+1-t_2-r')}\psi(c) \sigma_c )(c^2 - c^{(j+1-t_2 -r)}\sigma_c ) \in \Lambda_L(\Gamma) \otimes L[\Delta_n]. \]

\section{Yoshida lifts in families}

 The elements of \cref{def:veryregES} already give an Euler system for Hilbert cuspidal automorphic representations when the assumptions of that section are satisfied. However, it not terribly useful as stated, for two reasons. Firstly, the strong assumption on the regularity of the weight rules out the most interesting parallel-weight case. Secondly, we can only relate this Euler system to finitely many values of the $p$-adic $L$-function, all lying outside the Deligne-critical range. In this section (and those following), we shall set up the tools to solve both of these problems at once, using variation in a carefully chosen 1-parameter family inside the eigenvariety.

 \subsection{Families of twisted Yoshida lifts}
  \label{sect:yoshfam}

  Let $\upi$ denote a $p$-ordinary family of automorphic representations of $\GL_2 / K$ over some $\Omega \subset \cW_K$, as in \cref{sect:p1ordfam} above.

  All classical specialisations of $\upi$ have the same conductor $\fn$ and the same central character $\varepsilon$. So if $\psi$ is a choice of Dirichlet character such that $\psi \circ \Nm = \varepsilon$, then the twisted Yoshida lift $\uPi(x) = \Theta(\upi(x), \psi)$ exists for all $x \in \Omega_{\cl}$, and all these representations are generic and paramodular of level $N = \disc^2 \Nm(\fn)$.

  \begin{definition}
   Write $\Omega_{\cl} = \Omega_{\cl}^1 \sqcup \Omega_{\cl}^2 \sqcup \Omega_{\cl}^{\mathrm{bal}}$, where
   \begin{align*}
    \Omega_{\cl}^1 &\coloneqq \{ (\uk, \ut) \in \Omega^{\cl}: k_1 > k_2\},\\
    \Omega_{\cl}^2 &\coloneqq \{ (\uk, \ut) \in \Omega^{\cl}: k_1 < k_2 \},\\
    \Omega_{\cl}^{\mathrm{bal}} &\coloneqq \{ (\uk, \ut) \in \Omega^{\cl}: k_1 = k_2 \}.
   \end{align*}
  \end{definition}

  If $x \in \Omega^1_{\cl}$, then $\uPi(x)$ defines a classical point of $\cE(N, \psi)$ lying above $(\ell_1, \ell_2) = (\tfrac{k_1 + k_2}{2}, \tfrac{k_1 - k_2}{2} + 2) = (w-t_1 - t_2, t_2 - t_1 + 2)$. Similarly, if $x \in \Omega^2_{\cl}$, then we obtain a classical point of $\cE(N, \psi)$ of weight $(\tfrac{k_1 + k_2}{2}, \tfrac{k_2 - k_1}{2} + 2)$. However, if $x \in \Omega_{\cl}^{\mathrm{bal}}$, then $\uPi(x)$ has weight $(k, 2)$ for some $k \ge 2$, and thus is not cohomological; so it cannot define a classical point of $\cE(N, \psi)$.

  \begin{remark}
   Note that if $x \in \Omega^1_{\cl}$ then the ordinary $U_{p, 1}$-eigenvalue of $\uPi(x)$ is $\fa^\circ_{\fp_1}(\upi(x))$, whereas it is $\fa^\circ_{\fp_2}(\upi(x))$ if $x \in \Omega^2_{\cl}$.
  \end{remark}

  The following theorem is an instance of much more general results on functoriality of eigenvarieties. (We are grateful to Chris Williams for his explanations in relation to this issue.)

  \begin{theorem}
   We can find a closed immersion of rigid-analytic spaces
   \[ \Upsilon: \Omega \to \cE(N, \psi),\]
   such that for $x \in \Omega^1_{\cl}$, $\Upsilon(x)$ is the classical point associated to $\uPi(x)$.
  \end{theorem}

  \begin{proof}
   This follows from the same argument as the non-critical case of \cite[Proposition 5.4]{barrerawilliams}. The choice of $\upi$ determines a lifting of $\Omega$ to an affinoid $\tilde\Omega$ in the $\GL_2 / K$ cuspidal eigenvariety $\cE_K(\fn)$, such that the weight map restricts to an isomorphism $\tilde\Omega \to \Omega$. The eigenvariety $\cE_K(\fn)$ is a disjoint union of subvarieties indexed by Hecke characters modulo $\fn$, and we let $\cE_K(\fn, \psi)$ be the component corresponding to $\psi \circ \mathrm{Nm}$; clearly, $\tilde\Omega$ lies in $\cE_K(\fn, \psi)$.

    Using the general results on functoriality of eigenvarieties developed in \cite[\S 5]{hansen14}, we can define a morphism of rigid spaces $\upsilon: \cE_K(\fn, \psi) \to \cE(N, \psi)$, which is finite with closed image, interpolating the twisted Yoshida lifts at classical points with weight $k_1 > k_2$. This is compatible under the weight maps with the morphism $\cW_K \to \cW^2$ sending $(\kappa_1, \kappa_2, \tau_1, \tau_2)$ to $\left(w - \tau_1-\tau_2, \tau_2 - \tau_1\right)$. The latter map is clearly an isomorphism from $\Omega$ to its image in $\cW^2$; hence the restriction of $\upsilon$ to $\tilde\Omega$ is a closed immersion, and we can define $\Upsilon$ to be the composite $\Omega \xrightarrow{\,\sim\,} \tilde\Omega \xrightarrow{\,\upsilon\,} \cE(N, \psi)$.
  \end{proof}

  \begin{note}
   Note that if $x \in \Omega_{\cl} - \Omega^1_{\cl}$, then $\Upsilon^i(x)$ is \emph{not} a classical point in the $\GSp_4$ eigenvariety, since it will lie above a non-classical weight of the form $(\ell_1, \ell_2)$ with $\ell_2 < 3$. In particular, this construction is not at all symmetrical in the two embeddings of $K$.

   If $x \in \Omega^{2}_{\cl}$, then $\uPi(x)$ is cohomological and ordinary, so does lead to a point on $\cE(N, \psi)$, but $\Upsilon(x)$ is not this point; it is, rather, a non-classical ``companion'' of this point, analogous to the existence of weight $2-k$ overconvergent companion forms for weight $k$ CM-type modular forms.

   If $x \in \Omega^{\mathrm{bal}}_{\cl}$, then $\Upsilon(x)$ is non-cohomological. A particularly interesting case is when $x \in \Omega^{\mathrm{bal}}_{\cl}$ and $\upi(x)$ satisfies $\fa_{\fp_1}^{\circ} = \fa_{\fp_2}^{\circ}$ (which can happen); in this case, applying the construction to $\upi$ and $\upi^{\sigma}$ gives two distinct $p$-adic families passing through the same point of $\cE(N, \psi)$, showing that $\cE(N, \psi)$ is non-smooth at this point.
  \end{note}

%

 \subsection{A special one-parameter family}
  \label{sect:cZ}

  Let $\Omega$ be as above, and fix $(\uk, \ut) \in \Omega_{\cl}$.

  \begin{definition}
   Let $\cZ$ denote the one-dimensional subspace of $\Omega$ consisting of weights of the form $(k_1 + 4\mu, k_2 + 2\mu, t_1 - 2\mu, t_2 - \mu)$.
  \end{definition}

  Shrinking $\Omega$ if necessary, we can assume that $\cZ$ is an affinoid disc (isomorphic to $\operatorname{Max} L\langle X \rangle$). Note that $\cZ^1 \cap \Omega^1_{\cl}$ is Zariski-dense in $\cZ$, whereas the set of ``exceptional classical points'' $\cZ \cap (\Omega_{\cl} - \Omega^i_{\cl})$ is finite. It is these exceptional classical points which are the most interesting here.

  \begin{remark}
   We could more generally have considered weights of the form $(k_1 + 2a\mu, k_2 + 2b\mu, t_1 - a\mu, t_2 - b\mu)$ for any $\gamma = a/b \in [0, \infty]$. If $\gamma > 1$, we would obtain a 1-parameter family approaching $\pi$ through points of $\Omega^1_{\cl}$; if $\gamma < 1$, we would approach $\pi$ through $\Omega^2_{\cl}$. Switching the roles of the two embeddings $\sigma_1$, $\sigma_2$ corresponds to replacing $\gamma$ with $\tfrac{1}{\gamma}$.

   We have taken $\gamma = 2$, but any rational $\gamma \in (1, \infty)$ would give similar results. However, the values $\gamma = \{0, 1, \infty\}$ are special cases:
   \begin{itemize}

   \item  Taking $\gamma = 1$ would correspond to keeping $k_1 - k_2$ fixed on $\GL_2 / K$, and thus keeping $\ell_2$ fixed on $\GSp_4$. We could then replace the full $\GSp_4$ eigenvariety with the  ``small Klingen parabolic eigenvariety'' described in \cite{loeffler20}, which would allow us to drop the assumption that $p$ be split in $K$ (since Siegel ordinarity would play no role). However, if $k_1 = k_2$ this Klingen eigenvariety does not interpolate cohomological representations, and hence does not carry a natural family of Euler systems. This would rule out any interesting arithmetic applications in the $k_1 = k_2$ case, and so we have not pursued it here (although it might have interesting applications when $k_1 \ne k_2$ and $p$ is inert).

   \item Taking $\gamma = \infty$ would correspond to varying $k_1$ while keeping $k_2$ fixed, and hence the liftings to $\GSp_4$ would have weight $(\ell_1, \ell_2)$ with $\ell_1 - \ell_2$ fixed. In this case, we could replace $\cE_K$ with a small parabolic eigenvariety for $\GL_2 / K$ (which can be seen as an instance of the ``partially overconvergent cohomology'' of \cite{barreradimitrovjorza}), and $\cE(N, \psi)$ with a small parabolic eigenvariety for the Siegel parabolic of $\GSp_4$. By this method one can construct a non-trivial Euler system for $\rho_{\pi, v}^*$ for any $\pi$ which is ordinary at $\fp_1$, with no assumptions on the local factor at $\fp_2$ (it could even be supercuspidal). However, we do not know how to even formulate, let alone prove, an explicit reciprocity law in this setting. The picture for $\gamma = 0$ is similar with the roles of the two primes above $p$ interchanged.

   A second, more subtle issue with the $\gamma = \infty$ construction, even if $\pi$ is ordinary at both $\fp_i$, is that if $k_2 = 2$ we cannot rule out the possibility that the pullback to $U \times \cW$ of the $p$-adic $L$-function might be identically zero, since we do not have an analogue for $\GL_2 / K$ of the ``vertical'' non-vanishing theorems of Rohrlich \cite{rohrlich88} for $\GL_2 / \QQ$. This would imply that the Euler system would also vanish identically. (We had initially intended to use $\gamma = \infty$ in this project, and the switch to $\gamma = 2$ was motivated by this latter issue.)\qedhere
  \end{itemize}
  \end{remark}

  \begin{definition}
   We define a sheaf of Galois representations on $\cZ$ by
   \[ \cM_{\et} = \frac{(\Upsilon)^*(\cH_{\et}) |_{\cZ}}{\{\cO(\cZ)\text{-torsion}\}}.\]
  \end{definition}

  \begin{proposition}
   $\cM_{\et}$ is locally free of rank 4 over $\cO(\cZ)$; and if $x = (\uk', \ut')\in \cZ \cap \Omega_\cl$, then the fibre of $\cM_{\et}$ at $m$ is isomorphic to $\Ind_K^{\QQ}(\rho_{\upi(x), v})(t'_1)$ as a Galois representation.
  \end{proposition}

  \begin{proof}
   The local freeness is clear, since $\cO(\cZ)$ is a Dedekind domain\footnote{In fact, any locally-free module over a Tate algebra is free, but we do not need this here.} so any torsion-free module is locally free. If $x \in \cZ \cap \Omega^1_\cl$, then $\Upsilon(x)$ is a classical point of $\cE(N, \psi)$, so the fibre at $x$ is free of rank 4 and realises the Galois representation associated to $\uPi(x)$. This is isomorphic to $\Ind_K^{\QQ}(\rho_{\upi(x), v})(t'_i)$ by the construction of the Yoshida lifting.

   We must show that the above assertion also holds for the exceptional points $x \in \cZ \cap (\Omega_\cl - \Omega^1_{\cl})$. This follows by $p$-adic interpolation techniques: for a prime $\ell \nmid Np$, the trace  $\operatorname{tr}(\operatorname{Frob}_\ell^{-1} | \cM_{\et})$ is a rigid-analytic function on $\cZ$, as is the function given by
   \[ (\uk', \ut') \mapsto \begin{cases}
   \ell^{-t_1'} \sum_{\fq \mid \ell} \lambda_{\fq}(\uk', \ut') & \text{for $\ell$ split}\\
   0 & \text{for $\ell$ inert}.
   \end{cases}\]
   These functions agree at the Zariski-dense set $\cZ \cap \Omega^1_\cl$; hence they agree identically on $\cZ$, and in particular at the exceptional classical points.
  \end{proof}

  For $i = 1, 2$ let $\bt_i$ denote the unique character $\Zp^\times \to \cO(\Omega)^\times$ whose specialisation at any $(\uk', \ut') \in \Omega_{\cl}$ is $x \mapsto x^{t_i'}$. By restriction, we can regard $\bt_1$ and $\bt_2$ as $\cO(\cZ)^\times$-valued characters.

  \begin{corollary}
   There exists a uniquely determined rank 2 $\cO(\cZ)$-submodule $\cV(\upi) \subset \cM_{\et}(-\bt_1)$, stable under $G_K$, such that
   \[ \cM_{\et}(-\bt_1) = \cV(\upi) \oplus \sigma \cV(\upi) \]
   for any $\sigma \in G_{\QQ} - G_K$. For every $x \in \cZ \cap \Omega_{\cl}$, the fibre of $\cV(\upi)$ at $x$ is isomorphic to $\rho_{\upi(x), v}$.\qed
  \end{corollary}

  \begin{proposition}
   There exist unique locally-free rank 1 $\cO(\cZ)$-direct-summands $\cF^+_{\fp_i} \cV(\upi)$, for $i = 1, 2$, which are stable under $G_{K_{\fp_i}}$ and such that $\cF^+_{\fp_i} \cV(\upi)(\bt_i)$ is unramified as a $G_{K_{\fp_j}}$-module, with geometric Frobenius acting as multiplication by $\alpha_{\fp_i} \in \cO(\cZ)^\times$.
  \end{proposition}

  \begin{proof}
   This is a standard argument: if $\underline{\theta}$ denotes this family of characters, then $\Hom_{\cO(\cZ)[G_{K_{\fp_j}}]}(\underline{\theta}, \cV(\upi))$ is a finitely-generated, torsion-free $\cO(\cZ)$-module. As its specialisations at the Zariski-dense set $\cZ \cap \Omega_{\cl}$ have rank exactly 1, it must in fact be free of rank 1. Taking the image of a generator gives the required submodule. (Alternatively, and slightly more canonically, we can just define $\cF_{\fp_i}^+ \cV(\upi)$ to be the $I_{\fp_i}$-invariants of $\cV(\upi)(-\bt_j)$.)
  \end{proof}

\section{Euler systems in families}
 \label{sect:ESYfam}

 \subsection{Variation in families of Euler system classes}

  We now return to the situation of \cref{sect:cZ}, so $\upi$ is a family of $\GL_2 / K$ eigensystems over $\Omega \ni (\uk, \ut)$, and we have a distinguished 1-dimensional subspace $\cZ \subset \Omega$.

  \begin{theorem}
   Let $r \ge 0$. There exists a family of classes
   \[ {}_c\bz_m^{[r]}(\upi) \in H^1(K_{\infty}[m], \cV(\upi)^*)\]
   for $m \in \cR$ coprime to $c$, with the following properties:
   \begin{itemize}
    \item ${}_c\bz_m^{[r]}(\upi)$ is in the $(-1)^{r + t_2 + 1}$ eigenspace for complex conjugation.

    \item The localisation $\loc_{\fp_1}\left({}_c\bz_m^{[r]}(\upi)\right)$ lies in the image of the map
    \[ H^1(K_{\fp_1} \otimes \Qi[m], \cF_{\fp_1}^+ \cV(\upi)^*) \into
    H^1(K_{\fp_1} \otimes \Qi[m], \cV(\upi)^*).\]

    \item If $m, m\ell \in \cR$ then we have
    \[ \norm_{m}^{m\ell}\left({}_c\bz_{m\ell}^{[r]}(\upi)\right)
    = P_\ell(\ell^{-(1+\mathbf{j})}\sigma_\ell^{-1})\cdot  {}_c\bz_m^{[r]}(\upi). \]

    \item If $x = (\uk, \ut) \in \cZ \cap \Omega^1_{\cl}$ and $k_2-2 \ge r$, then the specialisation of ${}_c\bz_m^{[r]}(\upi)$ at $x$ is given by
    \[ \cE_{\fp_2}(\upi(x), \id,t_2 + r) \cdot {}_c\bz_m^{[\upi(x), r]}\]
    where ${}_c\bz_m^{[\upi(x), r]}$ is as in the previous section.

    \item All the classes ${}_c\bz_m^{[r]}(\upi)$ take values in a finitely-generated $\cO^+(\cZ)$-submodule of $\cV(\upi)^*$, where $\cO^+(\cZ)$ is the bounded-by-1 functions.
   \end{itemize}
  \end{theorem}

  \begin{proof}
   This theorem is exactly the specialisation to the case of Yoshida-type families of the general results on interpolation in families of $\GSp_4$ Euler systems proved in \cite[\S 7]{LRZ}. (In \emph{op.cit.} we varied the ``$r$'' variable $p$-adically as well, but this gives no additional benefit in the present application so we keep $r$ fixed here.)
  \end{proof}

 \subsection{Comparison isomorphisms in families}

  We would like to study the localisation of ${}_c\bz_m^{[r]}(\upi)$ at the other prime $\fp_2 \mid p$. By construction, if $(\uk, \ut) \in \cZ \cap \Omega^1_{\cl}$, then the map $\nu_{\dR}$ gives an identification
  \[ M_W(\Pi, L) \cong \mathbf{D}_{\dR}(K_{\fp_2}, \cF^+_{\fp_2} \cV(\pi)(1)) \cong \mathbf{D}_{\dR}\left(K_{\fp_2}, \frac{\cV(\pi)^*}{\cF_{\fp_2}^+}\right)^\vee. \]

  Composing with the map $\aleph$ above, we can write this as
  \[
   \aleph_{\dR, \fp_2}(\pi): M^\natural_B(\pi, L) \xrightarrow{\ \cong\ }
   \mathbf{D}_{\dR}\left(K_{\fp_2}, \cF_{\fp_2}^+ \cV(\pi)\right).
  \]

  We know that the source of the map $\aleph_{\dR, \fp_2}(\pi)$ interpolates in families: we have a sheaf $\cM_B^\natural(\upi)$ of $\cO(\cZ)$-modules (indeed of $\cO(\Omega)$-modules) whose fibre at $x \in \Omega_{\cl}$ is $\cM_B^\natural(\upi(x))$.

  More subtly, the target of this map also interpolates. Since the family of local Galois representations $\cF_{\fp_2}^+ \cV(\upi)$ over $\cO(\cZ)$ is a twist of a crystalline family by a character of $\Gamma$ (more precisely, $\cF_{\fp_2}^+ \cV(\upi)(\bt_2)$ is crystalline), there exists a locally-free rank 1 $\cO(\cZ)$-module $\cM^+_{\syn, \fp_2}(\upi)$ whose fibre at $x \in \cZ \cap \Omega_{\cl}$ is $\Dcris\left(\cF_{\fp_2}^+ \cV(\upi)(\bt_2)\right)$.

  One checks easily that for each $r$ there exists a non-zero-divisor ${}_c C_m^{[r]} \in \cO(\cZ\times \cW^{(-1)^{t_2 + 1 + r}})$ interpolating the ${}_c C_m^{[r]}(\uk, \ut, j)$.

  \begin{theorem} \
   \begin{enumerate}[(a)]
   \item The comparison isomorphisms $\aleph_{\dR, \fp_2}(\pi)$ interpolate meromorphically over $\cZ$. More precisely, if $\cK(\cZ) = \operatorname{Frac} \cO(\cZ)$, then we can find an isomorphism
      \[
       \aleph_{\dR, \fp_2}(\upi): \cK(\cZ) \otimes_{\cO(\cZ)} \cM_B^\natural(\upi) \xrightarrow{\ \cong\ }
       \cK(\cZ) \otimes_{\cO(\cZ)} \cM^+_{\syn, \fp_2}(\upi)
      \]
      such that for all but finitely many $x \in \cZ \cap \Omega^1_{\cl}$, $\aleph_{\dR, \fp_2}(\upi)$ is regular at $x$ and its specialisation at $x$ is $\aleph_{\dR, \fp_2}(\upi(x))$.
    \item For any $r$ and $n$, we have the explicit reciprocity law
    \[ \aleph_{\dR, \fp_2}(\upi)^\vee\left( \LPR\left({}_c z^{[\upi, r]}_m\right)\right) =
        {}_c C_n^{[r]}\cdot \left(\cL_p^{\eps}(\upi, \Delta_n)(\mathbf{j}) \otimes
        \cL_p^{-\eps}(\upi)(\bt_2 + r)\right)
        \]
    where $\eps = (-1)^{t_2 + 1 + r}$.
   \end{enumerate}
  \end{theorem}

  \begin{proof}
   Let $v_{\mathrm{syn}}$ be an element of $\cM^+_{\syn, \fp_2}(\upi)$ whose base-extension to $\cK(\cZ)$ is a basis, and similarly $v_\natural = v_+ \otimes v_-$ for the Betti cohomology.

   Let $\cY$ denote the set of points in $\cZ \times \cW$ at which the reciprocity law applies, i.e. the set of points $(x, j)$ with $x = (\uk', \ut')\in \cZ \cap \Omega^1_{\cl}$, and $j \in \ZZ$ with $t_2' - 1\ge j \ge t_1'$. This set $\cY$ is manifestly Zariski-dense in $\cZ \times \cW$.

   We first consider $r = 0$. We have a ``motivic'' $p$-adic $L$-function
   \[
    {}_c\cL^{\mathrm{mot}}(\upi) \coloneqq
    \left\langle \LPR\left({}_c z^{[\upi, 0]}_1\right), v_{\mathrm{syn}}\right\rangle \in \cO(\cZ \times \cW^{\eps}),
   \]
   where $\eps = (-1)^{1 + t_2}$. Similarly, choosing a basis $v_+ \otimes v_-$ of $\cM^\natural_B(\upi)$, we also have the ``analytic'' $p$-adic $L$-function
   \[
    \cL^{\mathrm{an}}(\upi) \coloneqq
     \langle \cL_p^{\eps}(\upi)(\mathbf{j}), v_+ \rangle \cdot \langle \cL_p^{-\eps}(\upi)(\bt_2), v_+ \rangle \in \cO(\cZ \times \cW^{\eps}).
   \]
   The analytic $p$-adic $L$-function is not a zero-divisor, by \cref{prop:nonvanish-padic}. By the reciprocity law, the motivic $p$-adic $L$-function is also non-vanishing at all points in $\cY$ of parity $\eps = (-1)^{1+t_2}$ that do not lie in the vanishing locus of either $\cL^{\mathrm{an}}(\upi)$ or $v_{\syn}$. So it is not a zero-divisor either, and we can form the ratio
   \[ R = \frac{{}_c\cL^{\mathrm{mot}}(\upi)}{{}_c C_1^{[0]} \cL^{\mathrm{an}}(\upi)} \in \operatorname{Frac}\cO(\cZ \times \cW^\eps)^{\times},\]
   i.e.~$R$ is a generically non-zero $p$-adic meromorphic function on $\cZ \times \cW^\eps$.

   Since $\cY$ is Zariski-dense, so is the set of $(x, j) \in \cY$ at which the numerator and denominator of $R$ are both non-zero. At any such point, the value $R(x, j)$ is equal to the (finite and non-zero) ratio of the periods $v_{\syn}(x)$ and $v_\natural(x)$, compared via $\aleph_{\dR, \fp_2}(\upi(x))$. In particular, it is independent of $j$. So we can conclude that $R$ itself is independent of the $\cW^{\eps}$ variable; that is, it lies in $\cK(\cZ)^{\times}$.

   So, scaling $v_\natural$ and $v_{\syn}$ appropriately, we can arrange that this ratio is in fact exactly 1. That is, the isomorphism $\aleph_{\dR, \fp_2}(\upi)$ sending $v_\natural$ to $v_{\mathrm{syn}}$ interpolates the $\aleph_{\dR, \fp_2}(\upi(x))$ for all but finitely many specialisations $x$.

   This proves part (a), and also part (b) for $r = 0$ and $n = 1$. The general case of (b) now follows readily from the Zariski-density of $\cY$.
  \end{proof}

  \begin{remark}
   By a closer inspection of the argument, one sees that for a given weight $x = (\uk', \ut') \in \Omega_{\cl}^1$, either the ``family'' comparison isomorphism $\aleph_{\dR, \fp_2}(\upi)$ specialises at $x$ to the known isomorphism $\aleph_{\dR, \fp_2}(\upi(x))$, or both sides of the explicit reciprocity law for $\upi(x)$ are 0 for every permissible value of $j$, $r$ and $n$. The latter possibility seems exceedingly improbable, but we are not able to rule it out a priori.
  \end{remark}

 \subsection{Behaviour at \texorpdfstring{$(\uk, \ut)$}{(k, t)}}

  We now consider the local behaviour of the above constructions at our chosen point $x_0 = (\uk, \ut)$. Note that the local ring $\cO_{\cZ, x_0}$ is a DVR, and it has a canonical uniformiser $X$ given by the image of $(k_1+4m, \dots) \mapsto m$.

  \begin{definition}
   We denote by $V(\pi)$ the fibre at $x_0$ of the sheaf $\cV(\upi)$.
  \end{definition}

  This is an $L$-vector space depending only on $\pi$ (and our choice of ordering of the two embeddings $K \into \RR$).

  Since $\aleph(\upi)$ is an isomorphism after inverting $X$, there is a uniquely-determined $q \in \ZZ$ such that if we let $v_{\natural} = v_+ \otimes v_-$ be a local basis of the Betti-cohomology sheaf around $x_0$, then $X^q \cdot \aleph_{\dR, \fp_2}(\upi)(v_\natural)$ is a local basis of $\cM_{\syn}^+$. Taking fibres at $x_0$ we have an isomorphism
  \begin{equation}
   \label{eq:alephprime}
   \aleph'_{\dR, \fp_2}(\pi): M_B^{\natural}(\pi, L) \xrightarrow{\ \cong\ } \Dcris\left(K_{\fp_2}, \cF^+_{\fp_2} V(\pi)\right),
  \end{equation}
  given by
  \[ v_{\natural} \pmod X \mapsto X^q \cdot \aleph_{\dR, \fp_2}(\upi)(v_\natural) \pmod X.\]

  \begin{remark}
   Note that even if $x_0 \in \Omega^1_{\cl}$, so that $\aleph_{\dR, \fp_2}(\pi)$ is defined, it does not follow formally that $\aleph'_{\dR, \fp_2}(\pi)$ necessarily agrees with $\aleph_{\dR, \fp_2}(\pi)$.
  \end{remark}

  Just as we have defined the comparision isomorphism $\aleph'_{\dR, \fp_2}(\pi)$ for $\pi$ as the ``leading term'' of a meromorphic object in the family $\upi$, we shall define an Euler system for $\pi$ using a similar leading-term construction. This construction is somewhat delicate, and may be of independent interest in other settings, so we have given a more general treatment in the next section; the discussion of the twisted Yoshida case will resume in \cref{sect:constructES}.

\section{Interlude: Kolyvagin systems and core ranks}
\label{sect:core}

 In this section we shall recall some results from \cite{mazurrubin04} and \cite[\S 12]{KLZ17}. We have formulated these results fairly generally, since they may be of use in other contexts than the present paper; later on we shall apply the results of this section with $T$ taken to be a lattice in $\Ind_K^{\QQ} V(\pi)^*$, with two different choices of the $\cO$-submodule $T^+$.

 \subsection{Setup}

  In this section, $p$ is prime, $L$ is a finite extension of $\Qp$ with ring of integers $\cO$, $\varpi$ is a uniformiser of $\cO$, and $T$ is a finitely-generated free $\cO$-module with a continuous action of $G_{\QQ}$, unramified outside a finite set of primes $S \ni p$. We let $\Gamma = \Gal(\Qi / \QQ)
  \cong \Zp^\times$, where $\Qi = \QQ(\mu_{p^\infty})$; and we set $\TT= T \otimes \Lambda(-\mathbf{q})$, where $\Lambda=\Lambda_{\cO}(\Gamma)$ and $\mathbf{q}$ its universal character.

  \begin{remark}
   Note that in \cite{mazurrubin04} the base ring is assumed to be local; $\Lambda$ is not a local ring, but it is a direct product of finitely many copies of the local ring $\Lambda_{\cO}( (1 + 2p\Zp)^\times)$, so the theory of \cite{mazurrubin04} extends immediately to this case.
  \end{remark}

  We suppose that the hypotheses (H.0)--(H.4) of \cite[\S 3.5]{mazurrubin04} are satisfied by $T$, and also by the twisted representation $T(\eta)$ for every $m \in \cR$ and every character $\Delta_m \to \bar{L}^\times$ (taking $R$ to be the ring of integers of the extension $L(\eta)$ generated by the values of $\eta$). Then these are also satisfied for $\TT$ (with $R = \Lambda$) and its twists. Finally, we suppose that $\rk T^{(c_\infty = 1)} = \rk T^{(c_\infty = -1)} = d$ for some $d \ge 1$, where $c_\infty$ is complex conjugation.

  \begin{remark}
   In fact the assumption that $\rk T^{(c_\infty = 1)} = \rk T^{(c_\infty = -1)}$ is not really needed -- it would suffice to assume that $T$ have rank $\ge 2$, and define $d = \rk T^{(c_\infty = -1)}$, as long as we replace $\QQ_\infty$ with its maximal totally-real subfield $\QQ_\infty^+$ throughout. This generalisation is not needed for in the present paper, but may be useful in future applications.
  \end{remark}

 \subsection{Selmer structures and core ranks}

  Let $T^+ \subseteq T$ be an $\cO$-direct summand stable under $G_{\Qp}$. As in \cref{sect:selmer} above, this defines a (derived) Selmer structure $\square^+$ on $\TT$, in which the local conditions at $\ell \ne p$ are the unramified ones, and the local condition at $p$ is given by the morphism of complexes
  \[ C^\bullet(\Qp, \TT^+) \longrightarrow C^\bullet(\Qp, \TT).\]

  If $\fQ$ is a height 1 prime of $\Lambda$, we write $T_{\fQ}$ for $\TT \otimes_{\Lambda} S_{\fQ}$, where $S_{\fQ}$ is the integral closure of $\Lambda/ \fQ$. Thus if $\fQ \ne (p)$, the ring $S_{\fQ}$ is the ring of integers of a finite extension of $L$, and $T_{\fQ}$ is the twist of $T$ by an $S_{\fQ}^\times$-valued character; whereas if $\fQ = (p)$, $S_{\fQ}$ is isomorphic to $(\cO/ \varpi \cO)[[X]]$.

  We use $T^+$ to define a \emph{simple} (i.e.~non-derived) Selmer structure $\square^+_{\can, \fQ}$ on $T_{\fQ}$ for each height 1 prime ideal $\fQ$ of $\Lambda$, as in \cite[\S 12.3]{KLZ17}; this generalises the canonical Selmer structure of \cite{mazurrubin04} (and reduces to it when $T^+ = T$). If $\fQ \ne p\Lambda$, then the Selmer structure $\square^+_{\can, \fQ}$ satisfies the Cartesian condition (H.6) of \cite[\S 3.5]{mazurrubin04}.

  \begin{proposition}
   If $\fQ$ is a height 1 prime ideal not containing $p$, then the ``core rank'' $\chi(T_{\fQ}, \square^+_{\can, \fQ})$, in the sense of \cite[Definition 4.1.11]{mazurrubin04}, is given by
   \[
    \chi(T_{\fQ}, \square^+_{\can, \fQ}) = \operatorname{max}\Big(0, d - \rank_{\cO}(T / T^+) + \rank_{S_{\fQ}} H^0(\Qp, (T^+_{\fQ})^*(1))\Big).
   \]
  \end{proposition}

  \begin{proof}
   This is a generalisation of \cite[Theorem 5.2.15]{mazurrubin04}, and is proved in exactly the same way.
  \end{proof}

  \begin{notation}
   For all but finitely many $\fQ$, the term $H^0(\Qp, (T^+_{\fQ})^*(1))$ is zero and hence the core rank is simply $\operatorname{max}\left(0, d - \rank_{\cO}(T / T^+)\right)$. We shall write $\chi(\TT, \square^+_{\Lambda})$ for this common value.
  \end{notation}

 \subsection{Kolyvagin systems}

  We consider the module
  \[ \overline{\KS}\left(\TT, \square^+_{\Lambda} \right) \coloneqq
  \varprojlim_k \varinjlim_j \KS\left(\TT  / \mathfrak{m}_{\Lambda}^k,\square^+_{\Lambda}, \cP_j\right)\]
  of Kolyvagin systems for $\TT$. Here, $\cP_j$ is a set of primes as defined in \cite[Def. 3.6.1]{mazurrubin04}, and $\Delta^+_{\Lambda}$ denotes the simple Selmer structure on $\TT / I\TT$, for any ideal $I$ of $\Lambda$, in which the local condition at $p$ is given by the image of $H^1(\Qp, \TT^+)$ in $H^1(\Qp,\TT / I\TT)$. (Note that this subspace is generally smaller than the image of $H^1(\Qp, \TT^+ / I\TT^+)$.)

  There is an analogous module $\overline{\KS}\left(T_{\fQ}, \square^+_{\can,\fQ} \right)$ for $T_{\fQ}$, for any height 1 prime $\fQ$, and we have a diagram of maps
  \[
   \begin{tikzcd}
    \overline{\KS}\left(\TT, \square^+_{\Lambda} \right) \rar["\bka \mapsto \bka^{(\fQ)}"] \dar["\bka \mapsto \kappa_1" {anchor=north, rotate=270, inner sep=1mm}] & \dar \overline{\KS}\left(T_{\fQ},\square^+_{\can, \fQ} \right) \\
    H^1_{\Iw}(\Qi, T) \rar & H^1(\QQ, T_{\fQ}).
   \end{tikzcd}
  \]

  \begin{theorem}[Mazur--Rubin]\label{thm:MR}
   If $\chi(T_{\fQ},\square^+_{\can, \fQ}) = 0$, then we have $\overline{\KS}\left(T_{\fQ}, \square^+_{\can,\fQ} \right) = 0$.
  \end{theorem}

  \begin{proof}
   By Theorem 4.2.2 of \cite{mazurrubin04}, if $\chi(T_{\fQ},\square^+_{\can, \fQ}) = 0$, then we have $\KS(T_{\fQ} / \varpi^k, \square^+_{\can, \fQ}, \cP_j) = 0$ for all $j \ge k$. Hence the inverse limit is clearly also zero.
  \end{proof}

  \begin{corollary}
   \label{cor:Ksyszero}
   If $\chi(\TT, \square^+_{\Lambda}) = 0$, then every Kolyvagin system $\bka \in \overline{\KS}\left(\TT,\square^+_{\Lambda}\right)$ satisfies $\kappa_1 = 0$.
  \end{corollary}

  \begin{proof}
   This follows from the same argument as \cite[Corollary 5.3.19]{mazurrubin04}: if $\kappa_1 \ne 0$, then for almost all $\fQ$ the image of $\kappa_1$ in $H^1(\QQ, T_{\fQ})$ is non-zero. So \emph{a fortiori} the image of $\bka$ in $\overline{\KS}(T_{\fQ}, \Lambda^+_{\can, \fQ})$ is nonzero. However, this module is identically 0, as we have just seen.
  \end{proof}

  \begin{remark}
   It seems likely that if $ \chi(\TT, \square^+_{\Lambda}) = 0$ then $\overline{\KS}\left(\TT,\square^+_{\Lambda}\right) = \{0\}$, but we have not proved this.
  \end{remark}

 \subsection{Euler systems}

  Now let $\cP$ be the set of primes that are $1 \bmod p$ and not in $S$, $\cR$ the set of square-free products of primes in $\cP$, and $\cK$ be the composite of $\Qi$ and the fields $\QQ[m]$ for $m \in \cR$. We define $\mathrm{ES}(T, T^+, \cK)$ as in \cite[Definition 12.2.1]{KLZ17}, with $\cP$ taken to be the set of primes in $\cR$; thus an element $\bc \in \mathrm{ES}(T, T^+, \cK)$ is given by a collection of classes
  \[ \bc = (c_m)_{m \in \cR}, \qquad c_m \in \Ht^1_{\Iw}(\Qi[m], T; \square^+). \]
  satisfying the Euler-system norm-compatibility relation. By Proposition 12.2.3 of \emph{op.cit.}, there is a canonical homomorphism
  \[ \mathrm{ES}(T, T^+, \cK) \to \overline{\KS}(\TT, \square_\Lambda^+), \]
  commuting with the maps from both sides to $H^1_{\Iw}(\Qi, T)$.

  \begin{remark}
   Strictly speaking this is only proved in \emph{op.cit.}~under the additional assumption that $\operatorname{Frob}_\ell^{p^k} - 1$ is injective on $T$ for every $\ell \in \cP$ and every $k \ge 0$. This condition tends to be automatically satisfied in applications, but in any case we can always reduce to this situation by replacing $T$ with $T \otimes \chi$ for a suitable character $\chi: \Gamma \to L^\times$, and using the fact that there are canonical bijections between Euler or Kolyvagin systems for $T$ and for $T(\chi)$. Each $\ell \in \cP$ only rules out countably many characters $\chi$, and since there are uncountably many characters $\Zp^\times \to L^\times$, we can always find a twist such that this condition holds. Cf. \cite[\S VI.1]{rubin00}.
  \end{remark}

  \begin{proposition}
   \label{prop:ESvanish}
   Suppose $\chi(\TT, \square^+_{\Lambda}) = 0$. Then the module of Euler systems $\mathrm{ES}(T, T^+, \cK)$ is zero.
  \end{proposition}

  \begin{proof}
   Let $\bc = (c_m)_{m \in \cR}$ be a non-zero element of $\mathrm{ES}(T, T^+, \cK)$. Then $c_m$ is non-zero for some $m \in \cR$. Since $H^1_{\Iw}(\Qi[m], T)$ is torsion-free, $c_m$ is non-zero as an element of $H^1_{\Iw}(\Qi[m], V)$, where $V = T[1/p]$. Extending $L$ if necessary, we can write this group as the direct sum $\bigoplus_{\eta} H^1_{\Iw}(\Qi, V(\eta))$, where $\eta$ varies over characters of $\Delta_m$. So $c_m$ must have non-zero image in $H^1_{\Iw}(\Qi, V(\eta))$ for some $\eta$.

   Replacing $T$ by $T(\eta)$ (and discarding the primes dividing $m$ from $\cP$), we can assume without loss of generality that $m = 1$, so we have an Euler system $\bc \in \mathrm{ES}(T, T^+, \cK)$ with $c_1 \ne 0$. However, if $\bka \in \overline{\KS}(\TT, \square_\Lambda^+)$ is the image of $\bc$, then we must have $\kappa_1 = 0$ by \cref{cor:Ksyszero}. This is a contradiction, since by construction $\kappa_1 = c_1$ as elements of $H^1_{\Iw}(\Qi, T)$.
  \end{proof}

 \subsection{Vanishing modulo ideals}

  We preserve the notation of the previous section.

  \begin{assumption}
   \label{ass:nonexcept}
   Assume that for all $m \in \cR$ the following groups are finite (or, equivalently, zero):
   \vspace{-1ex}
   \begin{multicols}{2}\begin{itemize}
    \item $H^0(\Qpi[m], T^+)$
    \item $H^0(\Qpi[m], (T^+)^*)$
    \item $H^0(\Qpi[m], (T/T^+))$
   \end{itemize}
   \end{multicols}
   \vspace{-2ex}
  \end{assumption}

  This implies, in particular, that the core ranks $\chi(T_{\fQ},\square^+_{\can, \fQ})$ are equal to $\chi(\TT,\square^+_{\Lambda})$ for \emph{all} characteristic 0 prime ideals $\fQ$ of height 1.

  \begin{proposition}
   \label{prop:divideES0}
   Suppose the above assumption holds, and the hypotheses of \cref{prop:ESvanish} are satisfied.  Let $\fQ$ be a height 1 prime ideal of $\Lambda$ not containing $p$, and suppose that $\bc \in \operatorname{ES}(T, \cK)$ is such that for every $m$, the image of $c_m$ in $H^1(\Qp[m], V_{\fQ} / V_{\fQ}^+)$ is zero, where $V_{\fQ} = T_{\fQ}[1/p]$.

   Then we have $\bc \in \fQ \cdot \operatorname{ES}(T, \cK)$.
  \end{proposition}

  \begin{proof}
   We shall show that $c_m$ has zero image in $H^1(\QQ[m], T_{\fQ})$, for all $m$. We shall in fact show that $c_m$ maps to 0 in $H^1(\QQ[m], V_{\fQ})$; this is equivalent, since $H^1(\QQ[m], T_{\fQ})$ is $p$-torsion-free.

   Suppose this is not the case. As above, after inverting $p$ and extending $L$, it follows that there is a character $\eta$ of $\Delta_m$ (for some $m$) such that the twisted Euler system $\bc^{\eta}$ has $c_1^{\eta}$ non-zero, and not $p$-torsion, modulo $\fQ$. As in the previous proof, without loss of generality we may assume $\eta = 1$.

   We claim that the Kolyvagin system $\bka^{(\fQ)}$ derived from $\kappa$ satisfies $p^a\bka^{(\fQ)}\in  {\KS}\left(T_{\fQ}, \square^+_{\can,\fQ} \right)$ for some $a$. This can be proved using the ``universal Euler system'' as in \cite[Proposition 12.2.3]{KLZ17}. In \emph{op.cit.} we used the existence of classes over the whole cyclotomic tower all satisfying the local condition, which does not apply here. However, since $H^0(\Qpi[m], T^+) = 0$ we certainly have $H^0(\Qp[m], T_{\fQ}^+) = 0$, and this allows us to apply the modified argument sketched in \cite[\S IX.1, `Construction of the derivative classes']{rubin00} (with $T^+$ and $\mathbb{T}^+ = \operatorname{Maps}(G_\QQ, T^+)$ in place of $W_M^+$ and $\mathbb{W}_M^+$). The factor $p^a$ arises since $H^1(\Qp[m], T_{\fQ}^+) \to H^1(\Qp[m], T_{\fQ})$ can have a finite kernel, and $H^1(\Qp[m], T_{\fQ}/T_{\fQ}^+)$ can have a finite torsion subgroup, and we need to kill both of these in order to obtain a map from the universal Euler system into $H^1(\Qp[m], T_{\fQ}^+)$.

   By \cref{thm:MR}, the Kolyvagin system $p^a \bka^{(\fQ)}$ is actually 0; and in particular $p^a \kappa_1^{(\fQ)} = p^a c_1 \bmod \fQ$ is zero, a contradiction.

   Consequently, we deduce that $c_m$ lies in the kernel of $H^1_{\Iw}(\Qi[m], T) \to H^1(\QQ[m], T_{\fQ})$ for every $m$. Fixing a generator $X_{\fQ}$ of $\fQ$, and considering the long exact sequence associated to multiplication by $X_{\fQ}$ on $\TT$, we conclude that $c_m \in X_{\fQ} \cdot H^1_{\Iw}(\Qi[m], T)$ for all $m$. Since $\TT / X_{\fQ} \TT$ has vanishing $H^0$, it follows that multiplication by $X_{\fQ}$ is injective on $H^1_{\Iw}(\Qi[m], T)$; so the elements $c_m' = X_{\fQ}^{-1} c_m$ are well-defined and satisfy the Euler system norm-compatibility relations. Thus $\bc = X_{\fQ} \cdot \bc'$ for an Euler system $\bc'$.
  \end{proof}

  The above method does not work for the height 1 prime ideals above $p$ (one for each component of $\cW$), since $c_m \ne 0 \bmod \varpi$ does not imply that there is an $\eta$ such that $c_1^{\eta} \ne 0 \bmod \varpi$. So in order to control powers of $p$ we shall use a slightly different argument.

   \begin{proposition}
    \label{prop:divideESp}
    Assume that $H^0(\Qpi, (T^{+})^* \otimes k_L) = 0$. Suppose that $\bc \in \operatorname{ES}(T, \cK)$ satisfies
    \[
     \loc_p(c_m) \in \ker\Big[H^1_{\Iw}(\Qpi[m], T) \to H^1_{\Iw}(\Qpi[m], (T / T^+) /\varpi^k)\Big],
    \]
    for some $k\ge 1$ and all $m \in \cR$. Then there exists an element $\mathcal{S} \in \Lambda$ such that $\Lambda / \mathcal{S}\Lambda$ is $p$-torsion-free, and for every height 1 prime \emph{not} containing $p$, the Kolyvagin system
    \[ \mathcal{S} \cdot \bka^{(\fQ)} \bmod \varpi^k \in \overline{\KS}\left(T_{\fQ} / \varpi^k, \square_{\mathrm{can}, \fQ} \right)\]
    obtained from $\mathcal{S} \cdot \bc$ is zero.
   \end{proposition}

  \begin{proof}
   It suffices to show that $\mathcal{S} \cdot \bka^{(\fQ)} \bmod \varpi^k$ is in the image of the map
   \[ \overline{\KS}\left(T_{\fQ} / \varpi^k,\square^{+}_{\can, \fQ} \right) \to
   \overline{\KS}\left(T_{\fQ} / \varpi^k,\square_{\can, \fQ} \right), \]
   since we have shown above that the left-hand space is 0.

   The classes $\kappa_m$ forming the Kolyvagin system $\bka^{(\fQ)}$ are defined as $S_{\fQ}[\Delta_m]$-linear combinations of the ``derivative'' classes $\kappa_{[\QQ, m, k]}$ defined in \cite[Chapter IV]{rubin00} (with $\bc$ replaced by its twist $\bc^{(\fQ)}$), so it suffices to prove that these satisfy the local condition at $p$ defining $\overline{\KS}\left(T_{\fQ} / \varpi^k,\square^{+}_{\can, \fQ} \right)$: that is, we must show that $\kappa_{[\QQ, m, k]}$ is in the image of the saturation of $H^1(\Qp, T_{\fQ}^+)$ in $H^1(\Qp, T_{\fQ})$. We shall, in fact, prove the stronger statement that it is in the image of $H^1(\Qp, T_{\fQ}^+)$.

   As in \cite[Proposition 12.2.3]{KLZ17}, we follow the argument of \cite[\S IV.6]{rubin00} (in its original form, not modified as in the previous proof). Let $m \in \cR$, and write $W_k = T_{\fQ}/ \varpi^k$ and $\mathbb{W}_k= \operatorname{Maps}(G_{\QQ}, W_k)$, and similarly $W_k^+$ and $\mathbb{W}_k^+$. The localisations at $p$ of the Euler system classes $c^{(\fQ)}_n$ for $n \mid m$ determine a compatible family of homomorphisms
   \[ \bc_{F[m], p}: \mathbf{X}_{F[m]} \to H^1(F[m]_p, W_k) \]
   for each finite extension of $\QQ$ inside $\QQ_\infty$, where $\mathbf{X}_{F[m]}$ is Rubin's universal Euler system. As in \emph{op.cit.}, we can lift these to homomorphisms
   \[ \mathbf{d}_{F[m], p}: \mathbf{X}_{F[m]} \to (\mathbb{W}_k / \Ind_\mathcal{D} W_k)^{G_{F[m]}},\]
   whose image under the boundary map to $H^1(F[m]_p, W_k)$ are the $\bc_{F[m], p}$. Each $\mathbf{d}_{F[m], p}$ is unique up to an element of $\Hom_{G_\QQ}(\mathbf{X}_{F[m]}, \mathbb{W}_k)$. Note that this construction relies on the vanishing of $H^0_{\Iw}(\Qpi[m], W_k)$.

   Since $H^0_{\Iw}(\Qpi, W_k / W_k^+)$ is also zero, the homomorphisms $\bc_{F[m], p}$ have compatible liftings to homomorphisms $\bc^+_{F[m], p}: \mathbf{X}_{F[m]} \to H^1(F[m]_p, W_k^+)$. Repeating the argument above, we obtain a second system of homomorphisms
   \[ \mathbf{d}^+_{F[m], p}: \mathbf{X}_{F[m]} \to
   (\mathbb{W}^+_k / \Ind_\mathcal{D} W^+_k)^{G_{F[m]}},
   \]
   whose image in $\mathbb{W}_k / \Ind_\mathcal{D} W_k$ differs from $\mathbf{d}_{F[m], p}$ by an element of $\Hom_{G_\QQ}(\mathbf{X}_{F[m]}, \mathbb{W}_k)$. As in the proof of \cite[Theorem IV.5.1]{rubin00}, this implies that the localisations at $p$ of the Kolyvagin classes $\kappa_{[F, m, k]}$ lie in the images of the maps $\mathbf{d}^+_{F[m], p}$. In particular, this holds for $F = \QQ$; so we have
   \[
    \loc_p \left(\kappa_{[\QQ, m, k]}\right) \in \operatorname{image}\left( H^1(\Qp, W_k^+) \to H^1(\Qp, W_k)\right).
   \]

   In general, $H^1(\Qp, W_k^+)$ may be larger than the image of $H^1(\Qp, T^+_{\fQ})$ modulo $\varpi^k$. However, the obstruction is exactly $H^2(\Qp, T^+_{\fQ})[\varpi^k]$, and by local Tate duality, $H^2(\Qp, T^+_{\fQ})$ is isomorphic to the Pontryagin dual of $H^0(\Qp, (T^+_{\fQ})^*(1) \otimes \Qp/\Zp)$, where $\wedge$ denotes Pontryagin duality. This is a submodule of
   \[ H^0\left(\Qpi, (T^+)^*(1) \otimes \Qp/\Zp\right) \otimes_{\cO} S_{\fQ} \]
   and since the $\Lambda$-module $H^0\left(\Qpi, (T^+)^*(1) \otimes \Qp/\Zp\right)$ is a finite group, we can choose the $\mathcal{S}$ above to annihilate it. Thus $\loc_p \left(\mathcal{S}\cdot \kappa_{[\QQ, m, k]}\right)$ lies in the image of $H^1(\Qp, T_{\fQ}^+)$ in $H^1(\Qp, W_k)$, as required.
  \end{proof}

 \subsection{Nested local submodules}

  We now consider the following situation: we are given two submodules $T^{++} \subset T^+ \subseteq T$ with $\rk_{\cO}(T^+) = d+1$ and $\rk_{\cO} T^{++} = d$. Then we can apply the preceding theory to either $T^+$ or $T^{++}$; and we have
  \[ \chi(\TT, \square^+_{\Lambda}) = 1, \qquad \chi(\TT, \square^{++}_{\Lambda}) = 0. \]
  We shall denote by $T^{\sharp}$ the rank 1 quotient $T^+/ T^{++}$. We assume that both $T^+$ and $T^{++}$ satisfy \cref{ass:nonexcept}.

  \begin{theorem}
   \label{thm:selmerbound}
   Suppose $\bc \in \mathrm{ES}(T, T^+, \cK)$, and let $\fQ$ be a height 1 prime ideal of $\Lambda$ and $k \ge 0$. Assume that for all $m \in \cR$, the image of $c_m$ under the composition
   \[
    \Ht^1_{\Iw}(\Qi[m], T; \square^+) \xrightarrow{\loc_p} H^1(\Qpi[m], T^+) \to H^1_{\Iw}(\Qpi[m], T^{\sharp})
   \]
   is divisible by $\fQ^k$. Then we have the Selmer group bound
   \[ \fQ^k \cdot \operatorname{char}_\Lambda \Ht_{\Iw}^2(\Qi,T;\square^+)\  \Big|\  \operatorname{char}_{\Lambda}\left(\frac{ \Ht_{\Iw}^1(\Qi,T;\square^+)}{\Lambda \cdot c_1} \right). \]
  \end{theorem}

  \begin{proof}
   First let us suppose $p \notin \fQ$. If $k = 0$, then $\fQ$ and $T^{++}$ play no role, and the claim reduces to Corollary 12.3.5 of \cite{KLZ17}. If $k > 0$, then \ref{prop:divideES0} shows that the entire Euler system is a multiple of $\fQ$, so we may use induction on $k$ to reduce to the case $k = 0$. Now suppose $\fQ$ is of residue characteristic $p$. Since we already know the result after localisation at any prime other than $\fQ$, it suffices to prove the divisibility after localising at $\fQ$. Proposition \ref{prop:divideESp} shows that after replacing $\bc$ by $\mathcal{S} \cdot \bc$, for some $\mathcal{S}$ which is a unit in the localisation at $\fQ$, the bounds on the Selmer groups of the characteristic 0 specialisations $T_{\fQ'}$ obtained from $\bka$ are all non-optimal by a factor of $\varpi^{k}$. Chasing this additional information through the proof we obtain the above sharpening of the final result.
  \end{proof}

  \begin{corollary}
  \label{cor:rank0bound}
   Under the same hypotheses as above, suppose $\lambda$ is a homomorphism of $\Lambda$-modules
   \[ H^1_{\Iw}(\Qpi, T^\sharp) \to \Lambda \]
   whose kernel and cokernel are pseudo-null. Then we have
   \[ \fQ^{k} \cdot \operatorname{char}_\Lambda \Ht_{\Iw}^2(\Qi,T;\square^{++})\  \big|\  \lambda(c_1). \]
  \end{corollary}

  \begin{proof}
   This follows from \cref{thm:selmerbound} by a standard argument using the exact triangle of Selmer complexes
   \[ \RGt_{\Iw}(\Qi, T; \square^{++}) \to \RGt_{\Iw}(\Qi, T; \square^{+}) \to \RG_{\Iw}(\Qpi, T^{\sharp})
  \to [+1], \]
  see e.g.~\cite[Theorem 11.6.4]{KLZ17}.
  \end{proof}

\section{An axiomatic ``leading term'' argument}
\label{sect:axiom}

 We now consider the following general setting:

 \begin{itemize}
  \item $L$ is a $p$-adic field.
  \item $\Omega$ is a one-dimensional affinoid disc, and $X$ is a uniformiser at a point of $\Omega$, so $\cO(\Omega) \cong L\langle X \rangle$.
  \item $\cM$ is a free module of finite rank over $\cO(\Omega)$ with a continuous action of $G_{\QQ}$, unramified outside some set $S \ni p$.
  \item $\cR$ denotes the set of square-free products of primes $\ell \notin S$ such that $\ell = 1\bmod p$.
 \end{itemize}

 \subsection{Leading terms of Euler systems}

  Suppose we are given some compatible collection
  \[ c_m \in H^1_{\Iw}(\Qi[m], \cM) : m \in \cR,\]
  satisfying the Euler-system norm compatibility relation.

  \begin{lemma}
   For each $m\in \cR$, we have $H^1_{\Iw}(\Qi[m], \cM)[X] = 0$.
  \end{lemma}

  \begin{proof}
   Since $H^0_{\Iw}(\Qi[m], \cM / X\cM) = 0$ the result follows by standard exact sequences.
  \end{proof}

  \begin{definition}
   For $m \in \cR$, let $h(m) \in \ZZ_{\ge 0} \cup \{\infty\}$ be the largest integer $n$ such that $c_m \in X^n \cdot H^1_{\Iw}(\Qi[m], \cM)$; and let $h = \inf_{m \in \cR} h(m)$.
  \end{definition}

  We shall assume that $c_1 \ne 0$, so that $h(1)$ and hence $h$ are finite. Since $H^1_{\Iw}(\Qi[m], \cM)[X] = 0$, there is a uniquely determined element $X^{-h} c_m \in H^1(\Qi[m], \cM)$ for every $m \in \cR$, and these satisfy the Euler system norm relation (since their images after multiplication by $X^h$ do, and multiplication by $X^h$ is an injection).

  \begin{definition}
   Let $V = \cM/ X\cM$, and let us write $c^{(h)}_m$ for the image of $X^{-h} c_m$ under the reduction map
   \[ H^1(\Qi[m], \cM) / X \into H^1(\Qi[m], V).\]
  \end{definition}

  Note that the injectivity of this reduction map follows from $H^0_{\Iw}(\Qi[m], \cM/X\cM) = 0$. By the definition of $h$, there exists $m \in \cR$ such that $c^{(h)}_m \ne 0$. Evidently, the classes $c^{(h)}_m$ satisfy the Euler-system norm relations for $\cM / X\cM$.

  \begin{proposition}
   Let $\cO^+(\Omega) \cong \cO\langle X \rangle$ denote the subring of power-bounded elements of $\cO(\Omega)$, and suppose that there is a finitely-generated $\cO^+(\Omega)$-submodule $\cM^+ \subset \cM$ (independent of $m$) which is stable under $G_{\QQ}$ and such that we have
   \[ c_m \in \operatorname{image}\left( H^1_{\Iw}(\Qi[m], \cM^+) \to H^1_{\Iw}(\Qi[m], \cM)\right)\, \forall m.\]
   Then we can find a $\cO$-lattice $T \subset V$ such that
   \[ c_m^{(h)} \in \operatorname{image}\left( H^1_{\Iw}(\Qi[m], T) \to H^1_{\Iw}(\Qi[m],V)\right)\, \forall m.\]
  \end{proposition}

  \begin{proof}
   Without loss of generality, we may assume that $\cM^+ / X \cM^+$ is $p$-torsion-free (by replacing $\cM^+$ with its reflexive hull, which is also finitely generated and $G_{\QQ}$-stable, and using the fact that a reflexive $\cO^+(X)$-module is free).

   Let us temporarily write $\cH^+ = H^1_{\Iw}(\Qi[m], \cM^+)$ and similarly $\cH = H^1_{\Iw}(\QQ(\mu_m), \cM) = \cH[1/p]$. The long exact sequence associated to multiplication by $X^h$ on $\cM^+$ gives an injection
   \[ \cH^+ / X^h \cH^+ \into H^1_{\Iw}(\Qi[m], \cM^+ / X^h \cM^+), \]
   and the latter module is $p$-torsion-free, since $\cM^+ / X^h \cM^+$ is a finitely-generated free $\cO$-module. So $\cH^+ / X^h \cH^+$ is also $p$-torsion-free.

   Since $c_m \in X^h \cH \cap \cH^+$, it follows that the class of $x$ in $\cH^+ / X^h \cH^+$ is $p$-torsion. From the above, we can conclude that in fact $c_m \in X^h \cH^+$. Thus $X^{-h} c_m \in H^1_{\Iw}(\Qi[m], \cM^+)$, and we may take $T$ to be the image of $\cM^+$ in $\cM / X \cM$.
  \end{proof}

  Since $H^1_{\Iw}(\Qi[m], T)$ is $p$-torsion-free, and the $c_m^{(h)}$ satisfy the Euler system norm relations after inverting $p$, it follows that $c_m^{(h)}$ are an Euler system for $(\cR, T)$. By construction, this is not the zero Euler system (that is, $c_m^{(h)} \ne 0$ for some $m$, although we cannot at this point rule out the possibility that $c_1^{(h)} = 0$).

 \subsection{Local conditions at p (I)}

  Since the classes $c_m^{(h)}$ lie in Iwasawa cohomology, they automatically satisfy the unramified local conditions at all primes away from $p$. However, local conditions at $p$ are a more delicate question.

  We suppose that the following conditions hold:
  \begin{itemize}
   \item $\cM$ has even rank $2d$, and the complex-conjugation eigenspaces $\cM^{(c_\infty = +1)}$ and $\cM^{(c_\infty = -1)}$ both have rank $d$.
   \item There exist saturated $\cO(\Omega)$-submodules $\cM^{++} \subset \cM^+ \subset \cM$, all stable under $G_{\Qp}$, such that $\cM^{++}$ has rank $d$ and $\cM^+$ has rank $d+1$.
   \item For each $m$ we have
   \[
    \loc_p(c_m) \in \operatorname{image}\Big(H^1_{\Iw}(\Qpi[m], \cM^{+}) \to H^1_{\Iw}(\Qpi[m], \cM)\Big).
   \]
  \end{itemize}

  \begin{notation}
   We let $V^{+}$ and $V^{++}$ denote the images of $\cM^+$ and $\cM^{++}$ in $V$.
  \end{notation}

  \begin{proposition}
   We have $\loc_p\left(c_m^{(h)}\right) \in \operatorname{image}\left(H^1_{\Iw}(\Qpi[m], V^+) \to H^1_{\Iw}(\Qpi[m], V)\right)$.
  \end{proposition}

  \begin{proof}
   Let us write $d_m = X^{-h} c_m \in H^1_{\Iw}(\QQ(\mu_{mp^\infty}),\cM)$. Then the image of $\loc_p(d_m)$ in the local Iwasawa cohomology of $\cM / \cM^+$ is annihilated by $X^h$. However, we have
   \[ H^1_{\Iw}\left(\Qpi[m], \cM / \cM^+\right)[X^h] = 0, \]
   since it is a quotient of $H^0_{\Iw}\left(\Qpi[m], \cM/\cM^+\right) = 0$. So $\loc_p(d_m)$ takes values in $\cM^+$, and hence its reduction modulo $X$ lands in $V^+$.
  \end{proof}

  So we have constructed an Euler system for $T$ with the Greenberg-type Selmer structure $\square_\Lambda^+$ determined by $\TT^+=\Lambda\otimes T^+$; compare \cite[\S 12]{KLZ17}.

  \begin{proposition}
   \label{prop:toomuchSelmer}
   Suppose that conditions (H.0)--(H.4) of \cite{mazurrubin04} are satisfied by $T$ and also by $T(\eta)$, for every $m \in \cR$ and every $\eta: \Delta_m \to \Qb^\times$. Then there exists an $m$ such that the image of $c_m^{(h)}$ in $H^1_{\Iw}(\Qpi[m], V^+ / V^{++})$ is non-zero.
  \end{proposition}

  \begin{proof}
   By construction, the $c_m^{(h)}$ are not all zero. So if $c_m^{(h)}$ lands locally in $V^{++}$ for every $m$, then $(c_m^{(h)})_{m \in \cR}$ is a non-zero element of the module $\operatorname{ES}(T, T^{++}, \cK)$ in the notation of \cref{prop:ESvanish}. Since $T^{++}$ has rank $d$, the proposition shows that $\operatorname{ES}(T, T^{++}, \cK) = \{0\}$, a contradiction.
  \end{proof}

 \subsection{Regulators at p}

  We now focus our attention on the subquotient $\cM^{\sharp} \coloneqq \cM^+ / \cM^{++}$. In the examples which interest us, there is some character $\tau: \Zp^\times \to \cO(\Omega)^\times$ such that $\cM^{\sharp}(\tau)$ is an unramified character. Hence we can make sense of $\Dcris(\cM^{\sharp})$, as a free module of rank 1 over $\Omega$. Then we have the $\Delta_m$-equivariant \emph{Perrin-Riou regulator} map associated to $\cM^{\sharp}$,
  \[
   \LPR_{\cM^{\sharp}, \Delta_m}: H^1_{\Iw}(\Qpi[m], \cM^{\sharp}) \to \Dcris(\cM^{\sharp}) \htimes \Lambda \otimes L[\Delta_m].
  \]
  This is compatible under reduction modulo $X$ with the regulator map $\LPR_{\Delta_m}$ for the 1-dimensional $\Qp$-linear representation $V^{\sharp} = V^+/V^{++}$ (cf.~\cref{def:PRreg}). We note that the mod $X$ map
  \[ \LPR_{V^{\sharp}, \Delta_m}: H^1_{\Iw}(\Qpi[m], V^{\sharp}) \to \Dcris(V^{\sharp}) \htimes \Lambda \otimes L[\Delta_m]\]
  has kernel equal to the torsion subgroup of $H^1_{\Iw}(\Qpi[m], V^{\sharp})$. We shall suppose that $H^0(\Qpi, V^{\sharp}) = 0$, so that $\LPR_{V^{\sharp}, \Delta_m}$ is injective and its cokernel is pseudo-null.

  \begin{remark}
   In general the map $\LPR_{\cM^{\sharp}, \Delta_m}$ and $\LPR_{V^{\sharp}, \Delta_m}$ can have poles at certain characters of $\Gamma$; but these poles only appear if $\operatorname{Frob}_p = 1$ on the unramified twist of $\cM^{\sharp}$ (resp.~$V^{\sharp}$) and this is also ruled out by the above assumption.
  \end{remark}

  \begin{assumption}
   There exists a collection of elements
   \[ \cL_m \in \cO(\Omega) \htimes \Lambda \otimes L[\Delta_m] \quad\text{for all $m \in \cR$}, \]
   and a section
   \[ \xi \in \operatorname{Frac} \cO(\Omega) \otimes_{\cO(\Omega)} \Dcris(\cM^{\sharp})^\vee, \]
   such that
   \begin{itemize}
    \item $\cL_m \bmod X$ is non-zero as an element of $\Lambda \otimes L[\Delta_m]$, for some $m$.
    \item For every $m$, we have
    \[ \langle \LPR_{\cM, \Delta_m}\left(\loc_p c_m\right), \xi \rangle = \cL_m.\]
   \end{itemize}
  \end{assumption}

  \begin{theorem}
   Under the above hypotheses, we necessarily have
   \[ \ord_X \left(\xi\right)= -h,\]
   where $h$ is as defined above; and if $\bar{\xi} = X^h \xi \bmod X$, then $\bar{\xi}$ is a basis of $\Dcris(V^{\sharp})^\vee$, and we have
     \[  \left\langle \LPR_{V, \Delta_m}\left(\loc_p c_m^{(h)}\right), \bar{\xi} \right\rangle = \cL_m \bmod X.\]
  \end{theorem}

  \begin{proof}
   Replacing every $c_m$ with $X^{-h} c_m$, and $\xi$ with $X^h \xi$, we may assume that $h = 0$. Thus $c_m^{(h)} = c_m \bmod X$.

   Hence, if $\Xi$ is a local basis of $\Dcris(\cM^{\sharp})^\vee$ around $X$, then $\cL_m^\flat = \langle \LPR_{\cM^{\sharp}, \Delta_m}(c_m), \Xi\rangle$ is regular at $X = 0$, and its value modulo $X$ is $\left\langle \LPR_{V^\sharp, , \Delta_m}(c_m \bmod X), \Xi \bmod X\right\rangle$, which is non-zero for some $m$, by \cref{prop:toomuchSelmer} and the injectivity of the regulator map.

   Scaling $\Xi$ by a unit in the localisation $\cO(\Omega)_{(X)}$ we can arrange that $\xi = X^t \Xi$ for some $t \in \ZZ$, so that $\cL_m = X^t \cL_m^\flat$. Since $\cL_m$ is regular at $X = 0$ for all $m$, we must have $t \ge 0$. On the other hand, if $t > 0$ then $\cL_m \bmod X$ is zero for all $m$, which contradicts the assumptions.
  \end{proof}

\section{Construction of the Euler system for \texorpdfstring{$\pi$}{pi}}
\label{sect:constructES}

 We now explain the outcome of the above process in our twisted-Yoshida setting. Recall that by the end of \cref{sect:ESYfam} we had singled out a canonical $L$-vector space $V(\pi)$ realising the Galois representation $\rho_{\pi,v}$, and by \eqref{eq:alephprime}, an isomorphism
 \[
  \aleph'_{\dR, \fp_2}(\pi): M_B^{\natural}(\pi, L) \xrightarrow{\ \cong\ } \Dcris\left(K_{\fp_2}, \cF^+_{\fp_2} V(\pi)\right).
 \]

 We now suppose that Condition (BI) of \cref{def:BI} is satisfied for $\pi$. Hence, for any $\cO$-lattice $T(\pi)^* \subset V(\pi)^*$, the representation $T = \Ind_{K}^{\QQ}T(\pi)^*$ satisfies the hypotheses of \cref{sect:core}.

 We can therefore apply the constructions of \cref{sect:axiom} with $\cM$ taken to be the module $\Ind_{K}^{\QQ} \cV(\upi)^* = \cM_{\et}^*(\bt_1)$ of \cref{sect:cZ}. We then have
 \[ \cM |_{G_{\Qp}} \cong \left(\cV(\upi)^*|_{G_{K_{\fp_1}}}\right) \oplus \left(\cV(\upi)^*|_{G_{K_{\fp_2}}}\right),
 \]
 and we define the submodules $\cM^{++}$ and $\cM^+$ in terms of this isomorphism by
 \begin{align*}
  \cM^{++} & \coloneqq
  \left(\cF^+_{\fp_1}\cV(\upi)^*|_{G_{K_{\fp_1}}}\right) \oplus
  \left(\cF^+_{\fp_2}\cV(\upi)^*|_{G_{K_{\fp_2}}}\right), \\
  \cM^{+} & \coloneqq
  \left(\cF^+_{\fp_1} \cV(\upi)^*|_{G_{K_{\fp_1}}}\right) \oplus
  \left(\cV(\upi)^*|_{G_{K_{\fp_2}}}\right).
 \end{align*}

 If $V^+$ and $V^{++}$ are the images of these in $\Ind_K^{\QQ} V(\pi)^*$, and $T^+, T^{++}$ the intersections of $V^+, V^{++}$ with  $\Ind_K^{\QQ} T(\pi)^*$, then Shapiro's lemma gives isomorphisms
 \begin{align*}
  \RGt_{\Iw}\left(\Qi, T; \square^{++}\right)
  &=\RGt_{\Iw}\left(K_\infty, T(\pi)^*; \square^{(p)}\right),
  & \RGt_{\Iw}\left(\Qi, T; \square^{+}\right)
  &=\RGt_{\Iw}\left(K_\infty, T(\pi)^*; \square^{(\fp_1)}\right)
 \end{align*}
 where the Selmer complexes over $K_\infty$ are as defined in \cref{sect:selmer}.

 \subsection{The Euler system}

  \begin{theorem}
   For each sign $\eps$, and each $m \in \cR$, there is a canonical map
   \[ \bz^{\eps}_{m}: M_B^{-\eps}(\pi, L)^\vee \to \Ht^1_{\Iw}(K_{\infty}[m], V(\pi)^*; \square^{(\fp_1)})^{\eps} \]
   with the following properties: the $\bz_m^{\eps}$ satisfy the Euler-system norm-compatibility relations as $m$ varies and the composite map
   \begin{multline*}
    M_B^{-\eps}(\pi, L)^\vee \xrightarrow{\bz_m^{\eps}}
    \Ht^1_{\Iw}\left(K_{\infty}[m], V(\pi)^*; \square^{(\fp_1)}\right)^{\eps} \xrightarrow{\LPR \circ \loc_{\fp_2}}
    \Dcris\left(K_{\fp_2}, V(\pi)^* / \cF^+_{\fp_2} \right) \otimes
    \Lambda(\Gamma)^{\eps} \otimes L[\Delta_m] \\
    \xrightarrow{\aleph'_{\dR, \fp_2}(\pi)^\vee}
    M_B^{\natural}(\pi, L)^\vee \otimes \Lambda(\Gamma)^{\eps} \otimes L[\Delta_m]
   \end{multline*}
   is given by
   \[ \gamma_{-\eps} \mapsto \gamma_{-\eps} \otimes \cL_p^{\eps}(\pi, \Delta_m) \quad\forall\ \gamma_{-\eps} \in M_B^{-\eps}(\pi, L)^\vee. \]
   Equivalently, for any $v_+, v_-$ bases of $M_B^{\pm}(\pi, L)$, we have
   \[
    \langle \bz^{\eps}_{m}\left(\gamma_{-\eps}\right), \aleph'_{\dR, \fp_2}(\pi)(v_+ \otimes v_-) \rangle
    =\langle \cL_p^{\eps}(\pi, \Delta_m), v_{\eps}\rangle \cdot \langle \gamma_{-\eps}, v_{-\eps}\rangle.\]
  \end{theorem}

  \begin{remark}
   The formulation adopted above is heavily influenced by the work of Kato for $\GL_2 / \QQ$, in particular \cite[Theorem 12.5]{kato04}. Note that the curious ``switch'' of signs -- that the $\epsilon$-part of Betti cohomology maps to the $(-\epsilon)$-part of Iwasawa cohomology -- is present in Kato's case as well; see the last displayed formula of Theorem 12.5(2) of \emph{op.cit.}.
  \end{remark}

  \begin{proof}
   We now indicate how to extract the above results from the general theory of \cref{sect:axiom}.

   We begin by noting that if a collection $\bz^{\eps}_{m}$ satisfying the conditions exists, then it is unique, since otherwise we would obtain a non-zero Euler system for a Selmer structure of Euler characteristic 0, and (as we have seen in the previous section) this is not possible.

   Let $r \ge 0$ be any integer, and $c > 1$ coprime to $6pN$. Then for $\eps = (-1)^{t_2 + r + 1}$ we have a family of classes
   \[ {}_c z_m^{[\upi, r]} \in \Ht^1_{\Iw}(K_\infty, \cV(\upi)^*; \square^{(\fp_1)})^{\eps}.\]

   Given bases $v_+, v_-$ of $M_B^{\pm}(\upi)$, we choose the $\xi$ and $\cL_m$ of the previous section as follows. We set $\cL_m = {}_c C_m^{[r]} \cdot \langle \cL^{\eps}_p(\upi), v_{\eps}\rangle$. We may assume that there is some $m$ for which $\cL_m \bmod X$ is non-zero (otherwise the conclusion of the theorem is satisfied with $\bz_m^\eps$ taken as the zero map).

   We define $\xi$ to be the meromorphic section
   \[
    \xi(v_{\eps}) = \frac{1}{\langle \cL_p^{-\eps}(\pi)(\bt_2 + r), v_{-\eps}\rangle}\cdot \aleph_{\dR, \upi}(v_+ \otimes v_-).
   \]
   This is independent of $v_{-\eps}$ and depends linearly on $v_{\eps}$, justifying the notation. Then the condition that $\langle \LPR({}_c z_m^{[\upi, r]}), \xi\rangle = \cL_m$ is satisfied; so we can ``run the machine'' and obtain a family of classes ${}_c \bz_m^{[\pi, r]}$, and a non-zero vector $\bar{\xi} \in \Dcris$, such that
   \[ \langle {}_c \bz_m^{[\pi, r]}, \bar\xi\rangle = \cL_m \bmod X = {}_c C_m^{[r]} \cdot \langle \cL^{\eps}_p(\pi), v_{\eps}\rangle.\]
   By construction, as bases of $\Dcris$ modulo $X$ we have
   \[ \aleph'_{\dR, \pi}(v_+ \otimes v_-) = \langle \Lambda^{[r]}, v_{-\eps}\rangle \cdot \bar{\xi}, \]
   where $\Lambda^{[r]} \in M_B^{-\eps}(\pi, L)^{\vee}$ is the leading term at $X = 0$ of $\cL_p^{-\eps}(\bt_2 + r)$. By definition this is non-zero. So we have shown that
   \[ \langle {}_c \bz_m^{[\pi, r]}, \aleph'_{\dR, \pi}(v_+ \otimes v_-)\rangle = {}_c C_m^{[r]} \cdot \langle \cL^{\eps}_p(\pi), v_{\eps}\rangle \cdot \langle \Lambda^{[r]}, v_{-\eps}\rangle.\]

   We define ${}_c \bz^{[\eps, r]}_{m}: M_B^{-\eps}(\pi, L)^{\vee} \to H^1_{\Iw}$ to be the unique linear map sending $\Lambda^{[r]}$ to ${}_c \bz_m^{[\pi, r]}$. Then we have
   \[ \langle {}_c \bz_m^{[\eps, r]}(\gamma_{-\eps}), \aleph'_{\dR, \pi}(v_+ \otimes v_-)\rangle =
   {}_c C_m^{[r]} \cdot \langle \cL^{\eps}_p(\pi, \Delta_m), v_{\eps}\rangle \cdot \langle \gamma_{-\eps}, v_{-\eps}\rangle.
   \]
   This depends on $c$ and $r$ only through the factor
   \[ {}_c C_m^{[r]} =(c^2 - c^{(\mathbf{j}+1-t_2-r')}\psi(c) \sigma_c )(c^2 - c^{(\mathbf{j}+1-t_2 -r)}\sigma_c ) \in \Lambda_{\cO}(\Gamma)^{\eps} \otimes \cO[\Delta_n],\]
   where $r' = k_2 - 2 - r$.

   We have already assumed that $p \ge 7$, as part of hypothesis (BI). We take $c$ to be a generator of $(\ZZ / p\ZZ)^\times$ with $\psi(c) = 1$. Then, for any given choice of sign $\eps$, we can find two integers $r_1, r_2$ with $(-1)^{t_2 + r + 1} = \eps$ such that $r_2$ is not congruent mod $p-1$ to either $r_1$ or $r_1' = k_2-2-r_1$. Hence at least one of ${}_c C_1^{[r_1]}$ and ${}_c C_1^{[r_2]}$ is invertible in each of the summands of $\Lambda_{\cO}(\Gamma)^{\eps}$ corresponding to the components of $\cW^{\eps}$. Since $\cO[\Delta_n]$ is a local ring, ${}_c C_m^{[r_1]}$ and ${}_c C_m^{[r_2]}$ generate the unit ideal for every $m$. So we can find a linear combination $\bz^{\eps}_{m}$ of ${}_c \bz^{[\eps, r_1]}_{m}$ and ${}_c \bz^{[\eps, r_2]}_{m}$ such that
   \[ \langle \bz_m^{\eps}(\gamma_{-\eps}), \aleph'_{\dR, \pi}(v_+ \otimes v_-)\rangle =
     \langle \cL^{\eps}_p(\pi, \Delta_m), v_{\eps}\rangle \cdot \langle \gamma_{-\eps}, v_{-\eps}\rangle,\]
   which is the characterising property of $\bz_m^{\eps}$ in the theorem.
  \end{proof}

    A less canonical, but more concrete, statement would be to fix bases $v_+, v_-$, with dual bases $\gamma_{\pm}$, and define
   \[ \bz_m = \bz_m^+(\gamma_+) + \bz_m^-(\gamma_-) \in \Ht^1_{\Iw}\left(K_{\infty}[m], V(\pi)^*; \square^{(\fp_1)}\right).\]
   The choice of $v_+, v_-$ also determines a basis of the relevant $\Dcris$, and complex periods $\Omega_\infty^{\pm}$, and with respect to these bases, the regulator of $\bz_m$ is the $p$-adic $L$-function.

 \subsection{Integral lattices}\label{sec:intlattice}

  We now suppose the conditions of \S\ref{sect:integralL} are satisfied. Let $M_B^{\pm}(\pi, \cO)$ denote the integral structures obtained from Betti cohomology with $\cO$-coefficients, and let $M_B^{\pm}(\pi, \cO)^\vee$ be their duals. We fix $\cO$-bases $\gamma_{\pm}$ of $M_B^{\pm}(\pi, \cO)^\vee$, and define $\bz_m(\pi) = \bz_m^+(\gamma_+) + \bz_m^-(\gamma_-)$ as above. We let $T(\pi)^*$ be a lattice in $V(\pi)^*$ such that, for all $m$, the class $\bz_m(\pi)$ is in the image of the injection
  \[ H^1_{\Iw}(K_\infty[m], T(\pi)^*) \to H^1_{\Iw}(K_\infty[m], V(\pi)^*). \]

  Let $v_{\pm}$ be the dual bases of the $\gamma_{\pm}$, and let $v_{\mathrm{cris}}$ denote a basis of the integral lattice $\Dcris(K_{\fp_2}, \cF^+_{\fp_2} T(\pi))$.

  \begin{definition}
   The \emph{integral comparison error} (for $T$) is the unique $q \in \ZZ$ such that
   \[  \aleph_{\dR, \fp_2}'(v_+ \otimes v_-) \in \varpi^q \cO^\times \cdot v_{\mathrm{cris}}. \]
  \end{definition}

  We have
  \begin{equation}\label{eq:powerofvarpiinLp}
    \langle \cL_p^{\eps}(\pi, \Delta_m), v_\eps \rangle= \varpi^q \langle \LPR_{\Delta_m}(\bz_m(\pi)), v_{\mathrm{cris}}\rangle.
  \end{equation}
  Thus $\langle \cL_p^{\eps}(\pi, \Delta_m), v_\eps \rangle \in \varpi^q \Lambda_{\cO}(\Gamma \times \Delta)$ for all $m$.

  \begin{proposition}\label{prop:nuexistence}
   For all $u \in \ZZ / (p-1) \ZZ$, we have $q \le \mu_{\mathrm{min}, u}$, where $\mu_{\mathrm{min}, u}$ is the minimal $\mu$-invariant of \cref{def:muinv}.

   If we set $\nu_u = \mu_{\mathrm{min}, u} - q$, then the image of $\bz_m$ in $e_u \cdot  H^1_{\Iw}(K_\infty[m], T(\pi)^* / \cF_{\fp_2}^+)$ is divisible by $\varpi^{\nu_u}$, for every $m$.
  \end{proposition}

  \begin{remark}
   It would seem sensible to choose $T(\pi)^*$ to be ``optimal'',  in the sense that some $\bz_m$ has non-zero image in $T(\pi) ^*/ \varpi$. However, our conclusions will be independent of the choice of $T$ anyway.
  \end{remark}


\section{Upper bounds for Selmer groups}
 \label{sect:mainthms}

 \subsection{Recap of hypotheses}

  For the reader's convenience we recall the hypotheses on $\pi$. Recall that $\pi$ is a Hilbert cuspidal automorphic representation of $\GL_2/K$ of level $\fn$ and weight $(k_1,k_2, t_1,t_2)$ with $k_i \ge 2$ and $k_1+2t_1=k_2+2t_2$; $L$ is a finite extension of $\Qp$, for some prime $p$; and we have fixed an embedding of the coefficient field of $\pi$ into $L$.

  For all of the results of this section, we suppose that:
  \begin{itemize}
   \item $(p, \fn)=1$ and $p$ splits in $K$.
   \item $\pi$ is ordinary at both primes $\fp_1,\, \fp_2$ above $p$ (with respect to our choice of embedding).
   \item The Galois representation of $\pi$ satisfies the condition (BI) of \cref{def:BI}.
   \item $\pi^\sigma \ne \pi$, and the central character of $\pi$ factors as $\psi \circ \Nm$, where $\psi$ is a Dirichlet character.
  \end{itemize}

  We fix $\cO$-bases $v_{\pm}$ of the Betti cohomology spaces $M_B^\pm(\pi,\cO)$, thus defining a $p$-adic $L$-function
  \[ L_p(\pi) = \langle \cL_p^+(\pi), v_+ \rangle + \langle \cL_p^-(\pi), v_- \rangle \in \Lambda_{\cO}(\Gamma).\]

 \subsection{The Iwasawa main conjecture}

  Our first result is a form of ``main conjecture without zeta-functions'' in the style of \cite[Chapter III]{kato04}. Recall the Selmer complexes $\RGt_{\Iw}(K_\infty, T(\pi)^*; \square^{(\fp_1)})$ and $\RGt_{\Iw}(K_\infty, T(\pi)^*; \square^{(p)})$ defined in \cref{sect:selmer}.

  \begin{theorem}
   \label{thm:withoutzeta}
   Let $u \in \ZZ / (p-1)\ZZ$ and suppose that $e_u L_p(\pi) \ne 0$. Then
   \begin{enumerate}
    \item $e_u \Ht^1_{\Iw}(K_\infty, T(\pi)^*; \square^{(\fp_1)})$ is free of rank $1$ over $e_u \Lambda$;
    \item $e_u \Ht^2_{\Iw}(K_\infty,T(\pi)^*; \square^{(\fp_1)})$ is $\Lambda$-torsion;
    \item we have
    \[ \varpi^{\nu_u} \mathrm{char}_{e_u \Lambda}\left( e_u \Ht^2_{\Iw}(K_\infty, T(\pi)^*; \square^{(\fp_1)})\right)\ \Bigm|\  \mathrm{char}_{e_u \Lambda}\left (e_u \cdot\frac{\Ht^1_{\Iw}(K_\infty, T(\pi)^*; \square^{(\fp_1)})}{\Lambda \cdot z_1(\pi)}\right), \]
    where $\nu_u \ge 0$ is the integer defined in \cref{prop:nuexistence}.
   \end{enumerate}
  \end{theorem}

  \begin{proof}
   This is exactly Theorem \ref{thm:selmerbound} applied to the induced representation $T = \Ind_{K}^{\QQ} T(\pi)^*$. By \cref{prop:nuexistence}, for each $u$ we can take the $\fQ$ of the theorem to be the ideal $(1 - e_u, \varpi)$, and $k = \nu_u$.
  \end{proof}

  We can now prove the main theorem of this section:

  \begin{theorem}[c.f. \cref{conj:mainconj}]
   \label{thm:withzeta}
  	For each $u \in \ZZ / (p-1)\ZZ$ we have
   \[  \varpi^{\mu_{\min, u}}\, \mathrm{char}_{e_u \Lambda}\left( e_u \Ht^2_{\Iw}(K_\infty, T(\pi)^*;\square^{(p)}_\Lambda)\right) \Bigm|  e_u L_p(\pi), \]
   where $\mu_{\min, u} \in \ZZ_{\ge 0}$ is as in Definition \ref{def:muinv}.
  \end{theorem}

  \begin{remark}
  Note that $\mu_{\mathrm{min}, u} \ge 0$ by definition, so this implies in particular that one divisibility holds in the Main Conjecture (\cref{conj:mainconj}). Moreover, if the full Main Conjecture holds, then we must have $\mu_{\min, u} = 0$.
  \end{remark}

  \begin{proof}
  	This is an instance of \cref{cor:rank0bound}. We take the map $\lambda$ to be the ``Coleman map'' defined by
   \[ \lambda(z) = \left\langle \LPR_{V^{\sharp}}(z), v_{\mathrm{cris}} \right\rangle \]
   where $v_{\mathrm{cris}}$ is a basis of the integral lattice $\Dcris(T^{\sharp})^\vee = \Dcris\left(K_{\fp_2}, \cF_{\fp_2} T(\pi)\right)$ defined using Wach modules, as in \cref{sec:intlattice}. We then have $
   \lambda(\bz_1(\pi)) = \varpi^{-q} L_p(\pi)$ by \eqref{eq:powerofvarpiinLp}. This factor of $\varpi^{-q}$ cancels out against the factor $\varpi^{\nu_u} = \varpi^{\mu_{\mathrm{min}, u} - q}$ on the left-hand side of the previous theorem, and hence we obtain the formula stated.
  \end{proof}

 \subsection{The Bloch--Kato conjecture}

  \begin{theorem}
   \label{thm:BK}
   Let $j$ be a Deligne-critical integer and $\chi$ a finite-order character of $\Zp^\times$. Then we have
   \[ \dim H^1_{\mathrm{f}}(\QQ, V(\pi)^*(-j-\chi)) \le \ord_{j + \chi} L_p(\pi) \]
   where $\ord_{\chi}$ denotes the order of vanishing at the point $j + \chi \in \cW$. In particular, if $L(\pi \otimes \chi^{-1}, 1 + j) \ne 0$, then the Bloch--Kato Selmer group $H^1_{\mathrm{f}}(K, V(\pi)^*(-j-\chi))$ is zero.
  \end{theorem}

  \begin{proof}
   This follows from \cref{thm:withzeta} using the base-change properties of Selmer complexes, since the Bloch--Kato local condition for twists in the critical range coincides with the Greenberg local condition $\square^{(p)}$.
  \end{proof}

  Note that the special case when $\uk = (2, 2)$, $j = 0$, and $\chi$ is trivial has been previously established by \Nek \cite{nekovar06}, using the theory of Heegner points (and this argument, unlike ours, extends to Hilbert modular forms over general totally real fields). However, the cases when $\chi$ is not trivial (or quadratic) or $j$ is not the central value are not accessible via Heegner theory, since the resulting motive is not self-dual.

 \subsection{Elliptic curves}
  \label{sect:ellcurves}

  Let $A / K$ be an elliptic curve. By \cite{freitaslehungsiksek15}, $A$ is modular. We shall suppose the following:
  \begin{itemize}
  \item $A$ does not have complex multiplication.
  \item $A$ is not isogenous over $\Qb$ to its conjugate $A^{\sigma}$ (i.e.~$A$ is not a $\QQ$-curve).
  \end{itemize}

  \begin{theorem}\label{thm:BSD}
   The equivariant Birch--Swinnerton-Dyer conjecture holds in analytic rank 0 for twists of $E$ by Dirichlet characters. That is, for any $\eta: (\ZZ / N\ZZ)^\times \to \Qb^\times$, we have the implication
   \[ L(A, \eta, 1) \ne 0\quad \Longrightarrow\quad A(K(\mu_N))^{(\eta)} \text{\ is finite}.\]
   Moreover, if $L(A, \eta, 1) \ne 0$ then, for a set of primes $p$ of density $\ge \tfrac{1}{2}$, the group $\Sha_{p^\infty}(A / K(\zeta_N))^{(\eta)}$ is trivial.
  \end{theorem}

  \begin{proof}
   Let $\pi = \pi_A \otimes \eta$, where $\pi_A$ is the automorphic representation associated to $A$. Then the hypotheses of Theorem \ref{thm:withoutzeta} are satisfied for all but finitely many primes $p$ split in $K$ such that $A$ is ordinary above $p$. Since the non-ordinary primes have density 0, and the primes split in $K$ have density $\tfrac{1}{2}$, the result follows.
  \end{proof}

  \begin{remark}
   Note that if $A$ is a $\QQ$-curve over $\cK$ and the isogeny $A \to A^{\sigma}$ is also defined over $\cK$, then the automorphic representation $\pi_A$ is a base-change from $\QQ$, and the statement of \cref{thm:BSD} for $A$ follows from Kato's Euler system for the corresponding elliptic modular form: see \cite[Theorem 14.2]{kato04}. $\QQ$-curves satisfying this property are said to be \emph{completely defined over $\cK$}. However, if $A$ is a $\QQ$-curve over $\cK$ which is not completely defined over $\cK$, then neither Theorem \ref{thm:BSD} nor Kato's results apply to $A$ (although it is not impossible that the methods of \cref{thm:BSD} might extend to cover this case).
  \end{remark}

\section{Lower bounds for Selmer groups}
 \label{sect:wan}
 We now compare the upper bound for the Selmer group over $K_\infty$ given by Theorem \ref{thm:withzeta} with the opposite bounds obtained by Wan in \cite{wan15b}.

 \begin{note}
  The arguments of this section are very closely parallel to the proof of Theorem 3.29 of \cite{skinnerurban14} in the $\GL_2 / \QQ$ case. In our setting, \cref{thm:withzeta} takes the place of the results of \cite{kato04} used in \emph{op.cit.}, and Theorem 101 of \cite{wan15b} playing the role of Theorem 3.26 of \cite{skinnerurban14}.
 \end{note}

 \subsection{Hypotheses}
  \label{sect:addhyp1}

  We continue to assume the hypotheses on $K$, $\pi$, and $p$ listed at the beginning of \S \ref{sect:mainthms}, and the following additional hypotheses:
  \begin{enumerate}[(i)]

   \item $\pi$ has parallel weight $k \ge 2$ (i.e.~$k_1 = k_2 = k$).

   \item The central character $\varepsilon_{\pi}$ is trivial (implying that $k$ is even).


  \end{enumerate}


  Without loss of generality we shall take $t_1 = t_2 = 1 - \tfrac{k}{2}$, so that the value of our $p$-adic $L$-function at the trivial character corresponds to the central value of the complex $L$-function. Since $\varepsilon_{\pi}$ is trivial, the $L$-function associated to $\pi$ satisfies a functional equation, with some global root number $w_{\pi} \in \{ \pm 1\}$.

 \subsection{Minimality and integral period relations}

  We will also need to consider the following condition: we say $\rho_{\pi, v}$ is \emph{$\fn$-minimal} if the Artin conductor of the residual representation $\bar{\rho}_{\pi, v}$ is equal to $\fn$ (which is the Artin conductor of the characteristic 0 representation $\rho_{\pi, v}$).

  \begin{remark}
   If $\fn$ is square-free, then this amounts to requiring that $\bar{\rho}_{\pi, v}(I_{\fq})$ is non-trivial for all $\fq \mid \fn$.
  \end{remark}

  \begin{theorem}[Dimitrov, Fujiwara, Tilouine--Urban]
   \label{thm:gorenstein}
   Assume the hypotheses of \S \ref{sect:mainthms} and \S \ref{sect:addhyp1} hold, and in addition that $\rho_{\pi, v}$ is $\fn$-minimal. Then:
   \begin{enumerate}[(i)]
    \item For any sign $\eps$, the localisation of the integral Betti cohomology
    \[ H^i_{B, c}\left(Y_{K, 1}(\fn), \cV_k(\cO)\right)^{\eps}_{\mathfrak{m}}\]
    is zero for $i \ne 2$, and free of rank 1 over $T_{\mathfrak{m}}$ for $i = 2$, where $\mathfrak{m}$ is the maximal ideal corresponding to $\bar{\rho}_{\pi, v}$, and $T_{\mathfrak{m}}$ the corresponding localised Hecke algebra over $\cO$.

    \item The ring $T_{\mathfrak{m}}$ is a complete intersection ring, and hence a Gorenstein ring.

    \item We have
    \[ \frac{L(\operatorname{Ad} \pi, 1) \Gamma(\operatorname{Ad}\pi, 1)}{\eta_{\pi} \cdot \Omega^+_{\pi} \Omega^-_{\pi}} \in \cO^\times, \]
    where $\Omega^{\pm}_{\pi}$ are integrally normalised periods as in \cref{sect:prelimhilb}, and $\eta_{\pi} \in \cO$ denotes a generator of the congruence ideal of $\pi$.
   \end{enumerate}
  \end{theorem}

  \begin{proof}
   This is Theorem 4.11 and Proposition 4.12 of \cite{tilouineurban18}, refining earlier results of Dimitrov and of Fujiwara (quoted as Theorem 8 in \cite{wan15b}).
  \end{proof}

 \subsection{P-adic $L$-functions and Selmer groups over $\cK$}

  We now define the objects needed to state the main result of \cite{wan15b} (in a simple case).

  \subsubsection{The CM field \texorpdfstring{$\cK$}{cK}}

  \begin{definition}
   Let $\cK$ denote a totally imaginary quadratic extension of $K$ satisfying the following conditions:
   \begin{itemize}
    \item $\cK$ is not contained in the narrow class field of $K$ (i.e.~some finite prime of $K$ ramifies in $\cK$).
    \item All primes dividing $p \fn \disc$ are split in $\cK / K$.
   \end{itemize}
   Let $\varepsilon_{\cK}$ denote the quadratic Hecke character of $K$ corresponding to $\cK / K$.
  \end{definition}

  As in \emph{op.cit.} we let $\cK_\infty$ denote the unique $\Zp^3$-extension of $\cK$ unramified outside $p$, and $\Gamma_{\cK}$ its Galois group. We have $\cK_\infty \supset \cK \cdot K_{\infty}^1$, where $K_{\infty}^1 \subset K_\infty$ is the cyclotomic $\Zp$-extension of $K$, with Galois group $\Gamma^1 \cong \Zp$. This gives us a projection map
  \[ \Lambda_{\cO}(\Gamma_\cK) \onto \Lambda_{\cO}(\Gamma^1) = e_0 \Lambda, \]
  where $e_0$ is the idempotent in $\Lambda = \Lambda_{\cO}(\Gamma)$ corresponding to the trivial action of $(\ZZ / p\ZZ)^\times$. We take the set $\Sigma$ of primes in \emph{op.cit.} to be precisely the primes above $p$.

  \subsubsection{P-adic L-functions}
  Let $f$ be the newform generating $\pi$. In \S 7.3 of \emph{op.cit.}, two $p$-adic $L$-functions are defined,
  \begin{align*}
   \tilde{L}^{\Sigma}_{f, \cK, 1} &\in \Lambda_L(\Gamma_{\cK}) &
   \text{and}&&
   L^{\Sigma}_{f, \cK, 1} &\in \Lambda_{\cO}(\Gamma_{\cK}),
  \end{align*}
  with the latter only defined if the Gorenstein property of \cref{thm:gorenstein}(ii) holds. When both are defined, the relation between the two is given by $L^{\Sigma}_{f, \cK, 1} = \eta_{\pi}\cdot \tilde{L}^{\Sigma}_{f, \cK, 1}$ where $\eta_{\pi}$ is a generator of the congruence ideal.

  \begin{proposition}
   If the hypotheses of \cref{thm:gorenstein} hold, then the image of $L^{\Sigma}_{f, \cK, 1}$ in $\Lambda_{\cO}(\Gamma^1)$ is given by
   \[ b_{\pi} \cdot ( e_0 L_p(\pi) ) \cdot (e_0 L_p(\pi \otimes \varepsilon_{\cK})), \]
   where $b_{\pi} \in \Lambda_{\cO}(\Gamma^1)^\times$. If we do not assume the hypotheses of \cref{thm:gorenstein}, then we have the same result for $\tilde{L}^{\Sigma}_{f, \cK, 1}$ up to a unit in $\Lambda_{L}(\Gamma^1)^\times$.
  \end{proposition}

  \begin{proof}
   This follows by comparing the interpolation formulae of \S\ref{sect:gl2kpadicL} with Theorem 82 of \cite{wan15b}, and applying the functional equation to relate values at $s = k_0 - 1$ with values at $s = 1$. Theorem \ref{thm:gorenstein}(iii) shows that the ``canonical period'' $\Omega_{f, \mathrm{can}}$ of \cite[Definition 79]{wan15b} is given by $\Omega_{\pi}^+ \Omega_{\pi}^-$ up to a unit in $\cO$. Without the Gorenstein property we cannot compare the periods integrally, but we nonetheless obtain an equality after inverting $p$.
  \end{proof}

  \subsubsection{Selmer groups}

   We refer to \S 2.5 of \emph{op.cit.} for the definition of a module $X_{\cK_\infty^+}^{\Sigma}(T(\pi))$ over $\Lambda_{\cO}(\Gamma^1)$, which is the Pontryagin dual of a Selmer group for $T(\pi)(1) \otimes \Qp/\Zp$ over $\cK_\infty^+ = \cK K_\infty^1$. We want to compare this with the Selmer complexes used above.

   \begin{notation}
   In this section when we use Selmer complexes $\RGt_{\Iw}(-)$, the local condition will always be $\square^{(p)}$, and we omit it from the notation.
   \end{notation}

   \begin{proposition}
    The dual Selmer group $X_{\cK_\infty^+}^{\Sigma}(T(\pi))$ has the same characteristic ideal as the direct sum
    \[e_0 \Ht^2_{\Iw}(K_\infty, T(\pi)^*)\oplus  e_0 \Ht^2_{\Iw}(K_\infty, T(\pi \otimes \varepsilon_{\cK})^*). \]
   \end{proposition}

   \begin{proof}
    Since $\Ind_{\cK}^{K}( T(\pi)^*|_{G_{\cK}}) = T(\pi)^* \oplus T(\pi \otimes \varepsilon_{\cK})^*$, we have
    \[ \RGt_{\Iw}(\cK_\infty^+, T(\pi)^*) = e_0 \RGt_{\Iw}(K_\infty, T(\pi)^*) \oplus \RGt_{\Iw}(K_\infty, T(\pi \otimes \varepsilon_{\cK})^*).\]
    However, a standard argument using Poitou--Tate duality for Selmer complexes shows that there is an exact sequence
    \[ 0 \to X_{\cK_\infty^+}^{\Sigma}(T(\pi)) \to \Ht^2_{\Iw}(\cK_\infty^+, T(\pi)^*) \to (\text{local $H^2$ terms})\]
    where the local $H^2$ terms are finite groups, and hence pseudo-null.
   \end{proof}

  \subsubsection{Wan's theorem}

   \begin{theorem}[Wan]
    Assume the hypotheses of \S \ref{sect:mainthms} and \S \ref{sect:addhyp1} hold. Then we have
    \[ e_0 \cdot L_p(\pi) \cdot L_p(\pi \otimes \varepsilon_{\cK})\ \Big|\ \Char_{e_0 \Lambda}\left(X_{\cK_\infty^+}^{\Sigma}(T(\pi))\right) \]
    as ideals of $\Lambda[1/p]$. If in addition $\pi$ satisfies the hypotheses of Theorem \ref{thm:gorenstein}, then this divisibility holds integrally, i.e.~as ideals of $\Lambda = \cO[[\Gamma]]$.
   \end{theorem}

   \begin{proof}
    As in \cite[Corollary 3.28]{skinnerurban14}, it suffices to prove the statement after replacing $\Sigma$ with a larger set of primes $\Sigma'$ (replacing the $p$-adic $L$-functions with $\Sigma'$-depleted versions). The results of  \cite[\S 2.5]{wan15b} show that if $\Sigma'$ is sufficiently large, then the characteristic ideal of $X_{\cK_\infty^+}^{\Sigma'}(T(\pi))$ is the image in $\Lambda(\Gamma^1)$ of the corresponding characteristic ideal over the larger extension $\cK_\infty$. The result now follows from Theorem 101 of \cite{wan15b}.
   \end{proof}

   \begin{remark}
    Note that if $k = 2$ we cannot rule out the possibility that $e_0 L_p(\pi) L_p(\pi \otimes \varepsilon_{\cK})$ might vanish identically (although we do not expect this to happen). If this occurs, then the theorem shows that $X_{\cK_\infty^+}^{\Sigma}(T(\pi))$ has positive rank.
   \end{remark}

 \subsection{The Main Conjecture}

  \begin{theorem}
  \label{thm:mainconj}
   Assume the hypotheses of \S \ref{sect:mainthms} and \S \ref{sect:addhyp1} hold. If $k = 2$, suppose also that $w_\pi = +1$. Then we have
   \[ \Char\left(e_0 \Ht^2_{\Iw}(K_\infty, T(\pi)^*)\right)= \Big(e_0 L_p(\pi)\Big)\]
   as ideals of $e_0 \Lambda_L(\Gamma)$. If in addition $\pi$ satisfies the hypotheses of Theorem \ref{thm:gorenstein}, then this holds as an equality of ideals of $e_0 \Lambda_{\cO}(\Gamma)$.
  \end{theorem}

  \begin{proof}
   Let us write $L_p^{\alg}(\pi)$ for a characteristic element of the $\Ht^2_{\Iw}$ (well-defined up to units). Applying Theorem \ref{thm:withzeta} to both $\pi$ and $\pi \otimes \varepsilon_{\cK}$, we have
   \[ e_0 L_p^{\alg}(\pi) \mid e_0 L_p(\pi) \qquad\text{and}\qquad e_0 L_p^{\alg}(\pi \otimes \varepsilon_{\cK}) \mid e_0 L_p(\pi \otimes \varepsilon_{\cK}).\]
   However, we also have
   \[e_0  L_p(\pi) L_p(\pi \otimes \varepsilon_{\cK}) \mid e_0  L_p^{\alg}(\pi)L_p^{\alg}(\pi \otimes \varepsilon_{\cK})\]
   by Wan's theorem. (Here these are to be understood as divisibilities in the Iwasawa algebra over $\cO$ if the hypotheses of Theorem \ref{thm:gorenstein} hold, and otherwise after inverting $p$.)

   Taken together, these formulae imply that one of the following must hold:
   \begin{itemize}
   \item $e_0 L_p^{\alg}(\pi) =  e_0 L_p(\pi)$ up to units in $e_0 \Lambda$;
   \item $e_0 L_p(\pi \otimes \varepsilon_{\cK}) = 0$.
   \end{itemize}

   The latter cannot occur if $k > 2$ by Proposition \ref{prop:nonvanish-padic}. If $k = 2$ and the global root number of $L(\pi, s)$ is +1, then by Waldspurger's theorem \cite[Theorem 4]{waldspurger91} we may suppose that $\cK$ is chosen such that $L(\pi \otimes \varepsilon_{\cK}, 1) \ne 0$, which also implies $e_0 L_p(\pi \otimes \varepsilon_{\cK}) \ne 0$.
  \end{proof}

\section{Applications to non-vanishing}
 \label{sect:nonvanish}
 We note here a by-product of our method, which is the following ``horizontal'' non-vanishing theorem for central $L$-values. Let $K$, $\pi$ and $p$ be as in \S\ref{sect:mainthms}.

 \subsection{Non-vanishing in characteristic 0}

 \begin{proposition}
   Suppose $w$ is even, and let $\chi$ be a finite-order character of $\Zp^\times$. Suppose equality holds in Theorem \ref{thm:BK} for $\chi$ and $j = \tfrac{w-2}{2}$, i.e. for this $j$ we have
    \[ \ord_{j + \chi} L_p(\pi) = \dim H^1_{\mathrm{f}}(\QQ, V(\pi)^*(-j-\chi)).\]
    \item There exists an $m \in \cR$ and $\eta: \Delta_m \to \Qb^\times$ such that $L(\pi \otimes \eta^{-1}\chi^{-1}, \tfrac{w}{2}) \ne 0$.
  \end{proposition}

  \begin{proof}
   If $L(\pi \otimes \eta^{-1}\chi^{-1}, 1 + j) = 0$ for all $m$ and $\eta$, then we can apply \cref{thm:selmerbound} with $\fQ$ the ideal corresponding to $j + \chi$, and some $k \ge 1$. This gives a strict inequality in Theorem \ref{thm:BK}, which contradicts the hypothesis.
  \end{proof}

  In particular, if the stronger hypotheses of \S \ref{sect:addhyp1} are satisfied, and $e_0 L_p(\pi)$ is not identically 0 (which is automatic if $k > 2$), then the hypothesis of the theorem follows from \cref{thm:mainconj}, so the non-vanishing result follows for every $\chi$ factoring through $1 + p\Zp$.

  \begin{remark} \
   \begin{enumerate}
   \item The assumption that $e_0 L_p(\pi) \ne 0$ is somewhat inconvenient: we are, essentially, assuming one kind of non-vanishing statement in order to deduce another. It may be possible to remove this assumption by considering leading terms along a Hida family as in \cref{sect:axiom}. We hope to return to this in a future paper.
   \item It is already known, by results of Rohrlich \cite{rohrlich89}, that one can always find a twist of $\pi$ by a Hecke character of $K$ such that the central $L$-value is non-zero. However, our result gives non-vanishing for some twist in the much sparser set of characters factoring through the norm map to $\QQ$.
   \end{enumerate}
  \end{remark}

  \subsection{Applications to $\mu$-invariants} Let $u \in \ZZ / (p-1) \ZZ$, and suppose $e_u L_p(\pi) \ne 0$.

  \begin{proposition}
   Suppose that the Iwasawa $\mu$-invariants of $e_u \Ht^2_{\Iw}(K_\infty, T(\pi)^*; \square^{(p)})$ and of $e_u L_p(\pi)$ are equal. Then we have $\mu_{\min, u} = 0$; that is, there exists some $m \in \cR$ such that $e_u L_p(\pi, \Delta_m)$ is non-zero modulo $\varpi$.
  \end{proposition}

  \begin{proof}
   This follows by the same argument as above with $\fQ$ taken to be the special fibre of $e_u \Lambda$.
  \end{proof}

  As before, if we impose the stronger hypotheses of \S \ref{sect:addhyp1} and also of Theorem \ref{thm:gorenstein}, and $u = 0$, then the equality of $\mu$-invariants is part of Theorem \ref{thm:mainconj}.

  \begin{remark}
   Note that it does \emph{not} follow from the theorem that $L_p(\pi \otimes \eta)$ is non-zero modulo $\varpi$ for some character $\eta$ of $\Delta_m$, since $\cO[\Delta_m]$ is very far from being the direct sum of its character components.
  \end{remark}

  \section{A numerical example}

   We illustrate our results with the first example of an elliptic curve over a real quadratic field in the LMFDB database \cite{lmfdb}: the elliptic curve $A$ over $K = \QQ(\sqrt{5})$ given by
   $ y^2 + x y + \phi y = x^{3} + \left(\phi + 1\right) x^{2} + \phi x$, where $\phi = \tfrac{1+ \sqrt{5}}{2}$. The LMFDB entry for this curve, at \url{https://www.lmfdb.org/EllipticCurve/2.2.5.1/31.1/a/1}, shows that its conductor is the prime $\fq  =(5\phi - 2)$ of norm 31. So for any prime $p\ne 31$ congruent to $\pm 1 \bmod 5$, $p$ will split in $K$ and $A$ will have good reduction at the primes above $p$. Moreover, $A$ has rank 0, and
   \[ L(A, 1) \sim 0.3599289594\dots.\]

   The database also confirms that $A$ is not a $\QQ$-curve, and that $\rho_{A, p}$ is surjective onto $\GL_2(\Zp)$ for all $p \ne 2$. Since the only nontrivial proper normal subgroup of $\SL_2(\mathbf{F}_p)$ for $p > 3$ is $\{\pm 1\}$, if $p  > 3$ then we must either have
   \[ (\rho_{A, p} \times \rho_{A^{\sigma}, p})(G_K) = \{ (x, y) \in \GL_2(\Zp): \det(x) = \det(y)\}, \]
   or
    \[ (\rho_{A, p} \times \rho_{A^{\sigma}, p})(G_K) \subseteq \{ (x, y) \in \GL_2(\Zp): x = \pm y \bmod p\}. \]
   If $p \ne 59$ and $\gamma$ is the Frobenius at the prime $\fq = 7\phi - 2$ of norm 59, then we have $\operatorname{tr} \rho_{A, p}(\gamma) = a_{\fq}(A) = 12$, but $\operatorname{tr} \rho_{A^{\sigma}, p}(\gamma) = a_{\sigma(\fq)}(A) = -4$. Since the only prime dividing $12 \pm (-4)$ is 2, we conclude that $\rho_{A, p} \times \rho_{A^{\sigma}, p}$ is surjective for all primes $p > 2$ (except possibly $p = 59$, but we can eliminate this possibility by performing a similar computation with the primes above 61). So $\rho_{A, p} \times \rho_{A^{\sigma}, p}$ has full image for all primes $p > 2$.

   To verify the additional ``minimality'' hypothesis of Theorem \ref{thm:gorenstein}, we need to check that $\bar{\rho}_{A, p}$ is ramified at the prime $\fq$. The primes where this fails are those which divide the valuation of the $j$-invariant at $\fq$; since the $j$-invariant has valuation $-1$, there are no such primes, and hence this condition is automatically satisfied for all $p$.

   Hence the Iwasawa main conjecture for $A$ over the cyclotomic $\Zp$-extension of $K$ holds for all primes $p \ge 7$ that split in $K$ and such that $A$ has good ordinary reduction above $p$. This includes all primes $p < 200$ with $p = \pm 1 \bmod 5$, except 71, 79 and 191.

 \appendix

 \section{Beilinson--Flach elements revisited}
  \label{sect:appendix}

  \subsection{Setting} In this brief appendix we record a consequence of the ``leading term argument'' of \cref{sect:core,sect:axiom} above, which is a refinement of our earlier results with Kings \cite{KLZ17} on the arithmetic of Rankin--Selberg convolutions. For simplicity we state the results for Artin twists of elliptic curves, although the arguments apply more generally to Artin twists of any modular form of weight $\ge 2$.

  Let $A$ be a non-CM elliptic curve over $\QQ$, and $\rho$ a 2-dimensional odd irreducible Artin representation, with the conductors $N_A$ and $N_\rho$ coprime. Bertolini, Darmon and Rotger have shown in \cite{BDR-BeilinsonFlach2} that the equivariant BSD conjecture holds for $E$ and $\rho$ in analytic rank 0: that is, we have the implication
  \[ L(A, \rho^*, 1) \ne 0\quad \Longrightarrow\quad \Hom_{\Gal(F / \QQ)}\left( \rho, A(F) \otimes \Qb\right) = 0,\]
  where $F$ is the fixed field of the kernel of $\rho$. Our aim here is to give bounds for the $\rho$-part of the Tate--Shafarevich group.

  \subsection{The result} Let $A$ and $\rho$ be as above, and assume that $L(A, \rho^*, 1) \ne 0$.

   We let $F$ be the coefficient field of $\rho$; and we choose a rational prime $p$ and a prime $v \mid p$ of $F$, and let $L$ be the completion of $F$ at $v$, with ring of integers $\cO$. We suppose $v$ is chosen such that:
  \begin{itemize}
   \item $p \ge 5$;
   \item $p \nmid N_A N_\rho$;
   \item $A$ is ordinary at $p$;
   \item the $\cO$-linear Galois representation $T = T_p(A) \otimes \rho$ satisfies the big-image hypothesis 11.2.1 of \cite{KLZ17}.
  \end{itemize}

  These are the same hypotheses as in \cite[Theorem 11.7.4]{KLZ17} with one crucial difference: we are \emph{not} supposing here that the eigenvalues of $\rho(\operatorname{Frob}_p)$ be distinct mod $v$. So our present hypotheses apply for all but finitely many ordinary primes, while the the hypotheses of \emph{op.cit.} rule out a positive-density set of primes (those which split completely in the subfield of $F$ cut out by the projective representation $\operatorname{Ad}(\rho)$) for which $\rho(\operatorname{Frob}_p)$ is scalar.

  We define
  \[ \Sha_{v^\infty}(A, \rho) = \Hom_{\Gal(F / \QQ)}\left(\rho, \Sha_{p^\infty}(A/F) \otimes \cO\right).
  \]

  \begin{theorem}
   The group $\Sha_{v^\infty}(A, \rho)$ is finite. If $p$ does not divide $[F : \QQ]$, its order is bounded by
   \[
    \operatorname{length}_{\cO} \Sha_{v^\infty}(A, \rho) \le \operatorname{ord}_v\left(\frac{L(E \times \rho^*, 1)}{\Omega^+_A \Omega^-_A G(\det \rho^*)}\right).
   \]
   In particular, $\Sha_{v^\infty}(A, \rho)$ is trivial for all but finitely many ordinary primes $v$.
  \end{theorem}

  \begin{proof}
   Since we know that the $\rho$-part of the Mordell--Weil group is trivial, there is a natural map
   \[ H^1_{\mathrm{f}}(\QQ, A[p^\infty] \otimes \rho^*) = \Hom_{\Gal(F / \QQ)}(\rho, \operatorname{Sel}_{p^\infty}(A/F)\otimes \cO) \into \Sha_{v^\infty}(A, \rho). \]
   The cokernel of this map injects into $\operatorname{Ext}^1_{\Gal(F/\QQ)}\left(\rho, A(F) \otimes L/\cO\right)$; so it is finite, and trivial if $p$ does not divide $[F : \QQ]$. So it suffices to bound $H^1_{\mathrm{f}}(\QQ, A[p^\infty] \otimes \rho^*)$.

   Let $g$ be the weight 1 newform corresponding to $\rho^*$. This must have at least one $p$-stabilisation $g_\alpha$ (extending $L$ if necessary). By results of Hida and Wiles recalled in \cite[\S 7.6]{KLZ17}, $g_{\alpha}$ defines a point on the $\GL_2$ ordinary eigencurve $\mathcal{C}(N_\rho)$. As in \S 7.5 of op.cit., the space $\mathcal{C}(N_\rho)$ has finitely many irreducible components (``branches'') $\mathcal{C}_{\mathbf{a}}$, corresponding to the minimal primes $\mathbf{a}$ of the ordinary Hecke algebra. In general there can be several components passing through $g_{\alpha}$, and the irreducible components need not be smooth there; but we shall choose a branch $\mathbf{a}$ passing through $g_{\alpha}$, and a point $x \in \tilde{\mathcal{C}}_{\mathbf{a}}(L)$ lifting $g_\alpha$ (after replacing $L$ by a finite extension if necessary). The analytic generic fibre of $\tilde{\mathcal{C}}_{\mathbf{a}}$ is a smooth rigid-analytic curve, so we can choose a neighbourhood $\Omega \ni x$ which is an affinoid disc.

   Hida theory gives us a rank 2 $\cO(\Omega)$-module $\mathcal{M}(\mathbf{a})$ with a $G_{\QQ}$-action, interpolating the Galois representations of the specialisations of the Hida family $\mathfrak{a}$. Exactly as in \S \ref{sect:ESYfam}, we can factor out the torsion in order to obtain a family of free modules, and by comparing traces, the fibre of this family at $x$ is isomorphic to $\rho$. We let $\cM$ be the 4-dimensional family given by tensoring $\mathcal{M}(\mathbf{a})^*$ with $V_p(A) = T_p(A)[1/p]$. This gives a family of representations specialising to $V_p(A) \otimes \rho^*$, and it has subrepresentations
   \[ \cM^{++} = \mathcal{F}^+ V_p(A) \otimes \mathcal{M}(\mathbf{a})^*, \qquad \cM^+ = \cM^{++} + V_p(A) \otimes \mathcal{F}^+\mathcal{M}(\mathbf{a})^* \]
   of rank 2 and rank 3 respectively, where $\mathcal{F}^+\mathcal{M}(\mathbf{a})^*$ and $\mathcal{F}^+ V_p(A)$ are rank 1 $G_{\Qp}$-stable submodules with unramified quotient.

   We now apply the theory of \cref{sect:axiom} to $\mathcal{M}$, taking the Euler system $(c_m)$ to be the Beilinson--Flach elements associated to $f_A$ and the family through $g_\alpha$ as in \cite{KLZ17}. By construction, these classes land in $\mathcal{M}^+$ locally at $p$, and the reciprocity law for Beilinson--Flach elements (Theorem B of \emph{op.cit.}) shows that their images in the quotient $\mathcal{M}^+ / \mathcal{M}^{++}$ are governed by the $L$-values of Dirichlet-character twists of $A \times \rho$. Exactly as in Theorem \ref{thm:withzeta} above, we deduce that the dual Selmer group for $T_p(A) \otimes \rho^*$ over the cyclotomic extension is bounded above by the $p$-adic $L$-function $L_p(A \times \rho^*)$, where we can form our $p$-adic $L$-functions using any period which makes the equivariant $p$-adic $L$-functions $L_p(A \times \rho^*, \Delta_m)$ be $p$-adically integral for all $m$. It follows readily from Hida's construction of $p$-adic Rankin--Selberg $L$-functions that the period $\Omega^+_A \Omega^-_A G(\det \rho^*)$ has this property. By a standard descent computation, we deduce the above bound for the Selmer group over $\QQ$.
  \end{proof}

\newlength{\bibitemsep}\setlength{\bibitemsep}{.2\baselineskip plus .05\baselineskip minus .05\baselineskip}
\newlength{\bibparskip}\setlength{\bibparskip}{0pt}
\let\oldthebibliography\thebibliography
\renewcommand\thebibliography[1]{%
  \oldthebibliography{#1}%
  \setlength{\parskip}{\bibitemsep}%
  \setlength{\itemsep}{\bibparskip}%
}
\newcommand{\noopsort}[1]{\relax}

\providecommand{\bysame}{\leavevmode\hbox to3em{\hrulefill}\thinspace}
\providecommand{\MR}[1]{%
 MR \href{http://www.ams.org/mathscinet-getitem?mr=#1}{#1}.
}
\providecommand{\href}[2]{#2}
\newcommand{\articlehref}[2]{\href{#1}{#2}}

\end{document}